\documentclass[a4paper,reqno,10pt,oneside]{amsart}
\usepackage{amsfonts,amssymb,amsmath,amsthm,amsfonts}
\usepackage{geometry} 
\usepackage[T1]{fontenc}
\usepackage[english]{babel} 
\usepackage{fix-cm}
\usepackage{nicefrac} 
\usepackage[colorlinks]{hyperref}
\usepackage{graphicx}
\usepackage{braket}
\usepackage{graphicx}
\usepackage[usenames,dvipsnames]{xcolor}
\usepackage{cancel}
\usepackage{filecontents}
\usepackage{latexsym}
\usepackage{bm}
\usepackage[sc]{mathpazo}

\newtheorem{thm}{Theorem}[section]
\newtheorem{lemma}{Lemma}[section]
\newtheorem{prop}{Proposition}[section]

\newtheorem{rmk}{Remark}[section]{}
\theoremstyle{definition}
\numberwithin{equation}{section}
\newcommand{\rr}{\mathbb R}
\newcommand{\al}{\alpha}
\newcommand{\de}{\delta}
\newcommand{\eps}{\epsilon}
\newcommand{\veps}{\varepsilon}

\newcommand{\ga}{\gamma}
\newcommand{\si}{\sigma}

\newcommand{\Om}{\Omega}
\newcommand{\om}{\omega}
\newcommand{\pl}{\partial}
\def\beq{\begin{equation}}
\def\beqs{\begin{equation*}}
\def\eeq{\end{equation}}
\def\eeqs{\end{equation*}}
\def\bal{\begin{aligned}}
\def\eal{\end{aligned}}
\def\baln{\begin{align}}
\def\ealn{\end{align}}
\def\bca{\begin{cases}}
\def\eca{\end{cases}}
\DeclareMathOperator{\diam}{diam}
\def\sideremark#1{\ifvmode\leavevmode\fi\vadjust{\vbox to0pt{\vss% the remark
 \hbox to 0pt{\hskip\hsize\hskip1em%                          will appear only
 \vbox{\hsize3cm\tiny\raggedright\pretolerance10000%          on the side
  \noindent #1\hfill}\hss}\vbox to8pt{\vfil}\vss}}}%

 \renewcommand{\(}{\left(}
\renewcommand{\)}{\right)}
\newcommand{\G}{\mathbf{G}}
\newcommand{\up}{u_p}
\newcommand{\Up}{\mathcal U_p}
\newcommand{\Rp}{\mathcal R_p}

\newcommand{\Udai}{U_{\alpha_i,\delta_i}}
\newcommand{\Udaj}{U_{\alpha_j,\delta_j}}

\newcommand{\Ud}{U_{\delta}}
\newcommand{\tvi}{v_{i}^{({w})}}
\newcommand{\tvj}{v_{j}^{({w})}}
\newcommand{\tUdai}{U_{i}^{({w})}}
\newcommand{\tUdaj}{U_{j}^{({w})}}

\newcommand{\taj}{\alpha_j^{({w})}}
\newcommand{\tai}{\alpha_i^{({w})}}
\newcommand{\ta}{\alpha^{({w})}}

\newcommand{\woai}{w_{\al_i}^0}
\newcommand{\wiai}{w_{\al_i}^1}
\newcommand{\woaj}{w_{\al_j}^0}
\newcommand{\wiaj}{w_{\al_j}^1}
\newcommand{\wlai}{w_{\al_i}^\ell}
\newcommand{\wlaj}{w_{\al_j}^\ell}

\newcommand{\Coai}{C_{\al_i}^0}
\newcommand{\Ciai}{C_{\al_i}^1}
\newcommand{\Coaj}{C_{\al_j}^0}
\newcommand{\Ciaj}{C_{\al_j}^1}
\newcommand{\Clai}{C_{\al_i}^\ell}
\newcommand{\Claj}{C_{\al_j}^\ell}

\newcommand{\wodai}{w_{\al_i,\delta_i}^0}
\newcommand{\wodaj}{w_{\al_j,\delta_j}^0}
\newcommand{\widai}{w_{\al_i,\delta_i}^1}
\newcommand{\widaj}{w_{\al_j,\delta_j}^1}
\newcommand{\wldai}{w_{\al_i,\delta_i}^\ell}

\newcommand{\foaj}{f_{\al_j}^0}

\newcommand{\fiai}{f_{\al_i}^1}

\newcommand{\dei}{\delta_i}
\newcommand{\dej}{\delta_j}
\newcommand{\dek}{\delta_k}
\newcommand{\ai}{\alpha_i}
\newcommand{\aj}{\alpha_j}
\newcommand{\ajj}{\alpha_{j+1}}
\newcommand{\ak}{\alpha_k}
\newcommand{\va}{v_\alpha}
\newcommand{\vai}{v_{\alpha_i}}
\newcommand{\vi}{v_{\al_i}}
\newcommand{\vj}{v_{\al_j}}

\newcommand{\ti}{\tau_i}
\newcommand{\tj}{\tau_j}

\newcommand{\Rj}{\mathcal R_j}
\newcommand{\cj}{c_j}
\newcommand{\gp}{\mathfrak g_p}
\newcommand{\gpp}{\mathfrak g_p'}

\newcommand{\sj}{s_j}
\newcommand{\Ai}{A_i}
\newcommand{\Aj}{A_j}
\newcommand{\Ajj}{A_{j+1}}
\newcommand{\Ak}{A_k}
\newcommand{\Bo}{E_1}
\newcommand{\Bj}{E_j}
\newcommand{\Bi}{E_i}
\newcommand{\Bk}{E_k}
\newcommand{\Eap}{E_a(p)}

\newcommand{\rj}{\rho_j}
\newcommand{\rjj}{\rho_{j+1}}

\newcommand{\rp}{\rho_p}
\newcommand{\tho}{\theta_0}
\newcommand{\phoa}{\varphi_{0,\alpha}}
\newcommand{\phiF}{\phi_F}
\newcommand{\Lp}{\mathcal L_p}
\newcommand{\Lpo}{\mathcal L_p^0}
\newcommand{\Lpone}{\mathcal L_p^1}
\newcommand{\Dai}{\mathcal D_{\alpha_i}}
\newcommand{\Dadi}{\mathcal D_{\alpha_i,\delta_i}}
\newcommand{\Daj}{\mathcal D_{\alpha_j}}
\newcommand{\Dadj}{\mathcal D_{\alpha_j,\delta_j}}
\newcommand{\Wp}{\mathcal W_p}
\newcommand{\Va}{\mathcal V_{\alpha}}
\newcommand{\Vai}{\mathcal V_{\alpha_i}}
\newcommand{\Vaj}{\mathcal V_{\alpha_j}}

\newcommand{\Vak}{\mathcal V_{\alpha_k}}

\newcommand{\zaiod}{z_{\ai,\dei}^0}
\newcommand{\zajo}{z_{\aj}^0}
\newcommand{\zajod}{z_{\aj,\dej}^0}

\newcommand{\phii}{\phi_n^i}
\newcommand{\phij}{\phi_n^j}
\newcommand{\phijo}{\phi^j}
\newcommand{\gi}{\gamma_i}
\newcommand{\gj}{\gamma_j}

\newcommand{\phin}{\phi_n}
\newcommand{\sis}{\sigma_{i,n}}
\newcommand{\sjs}{\sigma_{j,n}}
\newcommand{\bi}{b_i}
\newcommand{\bj}{b_j}
\newcommand{\bjj}{b_{j+1}}
\newcommand{\hn}{h_n}
\newcommand{\thn}{\widetilde h_n}
\newcommand{\pn}{p_n}
\newcommand{\Iaj}{\mathcal I_{\aj}}

\newcommand{\mVa}{\mathcal V_{\al}}
\newcommand{\tVa}{\widetilde{\mathcal V}_{\al}}
\newcommand{\Da}{\mathcal D_\al}
\newcommand{\tDa}{\widetilde{\mathcal D}_\al}
\newcommand{\Ia}{\mathcal I_\al}
\newcommand{\fa}{f_\al}
\newcommand{\za}{z_\al}
\newcommand{\Fa}{F_\al}

\newcommand{\trj}{\widetilde r_j}
\newcommand{\tRj}{\widetilde R_j}
\newcommand{\tri}{\widetilde r_i}
\newcommand{\trjj}{\widetilde r_{j+1}}
\newcommand{\tRi}{\widetilde R_i}
\newcommand{\tRk}{\widetilde R_k}

\newcommand{\Psuj}{\underline\Psi_j}
\newcommand{\Psoj}{\overline\Psi_j}
\newcommand{\Psui}{\underline\Psi_i}
\newcommand{\Psoi}{\overline\Psi_i}
\newcommand{\lamj}{\lambda_j}
\newcommand{\Lamj}{\Lambda_j}
\newcommand{\co}{c_0}
\newcommand{\uDj}{\underline D_j}
\newcommand{\oDj}{\overline D_j}
\newcommand{\uDi}{\underline D_i}
\newcommand{\oDi}{\overline D_i}
\newcommand{\Psij}{\Psi_j}

\newcommand{\psij}{\psi_j}
\newcommand{\tpsij}{\widetilde\psi_j}
\newcommand{\tpsijj}{\widetilde\psi_{j+1}}
\newcommand{\ovC}{\overline C}

\newcommand{\tAjj}{\widetilde{\mathcal A}_{j}}
\newcommand{\dejj}{\delta_{j+1}}

\newcommand{\vphio}{\varphi^0}
\newcommand{\vphii}{\varphi^1}
\newcommand{\ej}{\varepsilon_j}
\newcommand{\ejj}{\varepsilon_{j+1}}
\newcommand{\omp}{\omega_p}
\newcommand{\Np}{\mathcal N_p}
\newcommand{\phip}{\phi_p}
\newcommand{\Tp}{\mathcal T_p}
\newcommand{\gps}{\gp''}
\newcommand{\Fg}{\mathcal F_\gamma}
\newcommand{\xip}{\xi_p}
\newcommand{\txip}{\widetilde\xi_p}
\newcommand{\zep}{\zeta_p}
\newcommand{\tzep}{\widetilde\zeta_p}
\begin{document}
\title[New  sign-changing solutions ]{New  sign-changing solutions\\ for the 2D
Lane-Emden problem with large exponents}
\author[A.~Pistoia]{Angela Pistoia}
\address[A.~Pistoia] {Dipartimento SBAI, Sapienza Universit\`{a} di Roma, 
Via Antonio Scarpa 16, 00161 Rome, Italy. Phone: +39 06 4976 6677}
\author[T.~Ricciardi]{Tonia Ricciardi}
\address[T.~Ricciardi] {Dipartimento di Matematica e Applicazioni,
Universit\`{a} di Napoli Federico II, Via Cintia, Monte S.~Angelo, 80126 Napoli, Italy. Phone: +39 081 628802}
\email{tonia.ricciardi@unina.it}
\email{angela.pistoia@uniroma1.it}
\thanks{Work partially supported by the MUR-PRIN-20227HX33Z “Pattern formation in nonlinear phenomena” and the INdAM-GNAMPA group}
%\date{\today}
\begin{abstract}
We construct  a new family of sign-changing solutions  for a two-dimensional Lane-Emden problem
with large exponent whose shape resembles a {\em tower}  with alternating sign of  bubbles  solving different singular Liouville equations on the whole plane.
\end{abstract}
\subjclass[2000]{}
\maketitle
%%%%%%%%%%%%%%%%%%%%%%%%%%%%%%%%%%%%%%%%%%%%%%%%%%%%%%%%%%%%%%%%%%%%%%%%%%%%%%%
%%%%%%%%%%%%%%%%%%%%%%%%%%%%%%%%%%%%%%%%%%%%%%%%%%%%%%%%%%%%%%%%%%%%%%%%%%%%%%%~~~~~
%%%%%%%%%%%%%%%%%%%%%%%%%%%%%%%%%%%%%%%%%%%%%%%%%%%%%%%%%%%%%%%%%%%%%%%%%%%%%%%
%%%%%%%%%%%%%%%%%%%%%%%%%%%%%%%%%%%%%%%%%%%%%%%%%%%%%%%%%%%%%%%%%%%%%%%%%%%%%%%
\section{Introduction and statement of the main results}
\label{sec:intro}
We consider the classical Lane-Emden problem:
\beq
\label{eq:pb}
\left\{
\bal
-\Delta u=&|u|^{p-1}u&&\hbox{in\ }\Om \\
u=&0&&\hbox{on\ }\pl\Om,
\eal
\right.
\eeq
where $\Om\subset\rr^2$ is a smooth bounded domain and $p>1$.
\\
Standard variational tools ensure the existence of     infinitely many (possibly sign-changing) solutions of problem \eqref{eq:pb}.\\

In the last decade, a fruitful line of research has been the  study the asymptotic behaviour  of solutions to \eqref{eq:pb}
as $p$ approaches $+\infty$. 
In the pioneering works \cite{renwei1,renwei2}, Ren and Wei   study the profile
of the least energy solution $u_p$ (which is positive) and prove that it has a single point concentration and converges at zero locally uniformly outside the concentration point. Later, El Mehdi and Grossi \cite{elgro} and Adimurthi and Grossi \cite{adigro}   identify a limit problem by showing that 
a suitable scaling   converges to a radial solution of 
\begin{equation}\label{plim}
-\Delta U=e^U\ \hbox{in}\ \rr^2,\ \int\limits_{\rr^2}e^U=8\pi.\end{equation}
Concerning general positive solutions (i.e. not necessarily with least energy) the first asymptotic analysis was carried out by
De Marchis,  Ianni and Pacella \cite{dip} (see also \cite{deiapa2,degroiapa}). More recently, other contributions have been given by Kamburov and Sirakov \cite{kamsir} and
Thizy \cite{thi}. We can summarize the known results as follows: the set of positive solutions
is uniformly bounded in $p$, each positive solution 
 concentrates at a finite number $k$ of points $x_1,\dots,x_k$ in $\Omega$,  converges to zero locally uniformly outside the concentration set 
 and suitable scalings of it
 about each peak converge to the   bubble $U$, i.e  solution of the limit problem \eqref{plim}.
Moreover, the $k-$upla of peaks $ (x_1,\dots,x_k)$ is a critical point of the 
{\em Kirchhoff-Routh} $\Psi_k:{\mathcal D}\to\mathbb R$, with ${\mathcal D} =   \{ (x_1, \dots , x_k ) \in \Omega^k \, : \,
x_i \not =x_j  \}$ defined by
 \begin{equation}\label{kr}
\Psi_k(x_1,\dots,x_k):=\sum\limits_{i=1,\dots,k}\sigma_i^2H(x_i,x_i)+\sum\limits_{i,j=1,\dots,k\atop i\not=j}\sigma_i\sigma_jG(x_i,x_j). \end{equation}
with all the $\sigma_i$'s equal to $+1.$ 

Here $G(x,y)$ is  the Green's function
and   $H(x,y)$
its regular part. That is,
\beqs
\bca
-\Delta_xG(x,y)=\de_y&x,y,\in\Om,\ x\neq y\\
G(x,y)=0&x\in\partial\Om,\ y\in\Om
\eca
\eeqs
and
\beqs
G(x,y)=\frac{1}{2\pi}\ln\frac{1}{|x-y|}+H(x,y).
\eeqs
We denote by 
\beq
\label{def:Robin}
h(x)=H(x,x),\qquad x\in\Om,
\eeq 
the Robin function.  \\

The first contribution to the analysis of the concentration phenomenon
for sign-changing solutions of problem \eqref{eq:pb} as  $p\to+\infty$ is due to Grossi, Grumiau and Pacella in \cite{ggp2} who study  the profile and qualitative properties of the least energy nodal solution. In particular, the authors show that it  concentrates at two different points  $x_1$ and $x_2$ in $\Omega$ 
 and up to  suitable scalings around 
 each peak it converges to a positive and a negative   bubble $U$  (see \eqref{plim}).
Moreover, the pair of peaks $ (x_1,x_2)$ is a critical point of the 
{\em Kirchhoff-Routh}  function \eqref{kr} with $k=2$ and $\sigma_1=+1$, $\sigma_2=-1.$\\
Successively, the same  authors Grossi, Grumiau and Pacella in \cite{ggp}
 analyze the
asymptotic behavior of the least energy radial nodal solution  in the ball and prove that
its positive and negative parts  concentrate at the origin    and the limit profile looks like a tower of two bubbles given by a superposition of the   radial solution to \eqref{plim} and  the radial solution of the  {\em singular} Liouville problem in $\mathbb R^2$
\begin{equation}\label{plims}
-\Delta U=|x|^\beta e^U\ \hbox{in}\ \rr^2,\ \int\limits_{\rr^2}e^U=4\pi(\beta+2) \end{equation}
for a suitable non-integer and positive $\beta$.
Later, such a result  has been generalized   to other symmetric domains by De Marchis, Ianni and Pacella \cite{dip}.
Recently, Ianni and Saldana \cite{is} provide a complete and extremely accurate asymptotic analysis of the radial solutions of problem \eqref{eq:pb} when the domain is the unit ball.   
In particular, they prove that the radial solution with $m$ nodal lines looks like a superposition  (with alternating sign) of  the  radial  solution to \eqref{plim} and the radial solutions of $m$ different  Liouville problem \eqref{plims} with different non-integers and positive $\beta_1,\dots,\beta_m$.
 \\
 
 As far as it concerns the existence of positive and sign-changing solutions with $k$ concentration points, Esposito, Musso and Pistoia in \cite{emp1,emp2}
proved that any {\em non-degenerate} critical point $ (x^*_1,\dots,x^*_k)$ of the Kirchhoff-Routh function \eqref{kr} with $\sigma_i\in\{-1,+1\}$,
 generates a solutions with $k-$peaks
approaching the points $  x^*_1,\dots,x^*_k$  as $p$ is large enough being the peak $x^*_i$ positive or negative if $\sigma_i=+1$ or $\sigma_i=-1$, respectively.  
In this regard, it is useful to point out that Micheletti and Pistoia \cite{mipi} and Bartsch, Micheletti and Pistoia \cite{bmp} proved that for   generic domains $\Omega$ all the critical points of the  Kirchhoff-Routh function  are non-degenerate. The limit profile of the solutions built in \cite{emp1,emp2} resembles the sum of $k$  {\em bubbles} solutions to \eqref{plim} concentrated in  the different points $  x^*_i$. 
In particular,  in \cite{emp1} it is also shown that if the domain $\Omega$ is has a rich geometry (namely it is not contractible or it has a dumbell shape)  positive multipeak   solutions can be found. On the contrary if the domain $\Omega$ is convex, positive solutions with more than one peak do not exist as proved by Grossi and Takahashi in \cite{grota}. The scenery of sign-changing solutions is completely different. Indeed, 
in any domain $\Omega$ there exist  at least two pairs of solutions with one positive peak and one negative peak as proved  in \cite{emp2}. In particular, Bartsch, Pistoia and Weth \cite{bpw1,bpw2} prove that  if the domain is symmetric with respect to the  line $\mathbb R\times \{0\}$
 there exist  infinitely many sign-changing solutions whose peaks are  aligned on the symmetry line $\Omega\cap \mathbb R\times \{0\}$. We point out that in the case of the ball these solutions are not radially symmetric.
\\

 In the present paper, we will build a new type of sign-changing solutions in a general domain whose profile resembles the tower of bubbles as predicted in the radial case by \cite{is}.\\
For the sake of simplicity, we will assume that the domain $\Omega$ is symmetric with respect to the origin $0\in\Omega$, i.e. $x\in\Omega$ iff $-x\in\Omega.$ 
We aim to construct a sign-changing tower of peaks for \eqref{eq:pb} concentrating at $0\in\Om$, in the spirit of \cite{emp1, gropi}.
\\

In order to state our main result, we briefly recall some known definitions and facts.\\
In view of the standard expansion:
\beq
\label{eq:introtaylorzero}
\left(1+\frac{v}{p}\right)^p=e^v\left\{1+O\left(\frac{1}{p}\right)\right\}\qquad\qquad\hbox{as }p\to+\infty,
\eeq
for any fixed $v\in\rr$, the equation
$-\Delta v=v^p$ (corresponding to \eqref{eq:pb} for $u\ge0$) may be viewed as a perturbation of the Liouville equation $-\Delta v=e^v$.
Therefore, a family of solutions $u_p$ to \eqref{eq:pb}, with a peak at $0\in\Om$ as $p\to+\infty$,  is expected to exist in the form
\beq
\label{eq:zeroansatz}
u_p(x)=\tau PU_\delta(x)+\om_p,
\eeq
where $\tau,\de>0$ are parameters depending on $p$ with $\tau,\delta\to0$, $P$ denotes the standard projection on $H_0^1(\Om)$,
\[
U_\delta(x)=\ln\frac{8\de^2}{(\de^2+|x|^2)^2}
\] 
denotes the concentrating family of radial solutions to the Liouville equation \eqref{plim}
and $\omp$ is an error, see \eqref{eq:basicapprox} in the Appendix for details.
However, as observed in \cite{emp1, emp2}, where multiple isolated peak solutions 
to \eqref{eq:pb} are constructed, 
the first order expansion \eqref{eq:introtaylorzero} 
is not sufficiently accurate,
so that two higher order correction terms $w^0$ and $w^1$ must be included in ansatz~\eqref{eq:zeroansatz}.
More precisely, the \emph{third order} expansion of the nonlinearity in \eqref{eq:pb} is needed. For fixed $v,w^0,w^1\in\rr$ such
an expansion takes the form
\beq
\label{eq:introtaylor}
\left(1+\frac{v}{p}+\frac{w^0}{p^2}+\frac{w^1}{p^3}+o\left(\frac{1}{p^3}\right)\right)^p
=e^v\left\{1+\frac{w^0-\vphio(v)}{p}+\frac{w^1-\vphii(v,w^0)}{p^2}+o\left(\frac{1}{p^2}\right)\right\},
\eeq
where $\vphio,\vphii$ are the functions defined by
\beq
\label{def:phi}
\vphio(v):=\frac{v^2}{2},
\qquad
\vphii(v,w^0):=vw^0-\frac{v^3}{3}-\frac{(w^0)^2}{2}-\frac{v^4}{8}+\frac{v^2w^0}{2},
\eeq
see Lemma~\ref{lem:Taylor} in the Appendix for more details on this expansion.
Thus, ansatz~\eqref{eq:zeroansatz} is replaced by the following ansatz employed in \cite{emp1, emp2}: 
\beqs
u_p(x)=\tau P\left(U_\delta(x)+\frac{w^0(x/\de)}{p}+\frac{w^1(x/\de)}{p^2}\right)+\om_p,
\eeqs
where $w^0$, $w^1$ are the smooth radial solutions, whose existence is established in \cite{ChaeImanuvilov} (see also Lemma~\ref{lem:CI} in the Appendix), to the linearized equations
\beqs
\Delta w^0+e^vw^0=e^v\vphio(v)
\qquad\hbox{and}\qquad
\Delta w^1+e^vw^1=e^v\vphii(v,w^0)\qquad\hbox{on }\rr^2,
\eeqs
satisfying 
\beqs
w^\ell(y)=C^\ell\ln|y|+O\left(\frac{1}{|y|}\right)\qquad\hbox{as }|y|\to+\infty,\ \ell=0,1.
\eeqs
The first order approximation for the tower of peaks
is given by a superposition of a finite number of   bubbles $\Udai$, 
solutions to the singular Liouville problems
\beq
\label{eq:singLiou}
\left\{
\bal
&-\Delta U=|y|^{\ai-2}e^U&&\hbox{on\ }\rr^2\\
&\int_{\rr^2}|y|^{\ai-2}e^U\,dx=4\pi\ai,
\eal
\right.
\eeq
given by 
\beqs
\Udai(x)=\vi(\frac{x}{\dei})-\ai\ln\dei=\ln\frac{2\ai^2\dei^{\ai}}{(\dei^{\ai}+|x|^{\ai})^2},
\eeqs
where $\ai\ge2$, $\dei>0$ are suitably chosen and
\[
\vi(y)=\ln\frac{2\ai^2}{(1+|y|^{\ai})^2},\quad i=1,2,\ldots,k.
\]
The corresponding lower order corrections $\woai$, $\wiai$, are given by the radial solutions 
to the linearized equations
\beq
\label{def:wa}
\bal
&\Delta\woai+|y|^{\ai-2}e^{\vi}\woai=|y|^{\ai-2}e^{\vi}\,\vphio(\vi+\ln|y|^{\ai-2}),\qquad\hbox{on }\rr^2,\\
&\Delta\wiai+|y|^{\ai-2}e^{\vi}\wiai=|y|^{\ai-2}e^{\vi}\,\vphii(\vi+\ln|y|^{\ai-2},\woai)
\qquad\hbox{on }\rr^2,
\eal
\eeq
satisfying 
\beq
\label{bdrycond:wa}
\wlai(y)=\Clai\ln|y|+O(|y|^{-1})\qquad \hbox{as } |y|\to+\infty, \qquad\ell=0,1,
\eeq
whose existence is established in \cite{ChaeImanuvilov}. 
\par
With this notation our main result may be stated as follows.
\begin{thm}
\label{thm:main}
For any $k\in\mathbb N$ there exists a sufficiently large $p_0>0$ such that
for all $p\ge p_0$ there exists a sign-changing tower of peaks solution to
\eqref{eq:pb} of the form $\up=\Up+\om_p$,
\beqs
\Up(x)=\sum_{i=1}^k(-1)^{i-1}\ti P\left(\Udai(x)+\frac{\woai(\dei^{-1}x)}{p}+\frac{\wiai(\dei^{-1}x)}{p^2}\right),
\eeqs
where, for $i=1,2,\ldots,k$, there holds $\ti=O(p^{-1})$, $2=\al_1<\al_2<\ldots<\ai<\al_{i+1}\ldots<\ak$, 
$\dei=C_ie^{-b_ip}$ with $b_1>b_2>\ldots>b_k>0$ and $C_i>0$ and $\woai$, $\wiai$ are defined by 
\eqref{def:wa}.
\end{thm}

\begin{rmk}\rm 
We point out that by the building process we easily deduce that the solution $\up=\Up+\om_p$ has $k$ nodal regions  which shrink to the origin as $p\to\infty.$\\    It is also worthwhile noticing that the main order term of our solution, i.e. the sum of   bubbles $\sum _{i}(-1)^i\tau_iP\Udai$ with alternating sign,  coincide  in the radial case with the profile   described by
Ianni and Saldana in \cite[Theorem 2.5]{is}. 
%From the opposing point of view, it is important to  stress  that in order to construct the solution with such a profile,  a more accurate ansatz is needed.  %Actually, we have to refine the ansatz  
%by adding the two extra terms $\frac1p\sum _{i}(-1)^i\tau_iP w^0_{\alpha_i,\delta_i} $ and $ \frac1{p^2}\sum _{i}(-1)^i\tau_iP w^1_{\alpha_i,\delta_i}$.   \\
An interesting problem would be to show that all the  solutions to  \eqref{eq:pb} having $k$ nodal regions shrinking to the origin as $p\to\infty$ look like the first order approximation term  $\Up$. It should be the first step to prove the local uniqueness of the nodal solution   in  the same spirit of  Grossi, Ianni, Luo and Yan \cite{GILY}. We observe that the  unicity of the radial solution with $k$ nodal regions has been proved by Kajikiya in \cite{ka} using ODE techniques. 
\end{rmk}

\begin{rmk}\rm
Our   result claims the existence of symmetric  sign-changing solutions which look like a tower of bubbles centered at the origin  being the domain $\Omega$ symmetric with respect to it. 
The symmetry assumptions   simplify  the computations a lot because the linearized operator is invertible in the space of symmetric functions  and the solution can be found merely using a  contraction mapping argument. In general
we strongly believe that it is possible to carry out the construction around at any non-degenerate critical point $x_0$ of the Robin's function \eqref{def:Robin}.   This   could be managed  introducing  new parameters which are
the centers of the  bubbles  close to $x_0$,  which at the prices of heavy technicality should allow to perform a classical Ljapunov-Schmidt reduction.

\end{rmk}

This article is organized as follows. 
 Section~\ref{sec:ansatz} is dedicated to the choice of the parameters $\aj,\dej,\tj$ appearing in Theorem~\ref{thm:main}. More precisely,
we first derive conditions for the parameters in the form of a nonlinear a system (see \eqref{eq:adtsys}) which ensure smallness of the \lq\lq error" $\mathcal R_p:=\Delta\Up+\gp(\Up)$ with respect to
a suitable $L^\infty$-weighted norm $\rp$. By rather involved \textit{ad hoc} arguments we prove solvability of the system and check that $\aj\not\in\mathbb N$ for $j\ge2$, 
which is essential for the invertibility of the linearized operator.  
In Section~\ref{sec:Rp} we carry out the details of the estimates of the error $\Rp$. To this end, we partition $\Om$ into \lq\lq shrinking annuli" $\Aj$ in the spirit of \cite{gropi},
although our choice of the $\Aj$'s is more delicate. 
Sections 4 to 7 are dedicated to the analysis of the linearized operator $\Lp=\Delta+\gp'(\Up)$. 
In particular, Sections~4--5 are concerned with estimations in shrinking rings, and include delicate \textit{ad hoc} arguments to compute the integrals involving radial eigenfunctions.
Section~7 extends the estimates by a barrier function in the spirit of \cite{DelPinoEspositoMusso2012}, with new ingredients. 
Finally, in Section~8 we conclude the proof of Theorem~\ref{thm:main}
by a fixed point argument. For the reader's convenience, we collect in the Appendix some technical estimates as well as some known results.
  
\subsection* {Acknowledgement} We would like to thank Luca Battaglia for his   contribution in the proof of Proposition \ref{prop:aj}.
%\subsection*{Comments}
%Some intuitive facts.
%\begin{itemize}
%  \item 
%Fast bubbles are perceived as \lq\lq singularities". Slow bubbles are perceived as \lq\lq large constants".
%\item
%The form of the functions $\Vai$ is essential to the estimates.
%\item
%The Taylor expansion \lq\lq costs" some (harmless)logarithmic error terms.
%\item
%Estimates will be carried out in the \lq\lq shrinking rings" $\Aj$,
%where the $j$-th bubble dominates, and $\Bj\subset\Aj$, 
%where the Taylor expansion of the nonlinearity $\gp$ holds uniformly with respect to $p\to+\infty$.
%The set $\Aj$ contains \lq\lq most" of the $j$-th \lq\lq mass" $|x|^{\aj-2}e^{\Udaj}$.
%The choice of $\Aj$ is more delicate than in \cite{gropi}.
%\item
%The weight $\rj$ is chosen so that the $j$-th mass with logarithmic corrections is bounded,
%and all other masses vanish exponentially with respect to $p\to+\infty$.
%\item
%Many estimates depend on the uniform estimate
%\beqs
%\left|\Vaj(y)-\ln\frac{|y|^{\aj-2}}{1+|y|^{2\aj}}\right|\le C,
%\eeqs
%for some $C>0$ independent of $y\in\rr^2$.
%\item
%In fact, the logarithmic singularity at $y=0$ of $\Vaj$ for $\aj\neq2$ requires new estimates.
%\edz{check}The onle relevant property of the $\woaj$'s seems to be the logarithmic growth at infinity.
%\item
%Mention something about the weight $\rp$
%\end{itemize}

%%%%%%%%%%%%%%%%%%%%%%%%%%%%%%%%%%%%%%%%%%%%%%%%%%%%%%%%%%%%%%%%%%%%%%%%%%%%%%%%%%
%%%%%%%%%%%%%%%%%%%%%%%%%%%%%%%%%%%%%%%%%%%%%%%%%%%%%%%%%%%%%%%%%%%%%%%%%%%%%%%%%%%%%%%%%%%%%
%%%%%%%%%%%%%%%%%%%%%%%%%%%%%%%%%%%%%%%%%%%%%%%%%%%%%%%%%%%%%%%%%%%%%%%%%%%%%%%%%%%%%%%%%%%%%
\section{Ansatz for the $k$-peak tower and choice of parameters}
\label{sec:ansatz}
%%%%%%%%%%%%%%%%%%%%%%%%%%%%%%%%%%%%%%%%%%%%%%%%%%%%%%%%%%%%%%%%%%%%%%%%%%%%%%%%%%%%%%%%%%%%%
%%%%%%%%%%%%%%%%%%%%%%%%%%%%%%%%%%%%%%%%%%%%%%%%%%%%%%%%%%%%%%%%%%%%%%%%%%%%%%%%%%%%%%%%%%%%%
\par
In what follows, we denote by $\gp=\gp(t)$ the nonlinearity appearing in \eqref{eq:pb}, namely we set:
\beq
\label{def:gp}
\gp(t):=|t|^{p-1}t, \qquad t\in\rr.
\eeq

We recall from Section~\ref{sec:intro} that we seek solutions to \eqref{eq:pb} of the 
\lq\lq sign-changing tower of peaks" form:
\beq
\label{eq:ansatz}
\bal
&\up=\Up+\varphi_p\\
&\Up
=\sum_{i=1}^k\ti(-1)^{i-1}P\left(\Udai+\frac{\wodai}{p}+\frac{\widai}{p^2}\right),
\eal
\eeq
where for $\ai\ge2$, $\dei>0$, we have that
\beq
\label{def:singLiou}
\Udai(x)=\ln\frac{2\ai^2\dei^{\ai}}{(\dei^{\ai}+|x|^{\ai})^2},
\eeq
is the family of radial solutions to \eqref{eq:singLiou},
\beq
\wldai(x)=\wlai(\frac{x}{\dei}),\qquad\ell=0,1,\quad x\in\Om,
\eeq
where $\woai$, $\wiai$ are the correction 
terms defined in \eqref{def:wa}, see Lemma~\ref{lem:CI} for details.
\par

\subsection{Notation}
Henceforth, we denote by $C>0$ a general constant independent of $p$.
For any measurable set $E\subset\Om$ we denote by $\chi_E$ the characteristic function of $E$.
By uniform convergence in $E$ we understand uniform convergence in $\overline E$. 
For any two families of sets $A_p,B_p\subset\rr^2$ depending on $p$, by $A_p\cong B_p$
we understand that there exists a constant $K>0$ independent of $p$ such that
\beq
\label{def:AcongB}
K^{-1}B_p\subset A_p\subset KA_p,\qquad \mbox{for all }p\to\infty
\eeq
where $KB=\{Kx:x\in B\}$.
For $i=1,2,\ldots,k$ it will be convenient to set
\[
\vi(y):=U_{\ai,1}(y)=\ln\frac{2\ai^2}{(1+|y|^\ai)^2},\qquad y\in\rr^2,
\]
so that the scaling property may be written in the form
\beq
\label{eq:Udascaling}
\Udai(x)=\vi(\frac{x}{\dei})-\ai\ln\dei.
\eeq
It is convenient to set
\beq
\label{def:Va}
\Vai(y):=\vi(y)+(\ai-2)\ln|y|=\ln\frac{2\ai^2|y|^{\ai-2}}{(1+|y|^{\ai})^2},
\eeq
so that we may write \eqref{def:wa} in the form:
\beqs
\bal
&\Delta\woai+e^{\Vai(y)}\woai=e^{\Vai(y)}\varphi^0(\Vai)\qquad\hbox{on }\rr^2\\
&\Delta\wiai+e^{\Vai(y)}\wiai=e^{\Vai(y)}\varphi^1(\Vai,\woai)\qquad\hbox{on }\rr^2.
\eal
\eeqs
The \lq\lq mass" takes the form:
\beqs
|x|^{\ai-2}e^{\Udai(x)}=\frac{|y|^{\ai-2}e^{\vi(y)}}{\dei^2}=\frac{e^{\Vai(y)}}{\dei^2},
\qquad x=\dei y\in\Om.
\eeqs
We set
\beq
\label{def:wodai}
\wodai(x)=\woai(\frac{x}{\dei}),\qquad \widai(x)=\wiai(\frac{x}{\dei}),\qquad i=1,2,\ldots,k.
\eeq
It follows that $\wodai$, $\widai$, satisfy
\beq
\label{eq:wdea}
\bal
&\Delta\wodai+|x|^{\ai-2}e^{\Udai}\wodai
=|x|^{\ai-2}e^{\Udai}\vphio(\Vai(\frac{x}{\dei})),\qquad x\in\Om\\
&\Delta\widai+|x|^{\ai-2}e^{\Udai}\widai
=|x|^{\ai-2}e^{\Udai}\vphii(\Vai(\frac{x}{\dei}),\wodai(x)),\qquad x\in\Om.
\eal
\eeq
For later estimates it is essential to observe that $\Vai$ satisfies
\beq
\label{eq:Vaprop}
\ln\frac{|y|^{\ai-2}}{1+|y|^{2\ai}}-C\le\Vai(y)\le\ln\frac{|y|^{\ai-2}}{1+|y|^{2\ai}}+C
\eeq
for some $C>0$ independent of $y\in\rr^2$.
We observe that 
\beq
\label{eq:Uda}
\bal
-\Delta&\left(\Udai+\frac{\wodai}{p}+\frac{\widai}{p^2}\right)\\
&\qquad=|x|^{\ai-2}e^{\Udai}
\left(1+\frac{\wodai-\vphio(\Vai(\dei^{-1}x))}{p}+\frac{\widai-\vphii(\Vai(\dei^{-1}x),\wodai)}{p^2}\right).
\eal
\eeq
Occasionally within some proofs it will be convenient to denote by $\tUdai$, $\tvi$, $\tai$, 
$i=1,2,\ldots,k$,
the \lq\lq corrections" by the lower order terms $\wodai$, $\widai$, or by the related constants $\Coai$, $\Ciai$:
\beq
\label{def:tUdai}
\tUdai=\Udai+\frac{\wodai}{p}+\frac{\widai}{p^2},
\eeq
\beq
\label{def:tildev}
\tvi:=v_{\ai}+\frac{\woai}{p}+\frac{\wiai}{p^2}.
\eeq
and
\beq
\label{def:ta}
2\tai:=2\ai-\frac{\Coai}{p}-\frac{\Ciai}{p^2}.
\eeq

\subsection{The choice of parameters}
Our aim in this section is to 
choose $\al_i,\de_i,\ti$, so that the \lq\lq error" $\Rp$ defined by
\beq
\label{def:R}
\mathcal R_p:=\Delta\Up+\gp(\Up)
\eeq
vanishes with order $p^{-4}$ in a suitable weighted norm.
More precisely, in Proposition~\ref{prop:Rjest} we derive conditions for the parameters $\ai$, $\dei$, $\ti$, 
in the form of a nonlinear system, which imply \lq\lq smallness" of $\Rp$.
In Proposition~\ref{prop:solvability} we prove the solvability of the system.
Finally, in Proposition~\ref{prop:aj} we show that $\ai\not\in\mathbb N$ for $i\ge2$,
which is essential to prove invertibility of the linearized problem.
\par
We make the following assumptions on the parameters $\al_i=\al_i(p)$, $\de_i=\de_i(p)$, $\ti=\ti(p)>0$, $i=1,2,\ldots,k$: 
\begin{enumerate}
  \item[(A1)]
The singularity coefficients $\al_i$ increase as the index $i$ increases: 
\beqs
\label{ans:al}
2=\al_1<\al_2<\ldots<\al_i<\al_{i+1}<\ldots<\al_k\le C;
\eeqs
\item[(A2)]
the concentration parameters $\de_i$ vanish exponentially with respect to $p$,
with decreasing speed with respect to $i$:
\beqs
\label{ans:de}
\de_i=C_ie^{-b_ip},\quad i=1,2,\ldots,k;\quad \ C^{-1}\le b_k<b_{k-1}<\ldots<b_2<b_1\le C,
\eeqs
where $C^{-1}\le C_i=C_i(p)\le C$, $b_i=b_i(p)$, $i=1,2,\ldots,k$.
In particular,
\beq
\label{eq:deideii}
\de_i=o(\de_{i+1}),\quad i=1,2,\ldots,k-1;
\eeq
\item[(A3)]
the coefficients $\ti$ satisfy: 
\beqs
\label{ans:ti}
\ti=O(\frac{1}{p}), \qquad C^{-1}\le \frac{\ti}{\tj}\le C,\qquad i,j=1,2,\ldots,k.
\eeqs
\end{enumerate}
Ansatz~(A3) implies that $\ti P\Udai$ is $L^\infty$-bounded, consistently
with the estimates in \cite{renwei1, renwei2}. 
\par
%%%%%%%%%%%%%%%%%%%%%%%%%%%%%%%%%%%%%%%%%%%%%%%%%%%%%%%%%%%%%%%%%%%%%%%%%%%%
%%%%%%%%%%%%%%%%%%%%%%%%%%%%%%%%%%%%%%%%%%%%%%%%%%%%%%%%%%%%%%%%%%%%%%%%%%%%
\subsection*{A partition of $\Om$: the \lq\lq shrinking annuli" $\Aj$}
In order to estimate the error $\Rp$ defined in \eqref{def:R}, extending ideas in \cite{gropi}, we partition $\Om$ into suitable
sets $\Aj$, $j=1,2,\ldots,k$, such that the $j$-th bubble
$P\tUdaj$ is the dominant bubble in $\Aj$. Our choice of $\Aj$ is more delicate than in \cite{gropi},
and is optimal with respect to the property:
\beq
\label{est:Vaj}
\Vaj(y)\ge-p-C\ \hbox{uniformly for }y\in\frac{\Aj}{\dej},
\eeq
for some $C>0$ independent of $p\to+\infty$, see
Lemma~\ref{lem:ejineq}.
\par
For $j=1,\ldots,k-1$, let $0<\ej<1$ be defined by
\beq
\label{def:ejt}
\ej:=\frac{\tau_{j+1}}{\tj+\tau_{j+1}},
\eeq
where $\tj$ is the constant appearing in \eqref{eq:ansatz}.
We define
\beq
\label{def:Aj}
\bal
A_1:=&\{x\in\Om:\ 0\le|x|<\de_1^{\veps_1}\de_2^{1-\veps_1}\}\\
\Aj:=&\{x\in\Om:\ \de_{j-1}^{\veps_{j-1}}\de_{j}^{1-\veps_{j-1}}\le|x|<\de_j^{\ej}\de_{j+1}^{1-\ej}\},
\qquad j=2,\ldots,k-1\\
\Ak:=&\{x\in\Om:\ |x|\ge\de_{k-1}^{\veps_{k-1}}\de_k^{1-\veps_{k-1}}\}.
\eal
\eeq
We note that the set $A_1$ is a shrinking ball, the sets $\Aj$, $j=2,\ldots,k-1$ are shrinking annuli
and the set $A_k$ is an annulus invading $\Om$.
For later use we observe that
\beq
\bal
\label{def:Ajdej}
\frac{A_1}{\de_1}=&\left\{y\in\rr^2:\ \de_1 y\in\Om\ \hbox{and }
0\le|y|<\left(\frac{\de_{2}}{\de_1}\right)^{1-\varepsilon_1}\right\},\\
\frac{\Aj}{\dej}=&\left\{y\in\rr^2:\ \dej y\in\Om\ \hbox{and } \left(\frac{\de_{j-1}}{\dej}\right)^{\veps_{j-1}}
\le|y|<\left(\frac{\de_{j+1}}{\dej}\right)^{1-\ej}\right\},
\ j=2,\ldots,k-1,\\
\frac{\Ak}{\dek}=&\left\{y\in\rr^2:\ \de_k y\in\Om\ \hbox{and } |y|\ge\left(\frac{\de_{k-1}}{\dek}\right)^{\veps_{k-1}}
\right\}.
\eal
\eeq
In particular, the rescaled sets $\Aj/\dej$, $j=1,2,\ldots,k$, invade the whole space $\rr^2$ as $p\to+\infty$.
We also check that
by choice of $\aj$, $\dej$, there holds
\beq
\label{eq:Ajuseful}
\bal
&\frac{A_1}{\de_1}\cong\left\{y\in\rr^2:\ \de_1 y\in\Om\ \hbox{and }0\le|y|< e^{\frac{p}{4}}\right\},\\
&\frac{\Aj}{\dej}\cong\left\{y\in\rr^2:\ \dej y\in\Om\ \hbox{and }e^{-\frac{p}{\aj-2}}\le|y|<e^{\frac{p}{\aj+2}}\right\},\\
&\frac{\Ak}{\dek}\cong\left\{y\in\rr^2:\ \dek y\in\Om\ \hbox{and }|y|\ge e^{-\frac{p}{\ak-2}}\right\},
\eal
\eeq
where the relation $\cong$ is defined in Section~\ref{sec:intro}.
We note that the expansion rate of $A_1/\de_1$ in \eqref{eq:Ajuseful} is consistent with \cite{emp1}.
\subsection*{A family of level sets for $\Vaj$: the sets $\Bj$}
We shall  need the following subsets $\Bj\subset\Om$,
where the Taylor expansion as stated in Lemma~\ref{lem:Taylor} holds uniformly:
\beq
\label{def:Bj}
\bal
\Bj:=&\left\{x\in\Om:\ 1+\frac{\Udaj(x)+\ln|x|^{\aj-2}+\ln\dej^2}{p}>\frac{1}{2}\right\}
=\left\{x\in\Om:\ \Vaj(\frac{x}{\dej})>-\frac{p}{2}\right\}.
\eal
\eeq
In view of the form of $\Vaj$ as in \eqref{eq:Vaprop}, it is clear that the shrinking rate of $\Bj$ is given by
\beqs
\bal
&\Bo\stackrel{\sim}{=}\{x\in\Om:\ 0\le|x|<C\de_1e^{\frac{p}{8}}\},\\
&\Bj\stackrel{\sim}{=}\{x\in\Om:\ C^{-1}\dej e^{-\frac{p}{2(\aj-2)}}<|x|<C\dej e^{\frac{p}{2(\aj+2)}}\},
\quad j=2,\ldots,k,
\eal
\eeqs
see Lemma~\ref{lem:BsubsetA} below for the precise statement. 
In view of \eqref{eq:Ajuseful} we have
\beqs
\Bj\subset\Aj,\qquad j=1,2,\ldots,k.
\eeqs
\subsection*{The $j$-th error $\Rj$}
We recall from \eqref{def:R} that
\beqs
\mathcal R_p:=\Delta\Up+\gp(\Up).
\eeqs
For every $j=1,2,\ldots,k$ we define the \lq\lq$j$-th error''
\beq
\label{def:Rj}
\Rj:=\left[\tj(-1)^{j-1}\Delta\tUdaj+\gp(\Up)\right]\chi_{\Aj},
\qquad x\in\Om
\eeq
so that we may write 
\beq
\label{eq:Rsplit}
\Rp=\sum_{j=1}^k\Rj+\sum_{j=1}^k\chi_{\Aj}\sum_{\stackrel{i=1}{i\neq j}}^k(-1)^{i-1}\ti\Delta\tUdai.
\eeq
For $j=1,2,\ldots,k$, let
\beq
\label{def:cj}
c_j:=4\pi h(0)\sum_{i=1}^k(-1)^{i-j}\frac{\ti}{\tj}\tai
-\ln(2\al_j^2)
+\sum_{i>j}(-1)^{i-j}\frac{\ti}{\tj}\left(\frac{\woai(0)}{p}
+\frac{\wiai(0)}{p^2}\right),
\eeq
where $\tai$, $h$ are defined in \eqref{def:ta}, \eqref{def:Robin}, respectively.
Our first aim in this section is to establish the following.
\begin{prop}[Decay estimate for $\Rj$ under suitable assumptions on $\aj$, $\dej$, $\tj$]
\label{prop:Rjest}
Let $\Up$ be defined in \eqref{eq:ansatz}, where the parameters
$\al_i=\al_i(p)$, $\de_i=\de_i(p)$, $\ti=\ti(p)$, $i=1,2,\ldots,k$, are solutions to the system:
\beq
\label{eq:adtsys}
\left\{
\bal
&\al_1=2\\
&\al_j=2-\sum_{1\le i<j}\frac{\ti}{\tj}(-1)^{i-j}\left(2\ai-\frac{C_{\ai}^0}{p}-\frac{C_{\ai}^1}{p^2}\right),
&&j=2,\ldots,k\\
&-\left(\al_k-\frac{C_{\ak}^0}{p}-\frac{C_{\ak}^1}{p^2}+2\right)\ln\de_k+c_k=p\\
&-\left(\al_j-\frac{C_{\aj}^0}{p}-\frac{C_{\aj}^1}{p^2}+2\right)\ln\de_j\\
&\qquad-2\sum_{j<i\le k}\frac{\ti}{\tj}(-1)^{i-j}\left(\ai-\frac{C_{\ai}^0}{2p}-\frac{C_{\ai}^1}{2p^2}\right)\ln\de_i
+c_j=p,
&&j=1,2,\ldots,k-1\\
&p^p\tj^{p-1}\de_j^2=1,
&&j=1,2,\ldots,k
\eal
\right.
\eeq
and satisfy assumptions (A1)--(A2)--(A3),
and where the correction profiles $w_{\al_i}^0$, $w_{\al_i}^1$, $i=1,2,\ldots,k$
are defined by \eqref{def:wa}. 
Then, there holds the estimate
\beq
\label{eq:Rjest}
\bal
\Rj(x)=&|x|^{\aj-2}e^{U_{\aj,\dej}(x)}O\left(\frac{|(\aj-2)\ln\frac{|x|}{\dej}|^6+\ln^6(2+\frac{|x|}{\dej})}{p^4}\right)\\
=&\frac{|y|^{\aj-2}e^{\vj(y)}}{\dej^2}O\left(\frac{|(\aj-2)\ln|y||^6+\ln^6(2+|y|)}{p^4}\right),
\eal
\eeq
uniformly for $x=\dej y\in\Bj$.
\end{prop}
\begin{rmk}
The form of $\Rj$ determines the form of the suitable weighted norm $\rp$ to be introduced later.
In particular, we note that since $|x|^{\aj-2}e^{U_{\aj,\dej}(x)}=\dej^{-2}|y|^{\aj-2}e^{\vj(y)}$,
$x=\dej y$, we expect a $\dej^2$ term in $\rp$.
\end{rmk}
\begin{rmk}
We observe that in the \lq\lq one bubble case" $k=1$, $\al_1=2$, the error estimate in \eqref{eq:Rjest} takes the form
\beqs
\mathcal R_1(x)=\frac{8}{\de_1^2(1+|y|^2)^2}O\left(\frac{\ln^6(2+|y|)}{p^4}\right),\qquad y=\de_1^{-1}x,
\eeqs
in agreement with \cite{emp1}. 
\end{rmk}
Our next aim is to prove the solvability of system~\eqref{eq:adtsys}.
To this end, it is convenient to define the constants $s_j=s_j(p)>0$, $j=1,2,\ldots,k-1$,
by setting:
\beq
\label{def:sj}
s_j:=\frac{\tau_{j+1}}{\tj}=\left(\frac{\de_j}{\de_{j+1}}\right)^{\frac{2}{p-1}},
\qquad j=1,2,\ldots,k-1,
%=\frac{\sigma_{j+1}}{\sigma_j}.
\eeq
where the last equality holds true in view of the fifth equation in system \eqref{eq:adtsys}.
With this notation, we may write
\beq
\label{def:ej}
\ej=\frac{\tau_{j+1}}{\tj+\tau_{j+1}}=\frac{\sj}{1+\sj},
\qquad j=1,2,\ldots,k-1.
\eeq
\begin{prop}
\label{prop:solvability}
For every $k\in\mathbb N$ there exists $p_0>0$ such that:
\begin{enumerate}
\item[(i)]
(Existence): For all $p>p_0$ there exists a solution
$(\aj,\dej,\tj)=(\aj(p),\dej(p),\tj(p))$, $j=1,2,\ldots,k$, to system~\eqref{eq:adtsys}; 
\item[(ii)]
(Basic properties): The solution $(\aj,\dej,\tj)$ obtained in (i) satisfies 
assumptions~(A1)--(A2)--(A3);
\item[(iii)](Properties of the $\sj$'s)
The $\sj$'s form a bounded increasing sequence: 
\beq
\label{eq:sjprop}
\frac{3}{2e+1}<s_1<\ldots<\sj<s_{j+1}<1, \qquad j=1,2,\ldots,k-1.
\eeq
\item[(iv)](Properties of the $\tj$'s)
\beqs
0<\tau_k<\tau_{k-1}<\ldots<\tau_{j+1}<\tj<\ldots<\tau_2<\tau_1.
\eeqs 
Moreover,
\beqs
\tj=\frac{e^{2\bj}}{p}\left(1+O(\frac{\ln p}{p})\right).
\eeqs
\item[(v)](Properties of the $\bj$'s)
\beqs
\bal
&b_k=\frac{1}{\ak+2+O(\frac{1}{p})}\\
&\bj-\bjj=\frac{1+\sj}{\aj+2+O(\frac{1}{p})}=\frac{1+\sj}{\sj(\ajj-2)}.
\eal
\eeqs 
\item[(vi)](Properties of the $\ej$'s)
\beqs 
\frac{1}{e+1}<\veps_1<\veps_2<\ldots<\ej<\ejj<\ldots<\veps_k<\frac{1}{2}.
\eeqs
\end{enumerate}
\end{prop}
Finally, we establish the following bounds for the $\aj$'s, which imply that $\aj\not\in\mathbb N$ for all
$j\ge2$, and thus yield invertibility of the linearized operator.
\begin{prop}[Bounds for the $\aj$'s]
\label{prop:aj}
The solution $(\aj,\dej,\tj)$, $j=1,2,\ldots,k$, obtained in Proposition~\ref{prop:solvability}--(i) satisfies
\beq
\label{eq:alphabounds}
8j-6<\aj<8j-5
\eeq 
for all $j\geq2$.
In particular,
$\alpha_i\not\in\mathbb  N$ for any $i\ge2$.
\end{prop}
The remaining part of this section is devoted to the proofs of Proposition~\ref{prop:Rjest}, Proposition~\ref{prop:solvability}
and Proposition~\ref{prop:aj}.
\par
The following lemma establishes the \lq\lq leading profile" of the rescaled approximate solution $\Up(\dej y)$ 
for $x=\dej y\in\Aj$:
\begin{lemma}
\label{lem:Upexp}
Suppose $\ti$, $\ai$, $\dei$, $i=1,2,\ldots,k$, satisfy system \eqref{eq:adtsys}.
Then, there holds the expansion:
\beq
\label{eq:Upexp}
\bal
\Up(x)
=&(-1)^{j-1}\tj p\Big\{1+\frac{\Vaj(y)}{p}+\frac{\woaj(y)}{p^2}+\frac{\wiaj(y)}{p^3}+O(\frac{e^{-\beta_jp}}{p})\Big\}
\eal
\eeq
uniformly for $x=\dej y\in\Aj$, for some $\beta_j>0$.
In particular, 
\beqs
\Up(x)=\sum_{j=1}^k(-1)^{j-1}\tj p\Big\{1+\frac{\Vaj(\frac{x}{\dej})}{p}+\frac{\wodaj(x)}{p^2}+\frac{\widaj(x)}{p^3}+O(\frac{e^{-\beta_jp}}{p})\Big\}\chi_{\Aj}(x),
\eeqs
uniformly for $x\in\Om$.
\end{lemma}
\begin{proof}
The main underlying reason for the following expansions is that
\[
 P\Udai(x)\approx-2\ln(\dei^{\ai}+|x|^{\ai}).
\]
More precisely, we recall the definitions of $\tvi$, $\tai$, $i=1,2,\ldots,k$, in \eqref{def:tildev}--\eqref{def:ta}.
In view of Lemma~\ref{lem:modbubproj}, we have
\[
P\tUdai(\dej y)=\left\{
\bal
&\tvj(y)+2\taj\ln\frac{1}{\dej}+4\pi\taj h(0)-\ln(2\al_j^2)\\
&\qquad\qquad\qquad+O(\dej^{\ej}\dejj^{1-\ej})+O(\de_j),
&&\hbox{for }i=j;\\
&2\tai\ln\frac{1}{|y|}+2\tai\ln\frac{1}{\dej}+4\pi\tai h(0)\\
&\qquad\qquad\qquad+O(\dej^{\ej}\dejj^{1-\ej})+O(\frac{\de_{j-1}}{\dej})^{(1-\veps_{j-1})\ai}+O(\de_i),
&&\hbox{for }1\le i<j;\\
&2\tai\ln\frac{1}{\dei}+4\pi\tai h(0)+\frac{\woai(0)}{p}+\frac{\wiai(0)}{p^2}\\
&\qquad\qquad\qquad+O(\dej^{\ej}\dejj^{1-\ej})+O(\frac{\dej}{\dejj})^{\ej\ai}+O(\de_i),
&&\hbox{for }j<i\le k,
\eal
\right.
\]
uniformly for $x=\dej y\in\Aj$.
We deduce that
\[
\bal
\Up(\dej y)=&\tj(-1)^{j-1}\tvj(y)
+2\left(\sum_{1\le i<j}\ti(-1)^{i-1}\tai\right)\,\ln\frac{1}{|y|}
+2\left(\sum_{1\le i\le j}\ti(-1)^{i-1}\tai\right)\,\ln\frac{1}{\dej}\\
&+2\sum_{j<i\le k}\ti(-1)^{i-1}\tai\,\ln\frac{1}{\dei}
+(-1)^{j-1}\tj\cj\\
&+O\left(\sum_{i=1}^k\ti\dej^{\ej}\dejj^{1-\ej}+\sum_{i<j}\ti(\frac{\de_{j-1}}{\dej})^{(1-\veps_{j-1})\ai}
+\sum_{i>j}\ti(\frac{\dej}{\dejj})^{\ej\ai}+\sum_{i=1}^k\ti\de_i\right),
\eal
\]
where the bounded constants $\cj$ are defined in \eqref{def:cj}.
Equivalently, we may write
\beq
\label{eq:Upexp}
\bal
\Up(\dej y)=&\tj(-1)^{j-1}\Big\{\tvj(y)+2\left(\sum_{1\le i<j}\frac{\ti}{\tj}(-1)^{i-j}\tai\right)\,\ln\frac{1}{|y|}
-2\left(\sum_{1\le i\le j}\frac{\ti}{\tj}(-1)^{i-j}\tai\right)\,\ln\dej\\
&-2\sum_{j<i\le k}\frac{\ti}{\tj}(-1)^{i-j}\tai\,\ln\dei
+\cj+\om_j\Big\}
\eal
\eeq
where
\[
\om_j=O\left(\sum_{i=1}^k\frac{\ti}{\tj}\dej^{\ej}\dejj^{1-\ej}+\sum_{i<j}\frac{\ti}{\tj}(\frac{\de_{j-1}}{\dej})^{(1-\veps_{j-1})\ai}
+\sum_{i>j}\frac{\ti}{\tj}(\frac{\dej}{\dejj})^{\ej\ai}+\sum_{i=1}^k\frac{\ti}{\tj}\de_i\right).
\]
Since $\ti/\tj=O(1)$ in view of (A3), we have
\beq
\label{est:omj}
\om_j=O(e^{-\beta_jp})
\qquad\hbox{uniformly for\ }y\in A_j,
\eeq
for some $\beta_j>0$.
\par
We observe that the second equation in \eqref{eq:adtsys} implies that
\[
-2\sum_{1\le i<j}\frac{\ti}{\tj}(-1)^{i-j}\tai=\aj-2
\]
and therefore
\beqs
-2\sum_{i=1}^j\frac{\ti}{\tj}(-1)^{i-j}\tai=-2\sum_{1\le i<j}\frac{\ti}{\tj}(-1)^{i-j}\tai-2\taj=-2\taj+\aj-2.
\eeqs
Consequently, we may rewrite \eqref{eq:Upexp}
in the form
\beqs
\bal
\Up(\dej y)=\tj(-1)^{j-1}\Big\{\tvj(y)+(\aj-2)\ln|y|-(2\taj-\aj+2)\ln\dej\\
-2\sum_{j<i\le k}\frac{\ti}{\tj}(-1)^{i-j}\tai\ln\dei+\cj+\om_j\Big\}.
\eal
\eeqs
Using the fourth equation in \eqref{eq:adtsys} and the definition of $\tvj$ in \eqref{def:tildev},
we derive 
\beqs
\bal
\Up(\dej y)=&\tj(-1)^{j-1}\Big\{v_{\aj}(y)+\frac{\woaj(y)}{p}+\frac{\wiaj(y)}{p^2}+(\aj-2)\ln|y|+p+\om_j\Big\}.
%=&\tj(-1)^{j-1}p\Big\{1+\frac{v_{\aj}(y)+(\aj-2)\ln|y|}{p}+\frac{\woaj(y)}{p^2}+\frac{\wiaj(y)}{p^3}+O(\frac{e^{-\beta_jp}}{p})\Big\},
\eal
\eeqs
Now \eqref{est:omj} yields the asserted expansion.
\end{proof}
%In what follows it will be useful to set, for all $\al\ge2$,
%\beq
%
%\Va(y)=v_\al(y)+(\al-2)\ln|y|.
%\eeq
%Then, the function $\mathcal V_{\aj}$, $j=1,2,\ldots,k$ is the \lq\lq effective solution profile"
%observed in $\Aj$.
%\par
Lemma~\ref{lem:Upexp} and the facts $|w_{\aj}^\ell(y)|\le C\ln(|y|+2)=O(p)$ in $\Aj$, readily implies the following lower-order expansions, 
which will be also used in the sequel:
\begin{align}
\label{eq:Upexpfirst}
\Up(x)
=&(-1)^{j-1}\tj p\Big\{1+\frac{\Vaj(y)}{p}+\frac{\woaj(y)}{p^2}+\frac{O(\ln(|y|+2))}{p^3}\Big\}\\
\label{eq:Upexpzero}
\Up(x)
=&(-1)^{j-1}\tj p\Big\{1+\frac{\Vaj(y)+O(1)}{p}\Big\},
\end{align}
uniformly for $x=\dej y\in\Aj$.
Moreover, as a direct consequence of Lemma~\ref{lem:Upexp} and the Taylor expansions, as stated in Lemma~\ref{lem:Taylor}, 
we obtain the following expansions.
\begin{lemma}
\label{lem:gpgpp}
The following expansions hold true:
\beq
\tag{i}
\bal
\gp&(\Up(x))=(-1)^{j-1}\tj|x|^{\aj-2}e^{\Udaj(x)}\times\\
&\times\Big\{1+\frac{1}{p}\Big[\woaj(y)-\vphio(\Vaj(y))\Big]+\frac{1}{p^2}\Big[\wiaj(y)-\vphii(\Vaj(y),\woaj(y))\Big]\\
&\qquad\qquad\qquad+\frac{O(|\ln|y|^{\aj-2}|^6+\ln^6(2+|y|))}{p^3}\Big\}
\eal
\eeq
and
\beq
\tag{ii}
\bal
\gpp(\Up(x))=&|x|^{\aj-2}e^{\Udaj(x)}\times\\
&\times\left\{1+\frac{1}{p}\Big[\woaj(y)-\vphio(\Vaj(y))-\Vaj(y)\Big]+\frac{O(|\ln|y|^{\aj-2}|^4+\ln^4(2+|y|))}{p^2}\right\},
\eal
\eeq
uniformly for $x=\dej y\in\Bj$, $j=1,2,\ldots,k$.
\end{lemma}
\begin{proof}
Proof of (i).
In view of Lemma~\ref{lem:Upexp}, we have
\beqs
\bal
\gp(\Up(x))=&\gp(\Up(\dej y))\\
=&\gp\Big((-1)^{j-1}\tj p\Big\{1+\frac{\Vaj(y)}{p}+\frac{\woaj(y)}{p^2}+\frac{\wiaj(y)}{p^3}+O(\frac{e^{-\beta_jp}}{p})\Big\}\Big).
\eal
\eeqs
For $\dej y\in\Bj$ we have, by definition
\beqs
1+\frac{\Vaj(y)}{p}\ge\frac{1}{2}.
\eeqs
Therefore, for $\dej y\in\Bj$ and $p\gg1$ we have
\beqs
1+\frac{\Vaj(y)}{p}+\frac{\woaj(y)}{p^2}+\frac{\wiaj(y)}{p^3}+O(\frac{e^{-\beta_jp}}{p})\ge\frac{1}{4}.
\eeqs
Hence, we may write
\beqs
\gp(\Up(x))=(-1)^{j-1}(\tj p)^p\Big(1+\frac{\Vaj(y)}{p}+\frac{\woaj(y)}{p^2}+\frac{\wiaj(y)}{p^3}+O(\frac{e^{-\beta_jp}}{p})\Big)^p.
\eeqs
Now, the claim follows by the Taylor expansions, as stated in Lemma~\ref{lem:Taylor}--(i) with $t=|y|$,
$a(t)=\Vaj(y)$, $b(t)=\woaj(y)$, $c(t)=\wiaj(y)$ and the fact
$(\tj p)^p=\tj\dej^{-2}$.
Indeed, for $x=\dej y\in\Bj$ we derive
\beqs
\bal
\gp(\Up(x))=&(-1)^{j-1}\tj\dej^{-2} e^{\Vaj(y)}\Big\{1+\frac{1}{p}\left(\woaj(y)-\vphio(\Vaj(y))\right)\\
&+\frac{1}{p^2}\left(\wiaj(y)-\vphii(\Vaj(y),\woaj(y))\right)
+\frac{O(|\Vaj(y)|^6+1)}{p^4}\Big\},
\eal
\eeqs
which, recalling the definition of $\Vaj$, yields (i).
\par
Proof of (ii). Similarly as above, in view of \eqref{eq:Upexpfirst}, for $x=\dej y\in\Bj$ we have
\beqs
\bal
\gpp(\Up(x))=&p(\tj p)^{p-1}\Big\{1+\frac{\Vaj(y)}{p}+\frac{\woaj(y)}{p^2}+\frac{O(\ln(|y|+2))}{p^3}\Big\}^{p-1}\\
=&\frac{e^{\Vaj(y)}}{\dej^2}\Big\{1+\frac{1}{p}\Big[\woaj(y)-\vphio(\Vaj(y))-\Vaj(y)\Big]+\frac{O(|\Vaj(y)|^4+1)}{p^2}\Big\},
\eal
\eeqs
where we used Lemma~\ref{lem:Taylor}--(ii) with $\kappa=1$ to derive the last equality.
Now (ii) follows by the mass scaling property as stated in Lemma~\ref{lem:massscaling}, and recalling the definition of $\Vaj$.
\end{proof}
\begin{proof}[Proof of Proposition~\ref{prop:Rjest}]
From \eqref{eq:Uda} we derive
\[
-\Delta P\tUdaj(\dej y)
=\frac{|y|^{\aj-2}}{\dej^2}e^{\vj(y)}
\left(1+\frac{\woaj(y)-\vphio(\Vaj(y))}{p}+\frac{\wiaj(y)-\vphii(\Vaj(y),\woaj(y))}{p^2}\right).
\]
Therefore, we may write, for $x=\dej y\in\Bj$: 
\beq
\label{eq:Rj}
\bal
&(-1)^{j-1}\tj\Delta\tUdaj+\gp(\Up(x))\\
&\quad=\tj(-1)^{j}
\frac{|y|^{\aj-2}}{\dej^2}e^{\vj(y)}
\left(1+\frac{\woaj(y)-\vphio(\Vaj(y))}{p}+\frac{\wiaj(y)-\vphii(\Vaj(y),\woaj(y))}{p^2}\right)\\
&\qquad\qquad+\gp(\Up(\dej y)).
\eal
\eeq
We may now apply Lemma~\ref{lem:gpgpp}--(i) to derive
\[
(-1)^{j-1}\tj\Delta\tUdaj+\gp(\Up(x))=\tj\frac{|y|^{\aj-2}}{\dej^2}e^{v_{\aj}(y)}\frac{O(|\Vaj(y)|^6+1)}{p^3},
\]
which yields the asserted estimate since $\tj=O(p^{-1})$.
\end{proof}
We turn to the proof of Proposition~\ref{prop:solvability}.
\begin{lemma}
\label{lem:assys}
System~\eqref{eq:adtsys} implies the following system in terms of $\aj$, $j=1,2,\ldots,k$ and $\sj$,
$j=1,2,\ldots,k-1$:
\beq
\label{eq:assys}
\left\{
\bal
&\al_1=2\\
&\frac{1}{2}\left(\aj+2-\frac{\Coaj}{p}-\frac{\Ciaj}{p^2}\right)\ln\sj+1+\sj=\frac{\cj+\sj c_{j+1}-(1+\sj)}{p-1},&&j=1,2,\ldots,k-1\\
&\sj(\al_{j+1}-2)=\aj+2-\frac{\Coaj}{p}-\frac{\Ciaj}{p^2},&&j=1,2,\ldots,k-1.
\eal
\right.
\eeq
\end{lemma}
\begin{proof}
We combine the first and the second equation in \eqref{eq:adtsys} into a single formula:
\begin{equation}
\label{s1}
-2\sum\limits_{i=1}^{j-1 }(-1)^{i-j}{\ti}\tai=(\al_j-2)\tj,\quad j=1,\dots,k,
\end{equation}
where we agree that the sum is zero if $j=1$.
Writing \eqref{s1} for the index $j$ and for the index $j+1$, and adding the two resulting expressions, we obtain
the following recursive formula for $\al_j$ (in terms of the $\ti$'s):
\beq 
\label{eq:alpharec}
(2\taj-\aj+2)\tj=(\al_{j+1}-2)\tau_{j+1},\quad j=1,\dots,k-1,
\eeq
which yields the third equation in \eqref{eq:assys}.
Similarly, we combine into a single formula the third and the fourth equation in \eqref{eq:adtsys}:
\begin{equation}
\label{s3}
-(2\taj-\aj+2)\tj\ln\de_j
-2\sum\limits_{i=j+1}^{k}(-1)^{i-j}{\ti }\tai\ln\de_i 
+c_j\tj=p\tj,\quad j=1,\dots,k,
\end{equation}
where we agree that the sum is zero if $j=k$.
Writing \eqref{s3} for the index $j$ and for the index $j+1$, and adding the two resulting expressions,
we obtain
\beqs
\label{eq:deltarec1}
-(2\taj-\al_j+2)\tj\ln\dej+(\al_{j+1}-2)\tau_{j+1}\ln\de_{j+1}+c_j\tj+c_{j+1}\tau_{j+1}
=p(\tj+\tau_{j+1}),\quad j=1,\dots,k-1,
\eeqs
i.e., 
\[
(2\taj-\aj+2)\tj\ln\de_j=(\al_{j+1}-2)\tau_{j+1}\ln\de_{j+1}-p(\tj+\tau_{j+1})
+\tj\cj+\tau_{j+1}c_{j+1}.
\]
Using \eqref{eq:alpharec},
we obtain the following recursive formula for $\de_j$ (in terms of the $\sj$'s):
\beq
\label{eq:deltarec2}
\ln\de_j=\ln\de_{j+1}-\frac{1+s_j}{2\taj-\aj+2}\,p
+\frac{\cj+c_{j+1}s_j}{2\taj-\aj+2}.
\eeq
From \eqref{def:sj} we derive $\dej/\de_{j+1}=\sj^{(p-1)/2}$, and inserting into \eqref{eq:deltarec2}
we deduce 
\beqs
(2\taj-\aj+2)\ln\sj+(1+\sj)\frac{2p}{p-1}-\frac{2}{p-1}(\cj+\sj c_{j+1})=0,
\eeqs
and the second equation in \eqref{eq:assys} follows.
\end{proof}
\begin{rmk}
We shall use the last equation in \eqref{eq:assys} in the simplified form:
\beq
\label{eq:asuseful}
\sj(\ajj-2)=\aj+2+O(\frac{1}{p}).
\eeq
\end{rmk}
Now we can conclude the proof of Proposition~\ref{prop:solvability}.
\begin{proof}[Proof of Proposition~\ref{prop:solvability}--(i): Existence]
We first observe that
\beq
\label{eq:fractau}
\frac{\ti}{\tj}=
\left\{
\bal
&\frac{\tau_{j+1}}{\tj}\frac{\tau_{j+2}}{\tau_{j+1}}\cdots\frac{\ti}{\tau_{i-1}}
=\sj s_{j+1}\cdots s_{i-1},&&\hbox{if }i>j\\
&1,&&\hbox{if }i=j\\
&\left(\frac{\tj}{\ti}\right)^{-1}=\frac{1}{s_is_{i+1}\cdots s_{j-1}},&&\hbox{if }i<j.
\eal
\right.
\eeq
In particular, the $\cj$'s, as defined in \eqref{def:cj}, are continuously differentiable 
with respect to $(s_1,\ldots,s_{k-1})\in(0,1]^{k-1}$.
We set $t:=p^{-1}$ and we rearrange the parameters $\aj$, $\sj$ in the form
$((\aj,\sj)_{j=1,\ldots,k-1},\ak)$.
For $i=1,\ldots,k$ we set
\beq
\label{def:gast}
g_i((\aj,\sj)_{j=1,\ldots,k-1},\ak,t):=
\left\{
\bal
&\al_1-2,&&\hbox{if }i=1\\
&(\ai-2)s_{i-1}-(2\al^{({w})}_{i-1}-\al_{i-1}+2),&&\hbox{if }2\le i\le k
\eal
\right.
\eeq
and for $i=1,\ldots,k-1$ we set
\beq
\label{def:hast}
h_i((\aj,\sj)_{j=1,\ldots,k-1},\ak,t)
:=\frac{2\tai-\ai+2}{2}\ln s_i+1+s_i
-\frac{c_i+s_ic_{i+1}-(1+s_i)}{1-t}t.
\eeq
With this notation, system~\eqref{eq:assys} is equivalent to the equation
$\G((\aj,\sj)_{j=1,\ldots,k-1},\ak,t)=\mathbf{0}\in\rr^{2k-1}$, where
\beq
\label{def:G}
\bal
\G((\aj,\sj)_{j=1,\ldots,k-1},\ak,t)
=&\left(
\begin{array}{c}
g_1((\aj,\sj)_{j=1,\ldots,k-1},\ak,t)\\
h_1((\aj,\sj)_{j=1,\ldots,k-1},\ak,t)\\
g_2((\aj,\sj)_{j=1,\ldots,k-1},\ak,t)\\
h_2((\aj,\sj)_{j=1,\ldots,k-1},\ak,t)\\
\vdots\\
g_{k-1}((\aj,\sj)_{j=1,\ldots,k-1},\ak,t)\\
h_{k-1}((\aj,\sj)_{j=1,\ldots,k-1},\ak,t)\\
g_k((\aj,\sj)_{j=1,\ldots,k-1},\ak,t)
\end{array}
\right).
\eal
\eeq
We seek a branch of solutions $((\aj(t),\sj(t))_{j=1,\ldots,k-1},\ak(t),t)$ to equation~\eqref{def:G}, for small values of $t$,
by an implicit function argument.
We first consider the case $t=0$, namely we consider the equation 
\beq
\label{def:G0}
\G((\aj,\sj)_{j=1,\ldots,k-1},\ak,0)=\mathbf{0},
\eeq
corresponding to
the \lq\lq unperturbed system"
\beq
\label{eq:assysunpert}
\left\{
\bal
&\al_1=2\\
&\frac{\aj+2}{2}\ln\sj+1+\sj=0,&&j=1,2,\ldots,k-1\\
&(\al_{j+1}-2)\sj=\al_j+2,&&j=1,2,\ldots,k-1.
\eal
\right.
\eeq
We note that for any $k\in\mathbb N$,
equation\eqref{def:G0} admits a unique solution, denoted
$((\aj^0,\sj^0)_{j=1,\ldots,k-1},\ak^0)$.
Indeed, let us denote 
\beqs
\bal
g_i^0((\aj,\sj)_{j=1,\ldots,k-1},\ak):=&g_i((\aj,\sj)_{j=1,\ldots,k-1},\ak,0)\\
=&\begin{cases}
\al_1-2&\hbox{if }i=1\\
s_{i-1}(\ai-2)-(\al_{i-1}+2)&\hbox{if }2\le i\le k,
\end{cases}
\eal
\eeqs
for $i=1,2,\ldots,k$, and
\beqs
h_i^0((\aj,\sj)_{j=1,\ldots,k-1},\ak):=h_i((\aj,\sj)_{j=1,\ldots,k-1},\ak,0)=\frac{\ai+2}{2}\ln s_i+1+s_i,
\eeqs
for $i=1,2,\ldots,k-1$.
Then, we are reduced to solving the system
\beq
\label{sys:G0}
\bal
\G((\aj,&\sj)_{j=1,\ldots,k-1},\ak,0)\\
=&\left(
\begin{array}{c}
g_1^0((\aj,\sj)_{j=1,\ldots,k-1},\ak)\\
h_1^0((\aj,\sj)_{j=1,\ldots,k-1},\ak)\\
g_2^0((\aj,\sj)_{j=1,\ldots,k-1},\ak)\\
h_2^0((\aj,\sj)_{j=1,\ldots,k-1},\ak)\\
\vdots\\
g_{k-1}^0((\aj,\sj)_{j=1,\ldots,k-1},\ak)\\
h_{k-1}^0((\aj,\sj)_{j=1,\ldots,k-1},\ak)\\
g_k^0((\aj,\sj)_{j=1,\ldots,k-1},\ak)
\end{array}
\right)
=\left(
\begin{array}{c}
\al_1-2\\
\frac{\al_1+2}{2}\ln s_1+1+s_1\\
(\al_2-2)s_1-(\al_1+2)\\
\frac{\al_2+2}{2}\ln s_2+1+s_2\\
\vdots\\
(\al_{k-1}-2)s_{k-2}-(\al_{k-2}+2)\\
\frac{\al_{k-1}+2}{2}\ln s_{k-1}+1+s_{k-1}\\
(\ak-2)s_{k-1}-(\al_{k-1}+2)
\end{array}
\right)=\mathbf{0}.
\eal
\eeq
Setting $\zeta_\al(s):=\frac{\al+2}{2}\ln s+1+s$, it is elementary to check that for any 
fixed $\al\ge2$ there holds $\zeta_\al'(s)=(\al+2)/(2s)+1>0$ for all $s>0$,
$\lim_{s\to0^+}\zeta_\al(s)=-\infty$ and $\zeta_\al(1)=2$.
Therefore, for any $\al\ge2$ the nonlinear equation $\zeta_\al(s)=0$
admits a unique solution $s_\al\in(0,1)$.
Since we may write $h_i^0((\aj,\sj)_{j=1,\ldots,k-1},\ak)=\zeta_{\ai}(s_i)$, 
we deduce that system~\eqref{sys:G0} admits a unique solution $((\aj^0,\sj^0)_{j=1,\ldots,k-1},\ak^0)$ defined recursively.
\par
We claim that any solution $((\aj,\sj)_{j=1,\ldots,k-1},\ak)$ to \eqref{eq:assysunpert} satisfies
\beq
\label{eq:ajprop}
\al_{j+1}>\aj+4>0.
\eeq
Indeed, from the third equation in \eqref{eq:assysunpert} we have
$\al_{j+1}-2=(\aj+2)/\sj$. Since $\sj\in(0,1)$, we deduce that if
$\aj>0$ then $\al_{j+1}-2>\aj+2$, i.e., $\al_{j+1}>\aj+4>0$.
Since $\al_1=2$, we obtain \eqref{eq:ajprop} recursively. 
\par
We now check that the $(2k-1)\times(2k-1)$ Jacobian matrix of $\G$ with respect to
the variables $((\aj,\sj)_{j=1,\ldots,k-1},\ak)$ is invertible
at the solution $((\aj^0,\sj^0)_{j=1,\ldots,k-1},\ak^0)$.
To this end, it is readily checked that
\beqs
\bal
&\frac{\pl g_i^0}{\pl\aj}=
\begin{cases}
\de_{ij},&\hbox{if }i=1, \\
s_{i-1}^0\de_{ij}-\de_{i-1,j},&\hbox{if }2\le i\le k;
\end{cases},\ j=1,2,\ldots,k;\\
&\frac{\pl g_i^0}{\pl\sj}=
\begin{cases}
0,&\hbox{if }i=1\\
(\ai^0-2)\de_{i-1,j}&\hbox{if }2\le i\le k-1
\end{cases},\ j=1,2,\ldots,k;\\
&\frac{\pl h_i^0}{\pl\aj}=\frac{\ln s_i^0}{2}\de_{ij},\qquad i=1,2,\ldots,k-1,\ j=1,2,\ldots,k;\\
&\frac{\pl h_i^0}{\pl\sj}=\left(\frac{\ai^0+2}{2s_i^0}+1\right)\de_{ij}\qquad i=1,2,\ldots,k-1,\ \qquad j=1,2,\ldots,k-1,
\eal
\eeqs
where the $\de_{ij}$'s denote Kronecker deltas.
Consequently, the Jacobian matrix for the mapping $\mathbf{G}((\aj,\sj)_{j=1,\ldots,k-1},\ak,0)$,
is given by
\beqs
\bal
&D_{((\aj,\sj)_{j=1,\ldots,k-1},\ak)}\mathbf{G}((\aj,\sj)_{j=1,\ldots,k-1},\ak,0)=\\
&=\left(\begin{array}{ccccccccccc}
\frac{\pl g_1^0}{\pl\al_1}&\frac{\pl g_1^0}{\pl s_1}&\frac{\pl g_1^0}{\pl\al_2}&\frac{\pl g_1^0}{\pl s_2}
&\cdots&\frac{\pl g_1^0}{\pl\al_{k-2}}&\frac{\pl g_1^0}{\pl s_{k-2}}
&\frac{\pl g_1^0}{\pl\al_{k-1}}&\frac{\pl g_1^0}{\pl s_{k-1}}&\frac{\pl g_1^0}{\pl\al_k}\\
\frac{\pl h_1^0}{\pl\al_1}&\frac{\pl h_1^0}{\pl s_1}&\frac{\pl h_1^0}{\pl\al_2}&\frac{\pl h_1^0}{\pl s_2}
&\cdots&\frac{\pl h_1^0}{\pl\al_{k-2}}&\frac{\pl h_1^0}{\pl s_{k-2}}
&\frac{\pl h_1^0}{\pl\al_{k-1}}&\frac{\pl h_1^0}{\pl s_{k-1}}&\frac{\pl h_1^0}{\pl\al_k}\\
\frac{\pl g_2^0}{\pl\al_1}&\frac{\pl g_2^0}{\pl s_1}&\frac{\pl g_2^0}{\pl\al_2}&\frac{\pl g_2^0}{\pl s_2}
&\cdots&\frac{\pl g_2^0}{\pl\al_{k-2}}&\frac{\pl g_2^0}{\pl s_{k-2}}
&\frac{\pl g_2^0}{\pl\al_{k-1}}&\frac{\pl g_2^0}{\pl s_{k-1}}&\frac{\pl g_2^0}{\pl\al_k}\\
\frac{\pl h_2^0}{\pl\al_1}&\frac{\pl h_2^0}{\pl s_1}&\frac{\pl h_2^0}{\pl\al_2}&\frac{\pl h_2^0}{\pl s_2}
&\cdots&\frac{\pl g_2^0}{\pl\al_{k-2}}&\frac{\pl g_2^0}{\pl s_{k-2}}
&\frac{\pl h_2^0}{\pl\al_{k-1}}&\frac{\pl h_2^0}{\pl s_{k-1}}&\frac{\pl h_2^0}{\pl\al_k}\\
\vdots&\vdots&\vdots&\vdots&\vdots&\vdots&\vdots&\vdots&\vdots&\vdots\\
\frac{\pl g_{k-1}^0}{\pl\al_1}&\frac{\pl g_{k-1}^0}{\pl s_1}&\frac{\pl g_{k-1}^0}{\pl\al_2}&\frac{\pl g_{k-1}^0}{\pl s_2}
&\cdots&\frac{\pl g_{k-1}^0}{\pl\al_{k-2}}&\frac{\pl g_{k-1}^0}{\pl s_{k-2}}
&\frac{\pl g_{k-1}^0}{\pl\al_{k-1}}&\frac{\pl g_{k-1}^0}{\pl s_{k-1}}&\frac{\pl g_{k-1}^0}{\pl\al_k}\\
\frac{\pl h_{k-1}^0}{\pl\al_1}&\frac{\pl h_{k-1}^0}{\pl s_1}&\frac{\pl h_{k-1}^0}{\pl\al_2}&\frac{\pl h_{k-1}^0}{\pl s_2}
&\cdots&\frac{\pl h_{k-1}^0}{\pl\al_{k-2}}&\frac{\pl h_{k-1}^0}{\pl s_{k-2}}
&\frac{\pl h_{k-1}^0}{\pl\al_{k-1}}&\frac{\pl h_{k-1}^0}{\pl s_{k-1}}&\frac{\pl h_{k-1}^0}{\pl\al_k}\\
\frac{\pl g_k^0}{\pl\al_1}&\frac{\pl g_k^0}{\pl s_1}&\frac{\pl g_k^0}{\pl\al_2}&\frac{\pl g_k^0}{\pl s_2}
&\cdots&\frac{\pl g_k^0}{\pl\al_{k-2}}&\frac{\pl g_k^0}{\pl s_{k-2}}
&\frac{\pl g_k^0}{\pl\al_{k-1}}&\frac{\pl g_k^0}{\pl s_{k-1}}&\frac{\pl g_k^0}{\pl\al_k}
\end{array}\right)\\
&=\left(\begin{array}{ccccccccccc}
1&0&0&0&\cdots&0&0&0&0&0\\
\frac{\ln s_1^0}{2}&\frac{\al_1^0+2}{2s_1^0}+1&0&0&\cdots&0&0&0&0&0\\
-1&\al_2^0-2&s_1^0&0&\cdots&0&0&0&0&0\\
0&0&\frac{\ln s_2^0}{2}&\frac{\al_2^0+2}{2s_2^0}+1&\cdots&0&0&0&0&0\\
\vdots&\vdots&\vdots&\vdots&\vdots&\vdots&\vdots&\vdots&\vdots&\vdots\\
0&0&0&0&\cdots&-1&\al_{k-1}^0-2&s_{k-2}^0&0&0\\
0&0&0&0&\cdots&0&0&\frac{\ln s_{k-1}^0}{2}&\frac{\al_{k-1}^0+2}{2}+1&0\\
0&0&0&0&\cdots&0&0&-1&\al_{k-2}^0&s_{k-1}^0
\end{array}
\right).
\eal
\eeqs
In particular, it is a lower triangular matrix with positive diagonal entries given by
\beqs
1,\ \frac{\al_1^0+2}{2s_1^0}+1,\  s_1^0,\quad\frac{\al_2^0+2}{2s_2^0}+1,\
\ldots,\ \frac{\al_{k-2}^0+2}{2s_{k-2}^0}+1,\ s_{k-2}^0,\ \frac{\al_{k-1}^0+2}{2s_{k-1}^0}+1,\quad s_{k-1}^0.
\eeqs
It follows that $D_{((\aj,\sj)_{j=1,\ldots,k-1},\ak)}\mathbf{G}((\aj,\sj)_{j=1,\ldots,k-1},\ak,0)$
is invertible at $((\aj^0,\sj^0)_{j=1,\ldots,k-1},\ak^0)$. 
\par
Now, the implicit function theorem yields a unique branch of solutions 
\beqs
((\aj(p),\sj(p))_{j=1,\ldots,k-1},\ak(p),p),\qquad p=t^{-1},
\eeqs
to system~\eqref{eq:assys}, for all sufficiently large values of $p$, continuously depending on $p$.
\par
We are left to derive a branch of solutions to system~\eqref{eq:adtsys} from the branch of solutions to
system~\eqref{eq:assys}.
To this end, we note that in view of \eqref{eq:fractau} the solution $((\aj,\sj)_{j=1,\ldots,k-1},\ak)$
uniquely determines the constants $\cj$, $j=1,2,\ldots,k$, defined in \eqref{def:cj}.
Now, we are able to compute the $\de_j$'s.
Indeed, the from the third equation in \eqref{eq:adtsys} we derive
\beqs
\ln\de_k=\frac{c_k-p}{2\ta_k-\al_k+2}.
\eeqs
From $\de_k$ and the last equation in system~\eqref{eq:adtsys} we obtain 
\beqs
\tau_k=\frac{1}{p^{p/(p-1)}\de_k^{2/(p-1)}}.
\eeqs
Recursively, we derive $\tj$, $j=1,2,\ldots,k-1$, using the property
\beqs
\tj=\frac{\tau_{j+1}}{\sj}.
\eeqs
Finally, from the fourth equation in \eqref{eq:adtsys} we recursively derive $\de_1,\ldots,\de_{k-1}$.
The existence of the desired branch of solutions $(\aj,\dej,\tj)$ is now completely established.
\end{proof}
\begin{proof}[Proof of Proposition~\ref{prop:solvability}-(ii): Basic properties]
By continuity, it suffices to check properties~(A1)--(A2)--(A3) for solutions
to the unperturbed system~\eqref{eq:assys}. Hence, let $((\aj,\sj)_{j=1,2,\ldots,k-1},\ak)$
be a solution to system~\eqref{eq:assys}.
Since $s_j\in(0,1)$, and since  \eqref{eq:ajprop} implies $\al_{j+1}-2>0$, we deduce from 
system~\eqref{eq:assys} that
\[
\aj+2=\sj(\al_{j+1}-2)<\al_{j+1}-2,
\]
for all $j\ge1$. In particular, $\aj+4<\al_{j+1}$, for all $j\ge1$,
and therefore assumption~(A1) is satisfied.
\par
We note that (A2) holds true for $j=k$ with 
\beqs
\bal
&\ln C_k=\frac{c_k}{2\ta_k-\al_k+2},
&&b_k=\frac{1}{2\ta_k-\al_k+2}.
\eal
\eeqs
From \eqref{eq:deltarec2} we derive, recursively:
\beq
\label{eq:deltarec}
\bal
&\ln C_j=\ln C_{j+1}+\frac{c_j+c_{j+1}s_j}{2\ta_j-\aj+2}\\
&b_j=b_{j+1}+\frac{1+\sj}{2\ta_j-\aj+2}=b_{j+1}+\frac{1+\sj}{\sj(\al_{j+1}-2)},\
j=1,2,\ldots,k-1,
\eal
\eeq
and consequently:
\beqs
\bal
&b_j=\frac{1}{2\ta_k-\al_k+2}+\sum_{i=j}^{k-1}\frac{1+s_i}{2\ta_i-\al_i+2}\\
&\ln C_j=\frac{c_k}{2\ta_k-\al_k+2}+\sum_{i=j}^{k-1}\frac{c_i+c_{i+1}s_i}{2\ta_i-\al_i+2},
\qquad j=1,2,\ldots,k-1.
\eal
\eeqs
%
%\beqs
%\bal
%b_j=&b_{j+1}+\frac{1+\sj}{2\ta_j-\aj+2}=b_{j+1}+\frac{1+\sj}{\sj(\al_{j+1}-2)},\\
%\ln C_j=&\ln C_{j+1}+\frac{c_j+c_{j+1}\sj}{2\ta_j-\aj+2},
%\qquad j=1,2,\ldots,k-1.
%\eal
%\eeqs
Hence, assumption~(A2) is completely verified.
At this point, it is clear that assumption~(A3) is also satisfied.
\end{proof}
\begin{rmk}
We note that $b_1\to+\infty$ as $k\to+\infty$, that is, the concentration rate of the fast peaks 
increases as the number of peaks increases. The rate of the slowest peak does not change.
\end{rmk}
%\begin{rmk}
%For later use, we also point out that 
%\beq
%\label{eq:sjest}
%1>s_{j+1}>s_j>s_1>\frac{6}{4e+2}, \qquad j=1,2,\ldots,k-1.
%\eeq
%\edz{lower bound for $\sj$ added}
%\end{rmk}
\begin{proof}[Proof of Proposition~\ref{prop:solvability}-(iii): Properties of the $\sj$'s]
By continuity, it suffices to verify \eqref{eq:sjprop} for the solution
$((\aj^0,\sj^0)_{j=1,\ldots,k-1},\ak^0)$ to the unperturbed system~\eqref{sys:G0}.
To this end, we observe that
\beqs
h_j^0((\aj^0,\sj^0)_{j=1,\ldots,k-1},\ak^0)=\zeta_{\aj^0}(\sj^0)=\frac{\aj^0+2}{2}\ln\sj^0+1+\sj^0=0
\eeqs
and
\beqs
\bal
h_{j+1}^0((\aj^0,\sj^0)_{j=1,\ldots,k-1},\ak^0)=&\zeta_{\al_{j+1}^0}(\sj^0)\\
=&\frac{\aj^0+2}{2}\ln\sj^0+1+\sj^0+\frac{\aj^0+2}{2}\left(\frac{\al_{j+1}^0+2}{\aj^0+2}-1\right)\ln\sj^0\\
=&\frac{\aj^0+2}{2}\left(\frac{\al_{j+1}^0+2}{\aj^0+2}-1\right)\ln\sj^0<0.
\eal
\eeqs
Since $\zeta_{\al_{j+1}^0}$ is strictly increasing and since $\zeta_{\al_{j+1}^0}(s_{j+1}^0)=0$,
we deduce that $0<\sj^0<s_{j+1}^0<1$, for any $j=1,2,\ldots,k-1$.
We claim that
\beq
\label{eq:sjconvexitybound}
\frac{\aj+4}{(\aj+2)e^{\frac{4}{\aj+2}}+2}\le\sj\le e^{-\frac{2}{\aj+2}}.
\eeq
Indeed, from the second equation in \eqref{eq:assysunpert} we obtain the nonlinear equation
\beq
\label{eq:sjnonlin}
\sj=e^{-\frac{2(1+\sj)}{\aj+2}}, 
\eeq
which readily implies the upper bound for
$\sj$ since $\sj\in(0,1)$. By convexity of the
function $f(s)=e^{-\frac{2(1+s)}{\aj+2}}$ at $s=1$, it is elementary to
check that
\beqs
f(s)\ge e^{-\frac{4}{\aj+2}}-\frac{2}{\aj+2}e^{-\frac{4}{\aj+2}}(s-1)
=e^{-\frac{4}{\aj+2}}(1+\frac{2}{\aj+2})-\frac{2}{\aj+2}e^{-\frac{4}{\aj+2}}s.
\eeqs
Using again $\sj\in(0,1)$ and $\sj=f(\sj)$, we derive the lower bound in  \eqref{eq:sjconvexitybound}.
For $j=1$, we obtain the lower bound in \eqref{eq:sjprop}.
\end{proof}
We are left to establish \eqref{eq:alphabounds}.
It is readily seen that the $\aj$'s increase asymptotically linearly with respect to
$j$, more precisely $\aj=8j+O(1)$ as $j\to+\infty$.
To see this, we note that from \eqref{eq:sjnonlin} we derive
$\sj=1-\frac{2(1+\sj)}{\aj+2}+O(\aj^{-2})$ and therefore
\beqs
\sj=\frac{\aj}{\aj+4}+O(\frac{1}{\aj^2}).
\eeqs
Inserting into the third equation of \eqref{eq:assysunpert} we deduce that
\beqs
\bal
\al_{j+1}=\frac{\aj+2}{\sj}+2=\aj+8+O(\frac{1}{\aj})=8(j+1)+O(1).
\eal
\eeqs
\begin{proof}[Proof of Proposition~\ref{prop:aj}]
Let us introduce the Lambert function $W=f^{-1}$, where $f(x)=xe^x$ in $\mathbb R^+$. 
We remark that
$$ax\ln x=x+1
\ \Leftrightarrow\   {1\over ax}e^{1\over ax}={1\over a}e^{-{1\over a}}
\ \Leftrightarrow\   {1\over ax}=W\left({1\over a}e^{-{1\over a}}\right)\ \Leftrightarrow\   x=\frac1{a}\frac 1{W\left({1\over a}e^{-{1\over a}}\right)}$$
Therefore, by recurrence
$${1\over\sj}={2\over \alpha_j+2}{1\over W\({2\over \alpha_j+2}e^{-{2\over \alpha_j+2}}\)}$$
and so
\begin{equation}
\label{ai}
\alpha_1=2\ \hbox{and}\ \alpha_{j+1}=2+{2\over W\({2\over \alpha_j+2}e^{-{2\over \alpha_j+2}}\)},\ j=1,\dots,k-1.
\end{equation}
%We claim that
%\begin{equation}
%\label{conj}
%\alpha_i\not\in  \mathbb  N\ \hbox{for any}\ i\ge2.
%\end{equation}
%Indeed, we will prove that $\alpha_i$ lies between two consecutive integer numbers, namely:
%\begin{equation}
%\label{ai1}
%8i-6<\alpha_i<8i-5\ \hbox{for any}\ i\ge2.
%\end{equation}
%
First of all,  straightforward but tedious computations show that
$$0<\frac1{W\(xe^{-x}\)}-2-\frac 1x\le \frac43x^2\ \hbox{for any}\ x\in(0,1/2].$$
Indeed, using the definition of Lambert's function, it is equivalent to prove that
$$\min\limits_{x\in(0,1/2]}\frac1{2x+1}e^{{x\over 2x+1}+x}=\max \limits_{x\in(0,1/2]}\frac3{4x^3+6x+3}e^{{3x\over 4x^3+6x+3}+x}=1.$$
That implies (setting $x=\frac2{\alpha_i+2}$)
$$\alpha_i+8<\alpha_{i+1}\le \alpha_i+8+\frac{32}3{1\over(\alpha_i+2)^2}\ \hbox{for any}\ i\ge1.$$
Therefore, by recurrence we get
$$8i-6<\alpha_i\le 8i-6+\frac{32}3\sum\limits_{j=1}^{i-1} \frac1{(\alpha_i+2)^2}
 \ \hbox{for any}\ i\ge2.$$
and also
$$ \frac{32}3\sum\limits_{j=1}^{i-1} \frac1{(\alpha_i+2)^2}
\le  \frac{32}3\sum\limits_{j=1}^{i-1} \frac1{(8j-4)^2}\le  \frac23\sum\limits_{j=1}^{\infty} \frac1{(2j-1)^2}= {\pi^2\over 12}<1,$$
where we used the well-known fact $\sum_{n=1}^\infty n^{-2}=\pi^2/6$ and consequently
$\sum_{j=1}^\infty (2j)^{-2}=4^{-1}\sum_{j=1}^\infty j^{-2}=\pi^2/24$,
$\sum_{n=1}^\infty (2j-1)^{-2}=\pi^2/6-\pi^2/24=\pi^2/8$ to derive the last inequality.
Finally, the asserted estimate \eqref{eq:alphabounds} follows.
\end{proof}
%Using again the fact $\sj\in(0,1)$, we derive that
%\beqs
%\sj=f(\sj)\ge e^{-\frac{2}{\aj+2}}(1-\frac{2s}{\aj+2})
%\ge e^{-\frac{2}{\aj+2}}(1-\frac{2}{\aj+2})=e^{-\frac{2}{\aj+2}}\frac{\aj}{\aj+2},
%\eeqs
%which establishes \eqref{eq:sjprop}.
%Finally, we observe that from $\zeta_{2}(s_1^0)=2\ln s_1^0+1+s_1^0=0$ we obtain
%\beqs
%-\ln s_1^0=\frac{1+s_1^0}{2}<1.
%\eeqs
%\par
%Hence, $s_1^0>e^{-1}$, and \eqref{eq:sjest} follows.
%We claim that
%\beq
%\label{eq:sjprop}
%\frac{\aj}{\aj+2}e^{-\frac{2}{\aj+2}}\le\sj\le e^{-\frac{2}{\aj+2}}
%\eeq
%
%\begin{rmk}
%\label{lem:deltavalues}
%The coefficients $b_j$, $C_j$ satisfy the recursive formulae:
%Since
%\beqs
%\bal
%&b_k=\frac{1}{2\ta_k-\al_k+2},
%&&\ln C_k=\frac{c_k}{2\ta_k-\al_k+2},
%\eal
%\eeqs
%
%\end{rmk}
%\beq
%\label{def:Va}
%\Va(y)=\va(y)+(\al-2)\ln|y|=\ln\frac{2\al^2|y|^{\al-2}}{(1+|y|^\al)^2}.
%\eeq
\begin{lemma}
\label{lem:BsubsetA}
There exist constants $0<R_j'<R_j''$, $j=1,2,\ldots,k$, and $0<r_j''<r_j'$,
$j=2,\ldots,k$, independent of $p$,
such that
\beqs
\left\{|x|\le R_1'\de_1 e^{p/8}\right\}
\subset\Bo
\subset
\left\{|x|
\le R_1''\de_1e^{p/8}\right\}
\eeqs
and, for all $j=2,\ldots,k$,
\beqs
\left\{r_j'\dej e^{-\frac{p}{2(\aj-2)}}\le|x|\le R_j'\dej e^{\frac{p}{2(\aj+2)}}\right\}
\subset\Bj
\subset
\left\{r_j''\dej e^{-\frac{p}{2(\aj-2)}}\le|x|
\le R_j''\dej e^{\frac{p}{2(\aj+2)}}\right\}.
\eeqs
In particular, there holds $\Bj\subset\Aj$, for all sufficiently large values of $p$.
\end{lemma}
\begin{rmk}
The specific form \eqref{def:ej} of $\ej$ is essential for the proof.
\end{rmk}
\begin{proof}[Proof of Lemma~\ref{lem:BsubsetA}]
We recall from \eqref{def:Va} that
\beqs
\Va(y)=\va(y)+(\al-2)\ln|y|=\ln\frac{2\al^2|y|^{\al-2}}{(1+|y|^\al)^2}
\eeqs
for all $\al\ge2$ and that
\beqs
\Bj=\left\{x\in\Om:\ \mathcal V_{\aj}(\frac{x}{\dej})\ge-\frac{p}{2}\right\}.
\eeqs
Now, the asserted inclusions readily follow.
\par
Hence, we need only check that $\Bj\subset\Aj$.
We claim that
\beqs
\dej e^{\frac{p}{2(\aj+2)}}\le\dej^{\ej}\de_{j+1}^{1-\ej},
\eeqs
for all sufficiently large values of $p$.
We equivalently check that $e^{\frac{p}{2(\aj+2)}}\le(\de_{j+1}/\dej)^{1-\ej}$.
Recalling the properties of $\dej$ as in \eqref{eq:deltarec}, we have
\beqs
\bal
\ln\de_{j+1}-\ln\dej=\ln\frac{C_{j+1}}{C_j}+(b_j-b_{j+1})p
=\ln\frac{C_{j+1}}{C_j}+\frac{1+s_j}{\aj+2+o(1)}p
\eal
\eeqs
so that 
\beqs
\left(\frac{\dejj}{\dej}\right)^{1-\ej}=Ce^{\frac{(1+s_j)(1-\ej)}{\aj+2+o(1)}p}
\eeqs
Now the result follows since, in view of \eqref{def:ej} we have
\beqs
\frac{(1+s_j)(1-\ej)}{\aj+2+o(1)}=\frac{1}{\aj+2+o(1)}>\frac{1}{2(\aj+2)}
\eeqs
provided that $p$ is sufficiently large.
\par
Similarly, we claim that
\beqs
\dej e^{-\frac{p}{2(\aj-2)}}\ge\de_{j-1}^{\veps_{j-1}}\dej^{1-\veps_{j-1}},
\eeqs
for all sufficiently large values of $p$.
We equivalently check that $e^{-\frac{p}{2(\aj-2)}}\ge(\de_{j-1}/\dej)^{\veps_{j-1}}$.
Recalling the properties of $\dej$ in \eqref{eq:deltarec}, we have
\beqs
\bal
\ln\left(\frac{\de_{j-1}}{\dej}\right)^{\veps_{j-1}}=&C-\veps_{j-1}(b_{j-1}-b_j)p
=&C-\veps_{j-1}\frac{1+s_{j-1}}{(\aj-2)s_{j-1}}p=C-\frac{p}{\aj-2},
\eal
\eeqs
where we used \eqref{def:ej} to derive the last equality.
The asserted inclusions $\Bj\subset\Aj$ are now completely established.
\end{proof}
\begin{proof}
Proof of (iv). We have
\beqs
\tj^{p-1}p^p=\frac{1}{\dej^2}=C_j^{-2}e^{2\bj p}.
\eeqs
Hence,
\beqs
\tj p=\frac{C_j^{-\frac{2}{p-1}}e^{2\bj\frac{p}{p-1}}}{p^{\frac{1}{p-1}}}=e^{2\bj}(1+O(\frac{\ln p}{p})),
\eeqs
as asserted.
\end{proof}
\subsubsection*{The case of one bubble $k=1$}
We remark that for $k=1$ we have $\al_k=\al_1=2$, $b_k=b_1=(2\ta_k-\al_k+2)^{-1}=(4+o(1))^{-1}$,
and therefore
\[
\bal
c_k=&8\pi h(0)-\ln8+o(1)\\
\de_k=&\de_1=C_ke^{-\frac{p}{4+o(1)}}\\
\tau_kp=&\tau_1p=\frac{p}{p^{p/(p-1)}\de_k^{2/(p-1)}}=\sqrt{e}(1+O(\frac{\ln p}{p})),
\eal
\]
in agreement with \cite{emp1}.
%%%%%%%%%%%%%%%%%%%%%%%%%%%%%%%%%%%%%%%%%%%%%%%%%%%%%%%%%%%%%%%%%%%%%%%%%%%%%%%~~~~~
%%%%%%%%%%%%%%%%%%%%%%%%%%%%%%%%%%%%%%%%%%%%%%%%%%%%%%%%%%%%%%%%%%%%%%%%%%%%%%%
\section{Weighted estimation of $\Rp$}
\label{sec:Rp}
We recall from \eqref{def:R}, \eqref{def:Rj} and \eqref{eq:Rsplit} that the error $\Rp$ is defined by
\beqs
\mathcal R_p=\Delta\Up+\gp(\Up)=\sum_{j=1}^k\Rj+\sum_{j=1}^k\chi_{\Aj}\sum_{\stackrel{i=1}{i\neq j}}^k(-1)^{i-1}\ti\Delta\tUdai
\eeqs
where
\beqs
\Rj=\left[(-1)^{j-1}\tj\Delta\tUdaj+\gp(\Up)\right]\chi_{\Aj},
\qquad x\in\Om.
\eeqs
Our aim in this section is to estimate $\Rp$ with respect to a suitable weight function
$\rp(x)$ defined by
\beq
\label{def:rho}
\bal
\rp(x)=&\sum_{j=1}^k\rj(x)\chi_{\Aj}(x),\\
\rj(x)=&\frac{\dej^{2+\eta}+|x|^{2+\eta}}{\dej^\eta},
\qquad x\in\Aj,
\eal
\eeq
where $0<\eta<1$.
Then, in view of \eqref{eq:Rsplit}, we may write
\beq
\label{eq:Rpdecomp}
\rp(x)\Rp(x)=\sum_{j=1}^k\rj\Rj\chi_{\Aj}
+\sum_{j=1}^k\chi_{\Aj}\rj\sum_{\stackrel{i=1}{i\neq j}}^k\ti(-1)^{i-1}\Delta\tUdai.
\eeq
We note that upon rescaling we have:
\beq
\label{eq:rhoresc}
\rj(\dej y)=\dej^2(1+|y|^{2+\eta}),\quad\dej y\in\Aj.
\eeq
We set
\beq
\label{def:starnorm}
\|h\|_{\rp}:=\|\rp\,h\|_{L^\infty(\Om)},
\qquad h\in L^\infty(\Om).
\eeq
We observe that the choice of $\rj$ ensures uniform weighted boundedness of the $j$-th mass
with logarithmic errors, $j=1,2,\ldots,k$:
\beqs
\rj(x)|x|^{\aj-2}e^{\Udaj(x)}(|\Vaj(\frac{x}{\dej})|^q+1)
=O(\frac{|y|^{\aj-2}}{1+|y|^{2\aj-2-\eta}})(|\Vaj(\frac{x}{\dej})|^q+1)=O(1),
\eeqs
for any $q>0$,
see Lemma~\ref{lem:massdecay} below for a more precise statement.
\par
The main result in this section is the following.
\begin{prop}[Main error estimate]
\label{prop:errorest}
The following estimate holds true
\beqs
\|\Rp\|_{\rp}\le\frac{C}{p^4},
\eeqs
for some $C>0$ independent of $p$.
\end{prop}
We devote the remaining part of this section to the proof of Proposition~\ref{prop:errorest}.
The estimates contained in the following lemma will be used systematically in the sequel.
\begin{lemma}
\label{lem:prates}
The following estimates hold true:
\beq
\label{est:wf}
|\Vaj(y)|=O(p),\qquad
\wlai(y)=O(p),
\qquad\foaj(y)=O(p^2),
\qquad\fiai(y)=O(p^4),
\eeq
uniformly for $x=\dej y\in\Aj$, $j=1,2,\ldots,k$,
where $\foaj(y)=\varphi^0(\Vaj(y))$ and $\fiai(y)=\varphi^1(\Vaj(y),\woaj(y))$.
\end{lemma}
\begin{proof}
It suffices to observe that for $x=\dej y$ in $\Aj$ we have:
\beqs
\bal
|\ln|y||\le&C\ln\left(\frac{\de_{j-1}}{\dej}\right)^{\veps_{j-1}}=O(p),
&&\hbox{for\ }|y|\le1,\ j=2,\ldots,k;\\
\ln(|y|+2)\le&C\ln(2+\left(\frac{\dejj}{\dej}\right)^{1-\ej})=O(p),
&&\hbox{for\ }|y|\ge1,\ j=2,\ldots,k-1;\\
\ln(|y|+2)\le&C\ln(2+\left(\frac{\diam\Om}{\dek}\right)^{1-\veps_k})=O(p),
&&\hbox{for\ }|y|\ge1,\ j=k.
\eal
\eeqs
\end{proof}
\begin{proof}[Proof of Proposition~\ref{prop:errorest}, Part~1: decomposition]
We observe that in view of \eqref{eq:Uda} we have
\beq
\label{eq:DeltaUp}
\Delta\Up=\sum_{i=1}^k(-1)^i\ti|x|^{\ai-2}e^{\Udai}
(1+\frac{\wodai-f_{\ai,\dei}^0}{p}+\frac{\widai-f_{\ai,\dei}^1}{p^2}).
\eeq
In view of Lemma~\ref{lem:prates} and Proposition~\ref{prop:solvability}, we deduce that
\beq
\label{est:modmass}
\ti|x|^{\ai-2}e^{\Udai}
(1+\frac{\wodai-f_{\ai,\dei}^0}{p}+\frac{\widai-f_{\ai,\dei}^1}{p^2})
=O(p\,|x|^{\ai-2}e^{\Udai}),
\eeq
uniformly for $x\in\Aj$.
Therefore, using \eqref{eq:Rpdecomp} and \eqref{est:modmass} we may decompose the error estimate as follows:
\beq
\label{eq:Rdecomp}
\bal
\|\Rp\|_{\rp}\le&\sum_{j=1}^k\|\rj\Rj\|_{L^\infty(\Bj)}\\
&\quad+Cp\sum_{j=1}^k\sum_{i\neq j}\|\rj(x)|x|^{\ai-2}e^{\Udai(x)}\|_{L^\infty(\Aj)}
+\sum_{j=1}^k\|\rj\Rj\|_{L^\infty(\Aj\setminus\Bj)},
\eal
\eeq
where
\beq
\label{eq:Rdecomp2}
\|\rj\Rj\|_{L^\infty(\Aj\setminus\Bj)}\le Cp\|\rj(x)|x|^{\aj-2}e^{\Udaj(x)}\|_{L^\infty(\Aj\setminus\Bj)}
+\|\rj\gp(\Up)\|_{L^\infty(\Aj\setminus\Bj)}.
\eeq
\end{proof}
We estimate the right hand sides in \eqref{eq:Rdecomp}--\eqref{eq:Rdecomp2} term by term in the following lemmas.
\begin{lemma}[Leading term estimate in $\Bj$]
\label{lem:Bjmassdecay}
There holds
\beqs
\|\rj\Rj\|_{L^\infty(\Bj)}\le\frac{C}{p^4}.
\eeqs
\end{lemma}
\begin{proof}
In view of \eqref{eq:Rjest} we have, uniformly for $x=\dej y\in\Bj$:
\beqs
\bal
\rj(x)\Rj(x)=&\rj(x)\,|x|^{\aj-2}e^{\Udaj(x)}
\,\frac{O(|\Vaj(\frac{x}{\dej})|^6+1)}{p^4}=O(\frac{1}{p^4}).
%=&|x|^{\aj-2}\frac{2\aj^2\dej^{\aj}}{(\dej^{\aj}+|x|^{\aj})^2}\,
%\frac{\dej^{2+\eta}+|x|^{2+\eta}}{\dej^\eta}\,\frac{O(|\Vaj(\frac{x}{\dej})|^6+1)}{p^4}\\
%=&\frac{2\aj^2|y|^{\aj-2}}{(1+|y|^{\aj})^2}(1+|y|^{2+\eta})\frac{O(|\Vaj(y)|^6+1)}{p^4}\\
%=&\frac{|y|^{\aj-2}}{1+|y|^{2\aj-2-\eta}}\frac{O(|\Vaj(y)|^6+1)}{p^4}.
\eal
\eeqs
Now, the asserted estimate readily follows from \eqref{est:rpmass} with $q=6$.
\end{proof}
\begin{lemma}[Weighted mass estimates in $\Aj$]
\label{lem:massdecay}
Suppose $i\neq j$.
The following estimates hold true, uniformly for $x=\dej y\in\Aj$:
\beqs
\rj(x)|x|^{\ai-2}e^{\Udai(x)}=
\left\{
\bal
&O(\frac{\de_{j-1}}{\dej})^{(1-\veps_{j-1})\ai-2\veps_{j-1}},
&&\hbox{if\ }i<j;\\
&(1+|y|^{2+\eta})\frac{2\aj^2|y|^{\aj-2}}{(1+|y|^{\aj})^2},
&&\hbox{if\ }i=j;\\
&O(\frac{\dej}{\de_{j+1}})^{\ej\ai-(1-\ej)\eta},
&&\hbox{if\ }i>j,\\
\eal
\right.
\eeqs
where $(1-\veps_{j-1})\ai-2\veps_{j-1}>0$ and $\ej\ai-(1-\ej)\eta>0$.
In particular, for any $q>0$ we have
\beqs
\|\rj|x|^{\ai-2}e^{\Udai}(|\Vaj(y)|^q+1)\|_{L^\infty(\Aj)}=
\left\{
\bal
&O(p\,e^{-[(1-\veps_{j-1})\ai-2\veps_{j-1}]\,(b_{j-1}-b_j)\,p}),&&\hbox{if }i<j;\\
&O(1),&&\hbox{if }i=j;\\
&O(p\,e^{-[\ej\ai-(1-\ej)\eta]\,(b_j-b_{j+1})\,p}),&&\hbox{if }i>j.
\eal
\right.
\eeqs
\end{lemma}
\begin{proof}
Suppose $i<j$.
For $x=\dej y\in\Aj$ we have, using Lemma~\ref{lem:massscaling}:
\beqs
|x|^{\ai-2}e^{\Udai(x)}=\frac{2\ai^2\dei^{\ai}}{|\dej y|^{\ai+2}(1+(\frac{\dei}{|\dej y|})^{\ai})^2}
=O(\frac{\dei^{\ai}}{|\dej y|^{\ai+2}}),
\eeqs
where we used \eqref{eq:Ajprops} to deduce that $\dei/|\dej y|=o(1)$ for $\dej y\in\Aj$.
Therefore, for $x=\dej y\in\Aj$:
\beqs
\rj(x)|x|^{\ai-2}e^{\Udai(x)}
=\frac{O(\dej^{2+\eta}+|x|^{2+\eta})}{\dej^\eta}\frac{\dei^{\ai}}{|\dej y|^{\ai+2}}
=O(\frac{\dei}{\dej})^{\ai}\frac{(1+|y|^{2+\eta})}{|y|^{\ai+2}}.
\eeqs
For $y\in\Aj$, $|y|\ge1$, we estimate:
\beqs
\rj(x)|x|^{\ai-2}e^{\Udai(x)}=O(\frac{\dei}{\dej})^{\ai}\frac{1}{(1+|y|)^{\ai-\eta}}
=O(\frac{\dei}{\dej})^{\ai}=O(\frac{\de_{j-1}}{\dej})^{\ai}.
\eeqs
For $y\in\Aj$, $|y|\le1$, we estimate
\beqs
\bal
\rj(x)|x|^{\ai-2}e^{\Udai(x)}=&O(\frac{\dei}{\dej})^{\ai}\frac{1}{|y|^{\ai+2}}\\
\le&O(\frac{\dei}{\dej})^{\ai}(\frac{\dej}{\de_{j-1}})^{\veps_{j-1}(\ai+2)}
&&\hbox{because\ }|y|\ge\left(\frac{\de_{j-1}}{\dej}\right)^{\veps_{j-1}}\hbox{\ in }\Aj\\
\le&O(\frac{\de_{j-1}}{\dej})^{\ai}(\frac{\dej}{\de_{j-1}})^{\veps_{j-1}(\ai+2)}
&&\hbox{because\ }i\le j-1\\
\le&(\frac{\de_{j-1}}{\dej})^{(1-\veps_{j-1})\ai-2\veps_{j-1}}.
\eal
\eeqs
We observe that
\beqs
(1-\veps_{j-1})\ai-2\veps_{j-1}=\frac{\ai-2s_{j-1}}{1+s_{j-1}}>0.
\eeqs
Hence, the asserted estimate is established for $i<j$.
\par
Suppose $i>j$.
In view of Lemma~\ref{lem:massscaling},
we have
\beqs
|x|^{\ai-2}e^{\Udai(x)}=(\frac{\dej}{\dei})^{\ai}\frac{|y|^{\ai-2}}{\dej^2}\frac{2\ai^2}{(1+|\frac{\dej y}{\dei}|^{\ai})^2}
=O\left((\frac{\dej}{\dei})^{\ai}\frac{|y|^{\ai-2}}{\dej^2}\right),
\eeqs
where we used again \eqref{eq:Ajprops} to deduce that $|\frac{\dej y}{\dei}|=o(1)$ for $\dej y\in\Aj$.
Therefore, for $x=\dej y\in\Aj$, we estimate:
\beqs
\bal
\rj(x)|x|^{\ai-2}e^{\Udai(x)}=&\frac{(\dej^{2+\eta}+|x|^{2+\eta})}{\dej^\eta}|x|^{\ai-2}e^{\Udai(x)}\\
=&\dej^2(1+|y|^{2+\eta})\,O(\frac{\dej}{\dei})^{\ai}\frac{|y|^{\ai-2}}{\dej^2}\\
=&O(\frac{\dej}{\dei})^{\ai}(1+|y|^{2+\eta})|y|^{\ai-2}
=O(\frac{\dej}{\dei})^{\ai}(1+|y|)^{\ai+\eta}.
\eal
\eeqs
Since for $y\in\Aj/\dej$ we have $|y|\le(\de_{j+1}/\dej)^{1-\ej}$, we deduce that
\beqs
\rj(x)|x|^{\ai-2}e^{\Udai(x)}=O(\frac{\dej}{\dei})^{\ai}(1+\frac{\de_{j+1}}{\dej})^{(1-\ej)(\ai+\eta)}
=O(\frac{\dej}{\de_{j+1}})^{\ej\ai-(1-\ej)\eta},
\eeqs
where we used the fact $i\ge j+1$ to derive the last inequality.
We observe that
\beqs
\ej\ai-(1-\ej)\eta=\frac{\sj\ai-\eta}{1+\sj}>0.
\eeqs
The proof for $i=j$ follows by straightforward rescaling.
\end{proof}
\begin{lemma}[Residual mass decay in $\Aj\setminus\Bj$]
\label{lem:massAjBj}
There holds:
\beqs
\rj(x)|x|^{\aj-2}e^{\Udaj(x)}
\le Ce^{-\frac{\aj-\eta}{2(\aj+2)}p},
\eeqs
uniformly for $x\in\Aj\setminus\Bj$.
\end{lemma}
\begin{proof}
We recall from Lemma~\ref{lem:massscaling} that
\beqs
|x|^{\aj-2}e^{\Udaj(x)}=\frac{|y|^{\aj-2}}{\dej^2}e^{\vj(y)}
=\frac{2\aj^2|y|^{\aj-2}}{\dej^2(1+|y|^{\aj})^2},
\qquad x=\dej y.
\eeqs
Recalling \eqref{eq:rhoresc}, it follows that
\beqs
\rj(x)|x|^{\aj-2}e^{\Udaj(x)}
=\frac{2\aj^2(1+|y|^{2+\eta})|y|^{\aj-2}}{(1+|y|^{\aj})^2},
\qquad x=\dej y.
\eeqs
We recall from Lemma~\ref{lem:BsubsetA} that
\beqs
\frac{A_1^{\veps_1}\setminus\Bo}{\de_1}\subset
\left\{R_1'e^{\frac{p}{8}}\le|y|\le\left(\frac{\de_{2}}{\de_1}\right)^{1-\veps_1}\right\},
\eeqs
\beqs
\bal
\frac{\Aj\setminus\Bj}{\dej}\subset
\left\{\left(\frac{\de_{j-1}}{\dej}\right)^{\veps_{j-1}}\le|y|\le r_j'e^{-\frac{p}{2(\aj-2)}}\right\}
\cup&\left\{R_j'e^{\frac{p}{2(\aj+2)}}\le|y|\le\left(\frac{\de_{j+1}}{\dej}\right)^{1-\ej}\right\},\\
&j=2,\ldots,k-1,
\eal
\eeqs
and
\beqs
\frac{\Ak\setminus\Bk}{\dek}\subset
\left\{\left(\frac{\de_{k-1}}{\dek}\right)^{\veps_{k-1}}\le|y|\le r_k'e^{-\frac{p}{2(\ak-2)}}\right\}
\cup\left\{R_k'e^{\frac{p}{2(\ak+2)}}\le|y|\le\frac{\diam\Om}{\dek}\right\}.
\eeqs
For $0<|y|\le1$, $j=1$, there is nothing to prove. For $0<|y|\le1$, $j=2,\ldots,k$, we estimate:
\beqs
\rj(x)|x|^{\aj-2}e^{\Udaj(x)}\le C|y|^{\aj-2}.
\eeqs
Hence, for $x=\dej y\in\Aj\setminus\Bj$, $|y|\le1$, we have
\beqs
\rj(x)|x|^{\aj-2}e^{\Udaj(x)}
\le C(e^{-\frac{p}{2(\aj-2)}})^{\aj-2}=Ce^{-p/2}.
\eeqs
For $x=\dej y\in\Aj\setminus\Bj$, $|y|\ge1$, we estimate:
\beqs
\rj(x)|x|^{\aj-2}e^{\Udaj(x)}
\le C\frac{|y|^{\aj-2}}{(1+|y|)^{2\aj-2-\eta}}
\le\frac{C}{(1+|y|)^{\aj-\eta}}\le Ce^{-\frac{\aj-\eta}{2(\aj+2)}p}.
\eeqs
The asserted estimate follows, since $(\aj-\eta)/(\aj+2)<1$.
\end{proof}
\begin{lemma}[Expansion of $\gp(\Up)$]
\label{lem:gpUpexp}
The following expansion holds true:
\beqs
\gp(\Up(x))=\sum_{j=1}^k\tj|x|^{\aj-2}e^{\Udaj(x)}
\left\{1+\frac{\Vaj(\frac{x}{\dej})+O(1)}{p}\right\}\chi_{\Bj}(x)+\omp(x)\chi_{\Aj\setminus\Bj},
\eeqs
where $\|\rj\omp\|_{L^\infty(\Aj\setminus\Bj)}=O(e^{-\veps_0p}/p)$, where $\veps_0=\frac{\aj-\eta}{2(\aj+2)}$.
\end{lemma}
\begin{proof}
We estimate separately in the sets $\Bj$, $(\Aj\setminus\Bj)\cap\{\Up>0\}$ and $(\Aj\setminus\Bj)\cap\{\Up<0\}$,
respectively.
\par
\textit{Claim~1.} (Estimation of $\gp(\Up)$ in $(\Aj\setminus\Bj)\cap\{\Up>0\}$).
\par
There holds:
\beqs
\label{eq:gppos}
\rj(x)\gp(\Up)
\le\frac{C}{p}\,e^{-\frac{\aj-\eta}{2(\aj+2)}p},
\eeqs
uniformly for $x\in\Aj\setminus\Bj$, $\Up(x)>0$.
\par
Indeed we have, using Lemma~\ref{lem:gpp}--(i) and \eqref{eq:Upexpzero}, for $x=\dej y\in\Aj\setminus\Bj$,
$\Up(x)>0$:
\beqs
\bal
\rj(x)|\gp(\Up(x))|\chi_{\{\Up>0\}}
\le C(1+|y|^{2+\eta})\tj\frac{|y|^{\aj-2}}{(1+|y|^{\aj})^2}
\le C\tj\frac{|y|^{\aj-2}}{(1+|y|)^{2\aj-2-\eta}}.
\eal
\eeqs
For $x=\dej y\in\Aj\setminus\Bj$, $|y|\le1$, $j=2,\ldots,k$, we estimate:
\beqs
\bal
\rj(x)|\gp(\Up(x))|\chi_{\{\Up>0\}}\chi_{\Aj\setminus\Bj}
\le C\tj|y|^{\aj-2}\le C\tj(e^{-\frac{p}{2(\aj-2)}})^{\aj-2}
=C\tj e^{-p/2}.
\eal
\eeqs
For $x=\dej y\in\Aj\setminus\Bj$, $|y|\ge1$, $j=1,2,\ldots,k$, we estimate:
\beqs
\bal
\rj(x)|\gp(\Up(x))|\chi_{\{\Up>0\}}\chi_{\Aj\setminus\Bj}
\le \frac{C\tj}{(1+|y|)^{\aj-\eta}}
\le C\tj e^{-\frac{p}{2(\aj+2)}(\aj-\eta)}.
\eal
\eeqs
Now the claim follows recalling that $\tj=O(p^{-1})$.
\par
\textit{Claim~2.} (Estimation of $\gp(\Up)$ in $\Aj\cap\{\Up<0\}$).
\par
The following decay estimate holds true in $\{\Up<0\}\cap\Aj$:
\beq
\label{eq:gpAjBjneg}
\rj(x)|\gp(\Up(x))|\chi_{\{\Up<0\}}\le 
\frac{C}{p}
e^{-\frac{\aj-\eta}{\aj+2}p},
\eeq
uniformly for $x\in\Aj$.
\par
Indeed, we recall from \eqref{eq:Upexpzero} that
\beqs
\Up(x)=(-1)^{j-1}\tj p\left\{1+\frac{\Vaj(\frac{x}{\dej})+O(1)}{p}\right\},
\qquad x\in\Aj,
\eeqs
where $\Vaj$ is defined in \eqref{def:Va}.
We observe that $\Vaj(y)<0$ for $0<|y|\ll1$ and for $|y|\gg1$.
Moreover, we have for $x=\dej y\in\Aj$
\beqs
|\gp(\Up(x))|=\tj^pp^p|\gp(1+\frac{\Vaj(y)+O(1)}{p})|
=\frac{\tj}{\dej^2}|\gp(1+\frac{\Vaj(y)+O(1)}{p})|,
\eeqs
where we used the last equation in \eqref{eq:adtsys} to derive the last equality.
We have, using Lemma~\ref{lem:gpp}--(i),
\beqs
|\gp(1+\frac{\Vaj(y)+O(1)}{p})|\chi_{\{\Up<0\}}
\le e^{-\Vaj(y)-2p+O(1)}
\le C\frac{(1+|y|^{\aj})^2}{|y|^{\aj-2}}e^{-2p}
\eeqs
and therefore, in view of \eqref{eq:rhoresc},
\beq
\label{est:gpUpneg}
\bal
\rj(x)|\gp(\Up(x))|\chi_{\{\Up<0\}}
\le C\dej^2(1+|y|^{2+\eta})\frac{\tj}{\dej^2}\frac{(1+|y|^{\aj})^2}{|y|^{\aj-2}}e^{-2p}
\le C\tj\frac{(1+|y|)^{2\aj+2+\eta}}{|y|^{\aj-2}}e^{-2p}.
\eal
\eeq
For $x=\dej y\in\Aj$, $|y|\le1$, $j\ge2$ (for $j=1$ there is nothing to prove) we estimate, recalling that $|y|\ge(\de_{j-1}/\dej)^{\veps_{j-1}}$:
\beqs
\bal
\rj(x)|\gp(\Up(x))|
\le&C\frac{\tj e^{-2p}}{|y|^{\aj-2}}
\le C\tj(\frac{\dej}{\de_{j-1}})^{\veps_{j-1}(\aj-2)}e^{-2p}\\
=&C\tj e^{\veps_{j-1}(b_{j-1}-b_j)(\aj-2)p}e^{-2p}.
\eal
\eeqs
In view of \eqref{eq:deltarec} we may simplify the exponent above:
\beqs
\veps_{j-1}(b_{j-1}-b_j)(\aj-2)=\veps_{j-1}\frac{1+s_{j-1}}{s_{j-1}(\aj-2)}(\aj-2)
=1,
\eeqs
so that for $x=\dej y\in\Aj$, $|y|\le1$, $j\ge2$, we finally obtain that
\beqs
\rj(x)|\gp(\Up(x))|\le\frac{Ce^{-p}}{p}.
\eeqs
For $x=\dej y\in\Aj$, $|y|\ge1$, we estimate using \eqref{est:gpUpneg} and $|y|\le(\de_{j+1}/\dej)^{1-\ej}$:
\beqs
\bal
\rj(x)|\gp(\Up(x))|
\le&C\tj(1+|y|)^{\aj+4+\eta}e^{-2p}
\le C\tj(\frac{\de_{j+1}}{\de_j})^{(1-\ej)(\aj+4+\eta)}e^{-2p}\\
\le&C\tj e^{(1-\ej)(b_j-b_{j+1})(\aj+4+\eta)p}e^{-2p}.
\eal
\eeqs
In view of Proposition~\ref{prop:solvability}--(v) we may simplify the exponent above:
\beqs
\bal
(1-\ej)(b_j-b_{j+1})(\aj+4+\eta)=&(1-\ej)\frac{1+\sj}{\aj+2+O(p^{-1})}(\aj+4+\eta)
=\frac{\aj+4+\eta}{\aj+2}+O(\frac{1}{p}).
\eal
\eeqs
Since
\[
2-\frac{\aj+4+\eta}{\aj+2}=\frac{\aj-\eta}{\aj+2},
\]
we derive Claim~2.
\par
Finally, from Claim~1 and Claim~2 we deduce that
\beqs
\rj(x)|\gp(\Up(x))|\chi_{\Aj\setminus\Bj}\le\frac{C}{p}e^{-\frac{\aj-\eta}{2(\aj+2)}p}.
\eeqs
\end{proof}
Finally, we conclude the proof of the main error estimate.
\begin{proof}[Proof of Proposition~\ref{prop:errorest}, Part~2: conclusion]
In view of the decompositions of $\Rp$ as in \eqref{eq:Rdecomp}--\eqref{eq:Rdecomp2}
and the estimates in Lemma~\ref{lem:Bjmassdecay}, Lemma~\ref{lem:massdecay},
Lemma~\ref{lem:massAjBj} and Lemma~\ref{lem:gpUpexp}, we finally derive:
\beqs
\|\rj\Rj\|_{L^\infty(\Aj\setminus\Bj)}
\le\frac{C}{p}\|\rj|x|^{\aj-2}e^{\Udaj}\|_{L^\infty(\Aj\setminus\Bj)}
+\|\gp(\Up)\|_{L^\infty(\Aj\setminus\Bj)}\le Ce^{-\eps_0p},
\eeqs
for $\eps_0=\frac{\aj-\eta}{2(\aj+2)}$.
\end{proof}
%%%%%%%%%%%%%%%%%
%%%%%%%%%%%%%%%%%%%%%%%%%%%%%%%%%%%%%%%%%%%%%%%%%%%%%%%%%%%%%%~~~~~
%%%%%%%%%%%%%%%%%%%%%%%%%%%%%%%%%%%%%%%%%%%%%%%%%%%%%%%%%%%%%%%%%%%%%%%%%%%%%%%
\section{The linearized problem: estimates for $\Wp$ and choice of the $\ej$'s}
\label{sec:lin}
We define the linearized operator
\beq
\label{def:Lp}
\Lp\phi:=\Delta\phi+\Wp(x)\phi,
\qquad\phi\in C^2(\Om)\cap C(\overline\Om),
\eeq
where the \lq\lq potential" $\Wp$ is defined by
\beqs
\label{def:Wp}
\Wp(x):=\gp'(\Up(x))
\eeqs
and $\gp'(t)=p|t|^{p-1}$.
\par
Let
\beqs
\bal
\Daj(y):=&\woaj(y)-\Vaj(y)-\frac{\Vaj^2(y)}{2}\\
\Dadj(x):=&\Daj(\frac{x}{\dej}).
\eal
\eeqs
Note that
\beqs
-p\le\frac{\Daj(y)}{p}\le C,
\eeqs
uniformly for $y\in\Aj/\dej$.
Our aim in this section is to establish the following fact.
\begin{prop}
\label{prop:Wpest}
The following estimate holds true uniformly for $x\in\Om$:
\beq
\label{est:Wpmain}
\Wp(x)\le\ovC\,\sum_{i=1}^k|x|^{\ai-2}e^{\Udai(x)}\chi_{\Ai}(x),
\eeq
for some $\ovC>0$ independent of $p\to+\infty$.
Moreover, for any $j=1,2,\ldots,k$ there holds the expansion
\beqs
\Wp(x)=\sum_{j=1}^k|x|^{\aj-2}e^{\Udaj(x)}\{1+\frac{\Dadj(x)}{p}+\frac{O(|\Vaj(\frac{x}{\dej})|^4+1)}{p^2}\}\chi_{\Bj}(x)
+\omp(x),
\eeqs
where $\|\omp\|_{\rp}\le Ce^{-\beta_0p}$ for some $\beta_0>0$,
uniformly for $x\in\Om$.
\end{prop}
We beign by establishing some auxiliary results.
The following result justifies the choice of the $\ej$'s as in \eqref{def:ej}, namely
\beqs
\ej=\frac{\sj}{1+\sj}, \qquad j=1,2,\ldots,k-1.
\eeqs
\begin{lemma}
\label{lem:ejineq}
Let $\aj$, $\dej$, $j=1,2,\ldots,k$, and $\sj$, $j=1,2,\ldots,k-1$ be the parameters defined in Proposition~\ref{prop:solvability}. 
Let $0<\veps<1$.
The following implications hold true:
\begin{enumerate}
  \item[(i)]For all $j=2,\ldots,k,$ if
\beqs
\veps\le\frac{s_{j-1}}{1+s_{j-1}}
\eeqs
then
\beqs
(\aj-2)\ln|y|\ge-p-C\quad\hbox{for all}\quad\left(\frac{\de_{j-1}}{\dej}\right)^\veps\le|y|\le1.
\eeqs
  \item[(ii)] For all $j=1,2,\ldots,k-1$, if
\beqs
\veps\ge\frac{\sj}{1+\sj},
\eeqs then
\beqs
(\aj+2)\ln\frac{1}{|y|}\ge-p-C\quad\hbox{for all}\quad1\le|y|\le(\frac{\dejj}{\dej})^{1-\veps}.
\eeqs
\end{enumerate}
Consequently, if $\ej$, $j=1,2,\ldots,k-1$, is given by \eqref{def:ej}, then 
\beq
\label{ineq:VajpC}
\Vaj(y)\ge-p-C\quad\hbox{uniformly for}\quad y\in\frac{\Aj}{\dej},\ j=1,2,\ldots,k.
\eeq
\end{lemma}
\begin{proof}
Proof of (i).
Assuming that $\veps\le s_{j-1}/(1+s_{j-1})$, it suffices to show that
\beq
\label{eq:ejle}
(\aj-2)\ln\left(\frac{\de_{j-1}}{\dej}\right)^\veps\ge-p-C.
\eeq
We have:
\beqs
\bal
(\aj-2)\ln\left(\frac{\de_{j-1}}{\dej}\right)^\veps=&-(\aj-2)\veps(b_{j-1}-\bj)\,p-C
=-(\aj-2)\veps\frac{1+s_{j-1}}{s_{j-1}(\aj-2)}\,p-C\\
=&-\veps\frac{1+s_{j-1}}{s_{j-1}}p-C,
\eal
\eeqs
where we used Proposition~\ref{prop:solvability}-(v) to derive the last equality.
Hence, \eqref{eq:ejle} holds true if $\veps(1+s_{j-1})/s_{j-1}\le1$ and therefore Part~(i)
is established.
\par
Proof of (ii).
Assuming that $\veps\ge\sj/(1+\sj)$, it suffices to show that
\beq
\label{eq:ejge}
(\aj+2)\ln\left(\frac{\dej}{\dejj}\right)^{1-\veps}\ge-p-C.
\eeq
We have, using Proposition~\ref{prop:solvability}-(v):
\beqs
\bal
(\aj+2)\ln\left(\frac{\dej}{\dejj}\right)^{1-\veps}=&-(\aj+2)(1-\veps)(\bj-\bjj)\,p-C\\
=&-(\aj+2)(1-\veps)\frac{1+\sj}{\aj+2+O(\frac{1}{p})}\,p-C
=-(1-\veps)(1+\sj)\,p-C.
\eal
\eeqs
Hence, \eqref{eq:ejge} holds true if $(1-\veps)(1+\sj)\le1$, that is
if $\veps\ge\sj/(1+\sj)$.
Thus, Part~(ii) is established.
\par
Finally, we assume that $\ej$ is given by \eqref{def:ej}.
For $j=1$, the asserted inequality \eqref{ineq:VajpC} follows from Part~(ii), since $\mathcal V_{\al_1}$
is bounded at $y=0$.
For $j=2,\ldots,k-1$ the asserted inequality \eqref{ineq:VajpC} follows from Part~(i) and Part~(ii).
For $j=k$ we need only consider the case $|y|\ge1$.
Recalling from Proposition~\ref{prop:solvability}--(v) that $b_k=(\ak+2+O(p^{-1}))^{-1}$, 
we compute
\beqs
\bal
\Vak(y)\ge\ln\frac{1}{|y|^{\ak+2}}-C\ge&\ln\de_k^{\ak+2}-C=-b_k(\ak+2)\,p-C=-\frac{\ak+2}{\ak+2+O(\frac{1}{p})}\,p-C\\
\ge&-p-C.
\eal
\eeqs
\end{proof}
\begin{proof}[Proof of Proposition~\ref{prop:Wpest}]
Estimate~\eqref{est:Wpmain} is a direct consequence of Lemma~\ref{lem:gpgpp}--(ii).
\end{proof}
%%%%%%%%%%%%%%%%%%%%%%%%%%%%%%%%%%%%%%%%%%%%%%%%%%%%%%%%%%%%%%%%%%%%%%%%%%%%%%%%%%%%%
%%%%%%%%%%%%%%%%%%%%%%%%%%%%%%%%%%%%%%%%%%%%%%%%%%%%%%%%%%%%%%%%%%%%%%%%%%%%%%%%%%%%%
%%%%%%%%%%%%%%%%%%%%%%%%%%%%%%%%%%%%%%%%%%%%%%%%%%%%%%%%%%%%%%%%%%%%%%%%%%%%%%%%%%%%%
%%%%%%%%%%%%%%%%%%%%%%%%%%%%%%%%%%%%%%%%%%%%%%%%%%%%%%%%%%%%%%%%%%%%%%%%%%%%%%%%%%%%%
%%%%%%%%%%%%%%%%%%%%%%%%%%%%%%%%%%%%%%%%%%%%%%%%%%%%%%%%%%%%%%%%%%%%%%%%%%%%%%%%%%%%%

%%%%%%%%%%%%%%%%%%%%%%%%%%%%%%%%%%%%%%%%%%%%%%%%%%%%%%%%%%%%%%%%%%%%%%%%%%%%%%%~~~~~
%%%%%%%%%%%%%%%%%%%%%%%%%%%%%%%%%%%%%%%%%%%%%%%%%%%%%%%%%%%%%%%%%%%%%%%%%%%%%%%

%%%%%%%%%%%%%%%%%%%%%%%%%%%%%%%%%%%%%%%%%%%%%%%%%%%%%%%%%%%%%%%%%%%%%%%%%%%%%%%~~~~~
%%%%%%%%%%%%%%%%%%%%%%%%%%%%%%%%%%%%%%%%%%%%%%%%%%%%%%%%%%%%%%%%%%%%%%%%%%%%%%%
\section{Linearized problem: evaluation of integrals}
Let
\beq
\label{def:Daj}
\bal
\Daj(y):=&\woaj(y)-\Vaj(y)-\frac{\Vaj^2(y)}{2}\\
\Dadj(x):=&\Daj(\frac{x}{\dej}),
\eal
\eeq
where $\Vaj$ is defined in \eqref{def:Va} and $\zajo$ is the radial eigenfunction defined by
\beq
\label{def:zajo}
\zajo(y)=\frac{1-|y|^{\aj}}{1+|y|^{\aj}},
\eeq
which satisfies the linearized equation
\beq
\label{eq:linpb}
\Delta z+|y|^{\aj-2}e^{\vj(y)}z=0
\qquad\hbox{in }\rr^2.
\eeq
Observing that in view of equation~\eqref{eq:linpb} we may write  $|y|^{\aj-2}e^{\vj(y)}\zajo(y)=-\Delta\zajo$,
and integrating by parts,  we obtain the following integrals to be used below:
\beq
\label{eq:zajoid}
\bal
&\int_{\rr^2}|y|^{\aj-2}e^{\vj(y)}\zajo(y)\,dy=0\\
&\int_{\rr^2}|y|^{\aj-2}e^{\vj(y)}\zajo(y)\vj(y)\,dy=4\pi\aj\\
&\int_{\rr^2}|y|^{\aj-2}e^{\vj(y)}\zajo(y)\ln|y|\,dy=-2\pi,
\eal
\eeq
see also (4.30)--(4.31) in \cite{gropi}.
We consider the linear operator $\Lpone$ defined for $\phi\in C^2(\Om)$ by:
\beq
\label{def:Lpone}
\Lpone\phi=\Delta\phi+\sum_{j=1}^k|x|^{\aj-2}e^{U_{\aj,\dej}(x)}\{1+\frac{\Dadj}{p}\}\chi_{\Bj}(x)\phi.
\eeq
Our aim in this section is to show the following.
\begin{prop}
\label{prop:Lpmainest}
There exists $C>0$ such that for any solution $\phi\in C^2(\Om)\cap C(\overline\Om)$ to the problem
\beqs
\begin{cases}
\Lp\phi=h&\hbox{in }\Om\\
\phi=0&\hbox{on }\partial\Om
\end{cases}
\eeqs
there holds
\beqs
\|\phi\|_{L^\infty(\Om)}\le Cp\|h\|_{\rp}.
\eeqs
\end{prop}
We derive the proof of Proposition~\ref{prop:Lpmainest} by a contradiction argument.
Suppose that $\phin\in C^2(\Om)\cap C(\overline\Om)$, $n\in\mathbb N$, is such that
$\mathcal L_{p_n}\phin=h_n$ with $\pn\to+\infty$, $\pn\|h_n\|_{\rho_{p_n}}\to0$ and $\|\phi_n\|_{L^\infty(\Om)}=1$.
\begin{lemma}
[Asymptotic profile of $\phin$]
Let $\phij(y)=\phin(\dej y)$, $j=1,2,\ldots,k$.
Then, for every $j=1,2,\ldots,k$ there exists $\gj\in\rr$ such that $\phij(y)\to\gj\zajo(y)$ in $C_{\mathrm{loc}}^2(\rr^2\setminus\{0\})$,
weakly in $\mathbb D^{1,2}(\rr^2)$ and almost everywhere.
\end{lemma}
\begin{proof}
Let $j\in\{1,2,\ldots,k\}$ be fixed.
The rescaled function $\phij$ satisfies
\beqs
-\Delta\phij=\dej^2\,\mathcal W_{\pn}(\dej y)\phij-\dej^2\hn(\dej y)\qquad\hbox{in }\frac{\Om}{\dej},
\eeqs
and $\|\phij\|_{L^\infty(\Om/\dej)}=1$.
Let $\mathcal K\subset\rr^2\setminus\{0\}$ be compact. 
Since $\Bj$ invades $\rr^2$, we may assume that
$\mathcal K\subset\Bj$. In view of Lemma~\ref{lem:gpgpp}--(ii) we have
\beqs
\dej^2\,\mathcal W_{\pn}(x)=|y|^{\aj-2}e^{\vj(y)}(1+O(\frac{1}{\pn})),
\eeqs
uniformly for $x=\dej y\in\mathcal K$.
In view of Lemma~\ref{lem:embedding}--(ii) and the decay assumption on $\|\hn\|_{\rho_{\pn}}$ we have
\beqs
\|\dej^2\hn(\dej y)\|_{L^\infty(\Aj)}\le\|\hn\|_{\rho_{\pn}}=o(\frac{1}{\pn}).
\eeqs
By elliptic regularity there exists $\phijo\in C^2(\rr^2\setminus\{0\})$, satisfying 
\beq
\label{eq:phijo}
-\Delta\phijo=|y|^{\aj-2}e^{\vj(y)}\phijo
\eeq 
in $\rr^2\setminus\{0\}$,
such that $\phij\to\phijo$ in $\mathcal C_{\mathrm{loc}}^{1,\al}(\rr^2\setminus\{0\})$,
$\al\in(0,1)$. Since $\|\phij\|_{L^\infty(\Om/\dej)}=1$, the right hand side in \eqref{eq:phijo} is uniformly bounded in 
$\Om/\dej$, and thus we deduce that $\phijo$ satisfies \eqref{eq:phijo}
in whole space $\rr^2$. Now we recall that in view of Proposition~\ref{prop:aj} we have $\aj\not\in\mathbb N$ 
for all $j\ge2$. Therefore, by the characterization of bounded solutions to \eqref{eq:linpb} as established in
\cite{DelPinoEspositoMusso2012}, we deduce that $\phijo=\gj\zajo$ for some $\gj\in\rr$.
\end{proof}
Our aim is to show that $\gj=0$ for all $j=1,2,\ldots,k$. To this end,
extending and approach in \cite{gropi},
let
\beqs
\sjs:=p_n\int_{\Bj}|x|^{\aj-2}e^{\Udaj(x)}\left\{1+\frac{\Dadj(x)}{p_n}\right\}\phin(x)\,dx.
\eeqs
We begin by obtaining two linear relations between the $\sjs$'s and the $\gj$'s.
It is convenient to set
\beq
\label{def:Ia}
\Iaj=\int_{\rr^2}|y|^{\aj-2}e^{\vj}(\zajo(y))^2\Daj(y)\,dy
=2\aj^2\int_{\rr^2}\frac{|y|^{\aj-2}(1-|y|^{\aj})^2}{(1+|y|^{\aj})^4}\Daj(y)\,dy.
\eeq
With this notation, we have:
\begin{lemma}[Linear system for $\sjs,\gj$]
\label{lem:linrel}
For $j=1,2,\ldots,k$, the following linear relations hold true:
\beqs
\tag{$S_j^1$}
%\label{eq:sg1}
2\sum_{i<j}\sis+\sjs+\Iaj\gj=o(1)
\eeqs
and
\beqs
\tag{$S_j^2$}
%\label{eq:sg2}
\bj\sum_{i\le j}\sis+\sum_{i>j}\bi\sis+2\pi\sum_{i\ge j}\gi=o(1),
\eeqs
as $n\to\infty$.
\end{lemma}
\begin{proof}
We observe that in view of Proposition~\ref{prop:Wpest} we may write
\beqs
\mathcal W_{\pn}=\sum_{j=1}^k|x|^{\aj-2}e^{U_{\aj,\dej}(x)}\{1+\frac{\Dadj}{p}\}\chi_{\Bj}
+\widetilde\om_n,
\eeqs
where $\|\widetilde\om_n\|_{\rho_{\pn}}=O(\pn^{-2})$.
Therefore, setting $\thn=\hn-\widetilde\om_n\phin$, we obtain from the contradiction assumption that
\beqs
\mathcal L_{\pn}^1\phin=\thn, \qquad\pn\|\thn\|_{\rho_{\pn}}=o(1), \qquad\|\phin\|_{L^\infty(\Om)}=1,
\eeqs
where $\mathcal L_{\pn}^1$ is the operator defined in \eqref{def:Lpone} with $p=\pn$.
\par
Hence, we consider problem
\beq
\label{eq:phi1}
\left\{
\bal
-\Delta\phin=&\sum_{i=1}^k|x|^{\ai-2}e^{\Udai(x)}\{1+\frac{\Dadi}{\pn}\}\chi_{\Bi}(x)\phin-\thn
&&\hbox{in }\Omega\\
\phin=&0,
&&\hbox{on }\pl\Omega,
\eal
\right.
\eeq
and, for any fixed $j=1,2,\ldots,k$, we consider the problems
\beq
\label{eq:Pz1}
\left\{
\bal
-\Delta P\zajod=&|x|^{\aj-2}e^{\Udaj}\zajod,
&&\hbox{in }\Omega\\
P\zajod=&0,
&&\hbox{on }\pl\Omega
\eal
\right.
\eeq
and
\beq
\label{eq:PU}
\left\{
\bal
-\Delta P\Udaj=&|x|^{\aj-2}e^{\Udaj},&&\hbox{in\ }\Omega,\\
P\Udaj=&0&&\hbox{on\ }\pl\Omega.
\eal
\right.
\eeq
Proof of $(S_j^1)$. Testing problem~\eqref{eq:phi1} by $P\zajod$, integrating by parts
and using \eqref{eq:Pz1} we obtain the relation: 
\beq
\label{eq:inteq1}
\int_\Om|x|^{\aj-2}e^{\Udaj}\zajod\phin=\sum_{i=1}^k\int_{\Bi}|x|^{\ai-2}e^{\Udai}\{1+\frac{\Dadi}{\pn}\}\phin P\zajod
-\int_\Om\thn P\zajod.
\eeq
We may write
\beqs
\bal
\int_{\Bj}&|x|^{\aj-2}e^{\Udaj}\{1+\frac{\Dadj}{\pn}\}\phin P\zajod
-\int_{\Om}|x|^{\aj-2}e^{\Udaj}\zajod\phin\\
=&\int_{\Bj}|x|^{\aj-2}e^{\Udaj}\{1+\frac{\Dadj}{\pn}\}\phin(P\zajod-\zajod)
+\frac{1}{\pn}\int_{\Bj}|x|^{\aj-2}e^{\Udaj}\Dadj\zajod\phin\\
&\qquad-\int_{\Om\setminus\Bj}|x|^{\aj-2}e^{\Udaj}\zajod\phin.
%=&\int_{\Bj}|x|^{\aj-2}e^{\Udaj}\{1+\frac{\Dadj}{\pn}\}\phin(1+O(\dej^{\aj}))+\frac{1}{\pn}\int_{\Bj}|x|^{\aj-2}e^{\Udaj}\Dadj\zajod\phin,
\eal
\eeqs
Therefore, multiplying  \eqref{eq:inteq1} by $\pn$, we derive that
\beqs
\bal
\pn\int_{\Bj}|x|^{\aj-2}&e^{\Udaj}\{1+\frac{\Dadj}{\pn}\}\phin(P\zajod-\zajod)
+\int_{\Bj}|x|^{\aj-2}e^{\Udaj}\Dadj\zajod\phin\\
&+\pn\sum_{i\neq j}\int_{\Bi}|x|^{\ai-2}e^{\Udai}\{1+\frac{\Dadi}{\pn}\}\phin P\zajod\\
=&\pn\int_\Om\thn P\zajod-\pn\int_{\Om\setminus\Bj}|x|^{\aj-2}e^{\Udaj}\zajod\phin.
\eal
\eeqs
%
%\beqs
%\bal
%\int_{\Om\setminus\Bj}&|x|^{\aj-2}e^{\Udaj}\{1+\frac{\Dadj}{\pn}\}\phin P\zajod
%=\int_{\Bj}|x|^{\aj-2}e^{\Udaj}\{1+\frac{\Dadj}{\pn}\}\phin(P\zajod-\zajod)\\
%&+\frac{1}{\pn}\int_{\Bj}|x|^{\aj-2}e^{\Udaj}\Dadj\zajod\phin\\
%&+\sum_{i\neq j}\int_{\Bi}|x|^{\ai-2}e^{\Udai}\{1+\frac{\Dadi}{\pn}\}\phin P\zajod
%-\int_\Om\thn P\zajod.
%\eal
%\eeqs
%
Multiplying by $\pn$ and observing that
\beqs
\bal
&\pn\int_\Om\thn P\zaiod=o(1),\\
&-C\pn^2\le\Dai(y)\le C\pn\\
&\int_{\Om\setminus\Bj}|x|^{\aj-2}e^{\Udaj}\,dx=O(\frac{\de_{j-1}}{\dej}+(\frac{\dej}{\de_{j+1}})^{\aj/2})
\eal
\eeqs
we derive 
\beq
\label{eq:phiPzid}
\bal
\pn\int_{\Bj}|x|^{\aj-2}e^{\Udaj}&\{1+\frac{\Dadj}{\pn}\}\phin(P\zajod-\zajod)
+\int_{\Bj}|x|^{\aj-2}e^{\Udaj}\Dadj\zajod\phin\\
&+\sum_{i\neq j}\pn\int_{\Bi}|x|^{\ai-2}e^{\Udai}\{1+\frac{\Dadi}{\pn}\}\phin P\zajod=o(1).
\eal
\eeq
In view of Lemma~\ref{lem:radialefscaling}, the first term in \eqref{eq:phiPzid} takes the form
\beqs
\bal
\pn\int_{\Bj}|x|^{\aj-2}&e^{\Udaj}\{1+\frac{\Dadj}{\pn}\}\phin(P\zajod-\zajod)\\
=&\pn\int_{\Bj}|x|^{\aj-2}e^{\Udaj}\{1+\frac{\Dadj}{\pn}\}\phin(1+O(\dej^{\aj}))
=\sjs+o(1).
\eal
\eeqs
The second term in \eqref{eq:phiPzid} satisfies
\beqs
\int_{\Bj}|x|^{\aj-2}e^{\Udaj}\Dadj\zajod\phin\,dx=\Iaj+o(1).
\eeqs
In order to estimate the third term in \eqref{eq:phiPzid} we set $x=\dei y$:
\beqs
\bal
\pn\int_{\Bi}|x|^{\ai-2}e^{\Udai}&\{1+\frac{\Dadi}{\pn}\}\phin P\zajod\,dx\\
=&\pn\int_{\Bi/\dei}|y|^{\ai-2}e^{\vi(y)}\{1+\frac{\Dai(y)}{\pn}\}\phii(y)P\zajod(\dei y)\,dy.
\eal
\eeqs
For $i<j$ (fast scaling) we estimate, using \eqref{eq:efscaling}:
\beqs
\bal
\pn&\int_{\Bi/\dei}|y|^{\ai-2}e^{\vi(y)}\{1+\frac{\Dai(y)}{\pn}\}\phii(y)P\zajod(\dei y)\,dy\\
=&\pn\int_{\Bi/\dei}|y|^{\ai-2}e^{\vi(y)}\{1+\frac{\Dai(y)}{\pn}\}\phii(y)(2+O(\frac{\dei|y|}{\dej})^{\aj})\,dy
=2\sis+o(1).
\eal
\eeqs
For $i>j$ (slow scaling) we estimate, using \eqref{eq:efscaling}:
\beqs
\bal
\pn&\int_{\Bi/\dei}|y|^{\ai-2}e^{\vi(y)}\{1+\frac{\Dai(y)}{\pn}\}\phii(y) P\zajod(\dei y)\,dy\\
=&\pn\int_{\Bi/\dei}|y|^{\ai-2}e^{\vi(y)}\{1+\frac{\Dai(y)}{\pn}\}\phii(y)O((\frac{\dej}{\dei|y|})^{\aj}+\dej^{\aj})\,dy=o(1).
\eal
\eeqs
Hence, ($S_j^1$) is established.
\par
Proof of ($S_j^2$).
Testing problem~\eqref{eq:phi1} by $P\Udaj$, integrating by parts and recalling problem~\eqref{eq:PU}, we obtain the identity:
\beq
\label{eq:sjsid2}
\int_\Om|x|^{\aj-2}e^{\Udaj}\phin=\sum_{i=1}^k\int_{\Bi}|x|^{\ai-2}e^{\Udai}\{1+\frac{\Dadi}{\pn}\}\phin P\Udaj
-\int_\Om\thn P\Udaj.
\eeq
We note that $\int_\Om|x|^{\ai-2}e^{\Udai}\phin=o(1)$ by the first equation in \eqref{eq:zajoid}, and $\int_\Om\thn P\Udai=o(1)$
by decay rate of $\thn$.
By the change of variables $x=\dei y\in\Bi$, we obtain
\beqs
\bal
\int_{\Bi}|x|^{\ai-2}e^{\Udai}&\{1+\frac{\Dadi}{\pn}\}\phin P\Udaj(x)\,dx\\
=&\int_{\Bi/\dei}|y|^{\ai-2}e^{\vi(y)}\{1+\frac{\Dai(y)}{\pn}\}\phii(y)P\Udaj(\dei y)\,dy.
\eal
\eeqs
We estimate the integral above, using Lemma~\ref{lem:modbubproj}.
For $i=j$ (natural scaling), we estimate:
\beqs
\bal
\int_{\Bi/\dei}&|y|^{\ai-2}e^{\vi(y)}\{1+\frac{\Dai(y)}{\pn}\}\phii(y)P\Udaj(\dei y)\,dy\\
=&\int_{\Bj/\dej}|y|^{\aj-2}e^{\vj(y)}\{1+\frac{\Daj(y)}{\pn}\}\phij(y)\times\\
&\qquad\qquad\times\{\vj(y)-\ln(2\aj^2)-2\aj\ln\dej+4\pi\aj h(0)+O(|\dej y|)+O(\dej^{\aj})\}\,dy\\
=&\int_{\Bj/\dej}|y|^{\aj-2}e^{\vj(y)}\{1+\frac{\Daj(y)}{\pn}\}\phij(y)\vj(y)\,dy\\
&\qquad\qquad-2\ai\ln\dej\int_{\Bj/\dej}|y|^{\aj-2}e^{\vj(y)}\{1+\frac{\Daj(y)}{\pn}\}\phij(y)\,dy+o(1).
\eal
\eeqs
%Using (4.30) in  \cite{gropi}, namely: 
%\beqs
%\int_{\rr^2}|y|^{\ai-2}e^{\vi(y)}\zaio(y)\ln(1+|y|^{\ai})^2\,dy=-4\pi\ai,
%\qquad i=1,2,\ldots,k,
%\eeqs
In view of the second equation in \eqref{eq:zajoid}
we have
\beqs
\bal
\int_{\Bj/\dej}&|y|^{\aj-2}e^{\vj(y)}\{1+\frac{\Daj(y)}{\pn}\}\phij(y)\vj(y)\,dy\\
=&\gj\int_{\rr^2}|y|^{\aj-2}e^{\vj(y)}\zajo(y)\vj(y)\,dy+o(1)
=4\pi\aj\gj+o(1).
\eal
\eeqs
We deduce that if $i=j$, then
\beqs
\bal
\int_{\Bj/\dej}|y|^{\aj-2}e^{\vj(y)}\{1+\frac{\Daj(y)}{\pn}\}\phij(y)P\Udai(\dej y)\,dy
=4\pi\aj\gj+2\aj\bj\sjs+o(1).
\eal
\eeqs
For $i<j$ (fast scaling):
\beqs
\bal
\int_{\Bi/\dei}&|y|^{\ai-2}e^{\vi(y)}\{1+\frac{\Dai(y)}{\pn}\}\phii(y)P\Udaj(\dei y)\,dy\\
=&\int_{\Bi/\dei}|y|^{\ai-2}e^{\vi(y)}\{1+\frac{\Dai(y)}{\pn}\}\phii(y)\times\\
\times&\left\{-2\aj\ln\dej+4\pi\aj h(0)+O((\frac{\dei|y|}{\dej})^{\aj}+|\dei y|+\dej^{\aj})\right\}\,dy
=2\aj\bj\sis+o(1).
\eal
\eeqs
For $i>j$ (slow scaling) we have
\beqs
\bal
\int_{\Bi/\dei}&|y|^{\ai-2}e^{\vi(y)}\{1+\frac{\Dai(y)}{\pn}\}\phii(y)P\Udaj(\dei y)\,dy\\
=&\int_{\Bi/\dei}|y|^{\ai-2}e^{\vi(y)}\{1+\frac{\Dai(y)}{\pn}\}\phii(y)\times\\
\times&\left\{2\aj\ln\frac{1}{|y|}-2\aj\ln\dei+4\pi\aj h(0)
+O((\frac{\dej}{\dei|y|})^{\aj}+|\dei y|+\dej^{\aj})\right\}\,dy\\
=&2\aj\int_{\Bi/\dei}|y|^{\ai-2}e^{\vi(y)}\{1+\frac{\Dai(y)}{\pn}\}\phii(y)\ln\frac{1}{|y|}\,dy
+2\aj\bi\sis+o(1).
\eal
\eeqs
Using the third equation in \eqref{eq:zajoid}
%\beqs
%\int_{\rr^2}|y|^{\ai-2}e^{\vi(y)}\zaio(y)\ln|y|\,dy=-4\pi,
%\eeqs
we derive:
\beqs
\bal
\int_{\Bi/\dei}|y|^{\ai-2}e^{\vi(y)}\{1+\frac{\Dai(y)}{\pn}\}\phii(y)\ln\frac{1}{|y|}\,dy
%=&-\gj\int_{\rr^2}|y|^{\aj-2}e^{\vj(y)}\zajo(y)\ln|y|\,dy+o(1)
=2\pi\gi+o(1).
\eal
\eeqs
It follows that for $i>j$:
\beqs
\int_{\Bi/\dei}|y|^{\ai-2}e^{\vi(y)}\{1+\frac{\Dai(y)}{\pn}\}\phii(y)P\Udaj(\dei y)\,dy
=2\aj\{2\pi\gi+\bi\sis\}+o(1).
\eeqs
Inserting into \eqref{eq:sjsid2} we obtain
\beqs
4\pi\aj\gj+2\aj\bj\sjs+2\aj\bj\sum_{i<j}\sis+2\aj\sum_{i>j}\{2\pi\gi+\bi\sis\}=o(1),
\eeqs
that is,
\beqs
\bj\sum_{i\le j}\sis+2\pi\sum_{i\ge j}\gi+\sum_{i>j}\bi\sis=o(1).
\eeqs
Hence, ($S_j^2$) is established.
\end{proof}

%%%%%%%%%%%%%%%%%%%%%%%%%%%%%%%%%%%%%%%%%%%%%%%%%%%%%%%%%%%%%%%%%%%%%%%%%%%%%%%%%%%%%%%%%%%%%%%%%%%%%%%%%%%%%%%%%%%%%%%%%
%%%%%%%%%%%%%%%%%%%%%%%%%%%%%%%%%%%%%%%%%%%%%%%%%%%%%%%%%%%%%%%%%%%%%%%%%%%%%%%%%%%%%%%%%%%%%%%%%%%%%%%%%%%%%%%%%%%%%%%%%
\section{Linearized operator: uniform vanishing on shrinking rings}
%%%%%%%%%%%%%%%%%%%%%%%%%%%%%%%%%%%%%%%%%%%%%%%%%%%%%%%%%%%%%%%%%%%%%%%%%%%%%%%%%%%%%%%%%%%%%%%%%%%%%%%%%%%%%%%%%%%%%%%%%
%%%%%%%%%%%%%%%%%%%%%%%%%%%%%%%%%%%%%%%%%%%%%%%%%%%%%%%%%%%%%%%%%%%%%%%%%%%%%%%%%%%%%%%%%%%%%%%%%%%%%%%%%%%%%%%%%%%%%%%%%
\subsection{Evaluation of $\Ia$}
An easy computation shows that $w$  is a radial solution to  
\begin{equation}\label{l1}
\Delta w+2\alpha^2{|y|^{\alpha-2}\over(1+|y|^\alpha)^2}w=2\alpha^2{|y|^{\alpha-2}\over(1+|y|^\alpha)^2}f(|y|)\ \hbox{in}\ \mathbb R^2
\end{equation}
if and only if the function $\tilde w(y)=w\(|y|^{2\over\alpha}\)$ is a radial  solution to
 \begin{equation}\label{l2}
\Delta \tilde w+ {8\over(1+|y|^2)^2}\tilde w= {8\over(1+|y|^2)^2}\tilde f(|y|)\ \hbox{in}\ \mathbb R^2
\end{equation}
where
$\tilde f(|y|):=f\(|y|^{2\over\alpha}\)$.
Now let 
\begin{equation}\label{effe}
\fa(y)={1\over2}\(\underbrace{\ln {2\alpha^2\over (1+|y|^\alpha)^2}+(\alpha-2)\ln|y|}_{:=\mVa} |\)^2.
\end{equation}
Then,
$$\tVa(y)=  \ln {2\alpha^2\over (1+|y|^2)^2}+  {2(\alpha-2)\over\alpha}\ln|y|$$
Set
$\Da=w-\mVa-\frac12\mVa^2$
and
\beqs
\za(y)=\frac{1-|y|^\al}{1+|y|^\al}.
\eeqs
Our aim in this section is to evaluate the integrals
\beq
\label{def:Ia}
\Ia:=\int_{\rr^2}e^{\Va(y)}z_{\al}^2(y)\Da(y)\,dy
=\int\limits_{\mathbb R^2}2\alpha^2{|y|^{\alpha-2}(1-|y|^\alpha)^2\over(1+|y|^\alpha)^4}\Da(y)\,dy.
\eeq
Indeed, we establish the following.
\begin{lemma}
\label{lemma:Ia}
There holds
\beq
\label{eq:Ia}
\Ia=8\pi \(- \ln(2\alpha^2) +\frac{3\alpha-2}{\alpha }\)<0
\eeq
for all $\al\ge2$.
\end{lemma}
\begin{proof}
By change of variables, we have
\begin{equation}\label{I}
\begin{aligned}
\Ia=&\int\limits_{\mathbb R^2}2\alpha^2{|y|^{\alpha-2}(1-|y|^\alpha)^2\over(1+|y|^\alpha)^4}\Da(y)\,dy\\
=&                                                                                                                                                                                                                  2\pi\int\limits_0^\infty 2\alpha^2{r^{\alpha-1}(1-r^\alpha)^2\over(1+r^\alpha)^4}\Da(r)dr\\
=&                                                                                                                                                                                                                             8\pi\alpha\int\limits_0^\infty s{(1-s^2)^2\over(1+s^2)^4}\Da\(s^{2\over\alpha}\)ds\\
=&                                                                                                                                                                                                                             8\pi\alpha\int\limits_0^\infty s{(1-s^2)^2\over(1+s^2)^4} \tDa\(s \)ds\\
=&                                                                                                                                                                                                                             8\pi\alpha\int\limits_0^\infty s{(1-s^2)^2\over(1+s^2)^4} \(\tilde w\(s \)-\tVa(s)-\frac12\tVa^2(s)\)ds\\
=:&\mathcal I_\al'+\mathcal I_\al''.
\end{aligned}\end{equation}
We compute
\begin{equation}\label{I1}
\begin{aligned}
\mathcal I_\al'&=-8\pi\alpha\int\limits_0^\infty s{(1-s^2)^2\over(1+s^2)^4}  \tilde v (s) ds\\
&=-8\pi\alpha\int\limits_0^\infty s{(1-s^2)^2\over(1+s^2)^4} \( \ln {2\alpha^2\over (1+s^2)^2}+  {2(\alpha-2)\over\alpha}\ln s\)ds\\
&=-8\pi\alpha\(\frac16\ln(2\alpha^2)-\frac 49\)
\end{aligned}
\end{equation}
because
$$\int\limits_0^\infty s{(1-s^2)^2\over(1+s^2)^4} ds=\frac16$$
$$\int\limits_0^\infty s{(1-s^2)^2\over(1+s^2)^4} \ln (1+s^2)ds=\frac29$$
$$\int\limits_0^\infty s{(1-s^2)^2\over(1+s^2)^4} \ln sds=0.$$
Moreover, taking into account that $\tilde w$ solves
$$\tilde w''+{1\over r}\tilde w'+{8\over(1+r^2)^2}\tilde w={4\over(1+r^2)^2}\(\ln {2\alpha^2\over (1+r^2)^2}+  {2(\alpha-2)\over\alpha}\ln r\)^2=:\Fa(s)$$
In view of Lemma~\ref{lem:CI}, 
we have the representation
$$\tilde w(r)={1-r^2\over 1+r^2}\(\phi_{\Fa}(1) {r\over r-1}+\int\limits_0^r  {\phi_{\Fa}(s)-\phi_{\Fa}(1)\over (s-1)^2}ds\)$$
and
$$\phi_{\Fa}(s)=\(1+s^2\over 1-s^2\)^2{(s-1)^2\over s}\int\limits_0^s t{1-t^2\over 1+t^2}{\Fa}(t)dt=
\(1+s^2\over 1+s\)^2{1\over s}\int\limits_0^s t{1-t^2\over 1+t^2}{\Fa}(t)dt.$$
We know that
$$Z={1-|y|^2\over 1+|y|^2}$$
solves
$$-\Delta Z=  {8\over (1+|y|^2)^2} Z= 8{1-|y|^2\over (1+|y|^2)^3} $$
Then $Z^2$ solves
$$\nabla Z={-2y\over 1+|y|^2}-2y{1-|y|^2\over (1+|y|^2)^2}={-4y\over (1+|y|^2)^2}$$
$$|\nabla Z|^2= {16|y|^2\over (1+|y|^2)^4} $$
$$\Delta Z^2=2Z\Delta Z+2|\nabla Z |^2=-{16(1-|y|^2)^2\over (1+|y|^2)^4} +{32|y|^2\over (1+|y|^2)^4}= -16{ 1+|y|^4-4|y|^2\over (1+|y|^2)^4}$$
\begin{equation}\label{I3}
\begin{aligned}
\Ia''&=8\pi\alpha\int\limits_0^\infty s{(1-s^2)^2\over(1+s^2)^4} \(\tilde w\(s\)-\frac12 \tVa^2 (s)\)\,ds\\
 &=4 \alpha\int\limits_{\mathbb R^2}  {(1-|y|^2)^2\over(1+|y|^2)^4} \(\tilde w\(y\)-\frac12 \tVa^2 (y)\)\,dy\\
&=-\frac12 \alpha\int\limits_{\mathbb R^2}  {Z^2(y)} \Delta \tilde w(y)\,dy\\
&=-\frac12 \alpha\int\limits_{\mathbb R^2}  \Delta {Z^2(y)}  \tilde w(y)\,dy\\
&=8 \alpha\int\limits_{\mathbb R^2} { 1+|y|^4-4|y|^2\over (1+|y|^2)^4}  \tilde w(y)\,dy\\
&=16\pi \alpha\int\limits_0^\infty r{ 1+r^4-4r^2\over (1+r^2)^4}  \tilde w(r) dr\\
&=16\pi \alpha\int\limits_0^\infty r{ (1-r^2)(1+r^4-4r^2)\over (1+r^2)^5} \(\phi_{\Fa}(1) {r\over r-1}+\int\limits_0^r  {\phi_{\Fa}(s)-\phi_{\Fa}(1)\over (s-1)^2}ds\)
\\
&=16\pi \alpha\left[-\phi_{\Fa}(1)\underbrace{\int\limits_0^\infty r^2{(1+r)(1+r^4-4r^2)\over  (1+r^2)^5} dr}_{= {\pi\over 128}} 
+\int\limits_0^\infty r{ (1-r^2)(1+r^4-4r^2)\over (1+r^2)^5}  \int\limits_0^r  {\phi_{\Fa}(s)-\phi_{\Fa}(1)\over (s-1)^2}ds \right]\\
&=16\pi \alpha \left[-{\pi\over 128}\phi_{\Fa}(1)+\int\limits_0^\infty r{ (1-r^2)(1+r^4-4r^2)\over (1+r^2)^5}  \int\limits_0^r  {\phi_{\Fa}(s)-\phi_{\Fa}(1)\over (s-1)^2}ds \right]\\
&=  16\pi \alpha\left[-{\pi\over 128}\phi_{\Fa}(1)-\int\limits_0^\infty {(r^2 - r^4 + r^6)\over  2 (1 + r^2)^4 }{\phi_{\Fa}(r)-\phi_{\Fa}(1)\over (r-1)^2}dr\right]\\
&\\ &=16\pi \alpha\left[-{\pi\over 128}\phi_{\Fa}(1)\right]\\
&+16\pi \alpha\Big\{-{1 - 2 r^2\over  8 (-1 + r^4)}\int\limits_0^r t{1-t^2\over 1+t^2}{\Fa}(t)dt-\int\limits_0^r
t{1 - 2 t^2\over  8  (1+t^2)^2}{\Fa}(t)dt \\
&\qquad\qquad+\phi_{\Fa}(1)\left[ {r (1 + r)^3 (1 - 4r + r^2) \over 64 (-1 + r) (1 +r^2)^3} + {1\over 64} \arctan r\right]\Big\}_{r=0}^{r=\infty}\\
&=16\pi \alpha[-{\pi\over 128}\phi_{{\Fa}}(1)-\int\limits_0^\infty
t{1 - 2 t^2\over  8  (1+t^2)^2}{\Fa}(t)dt+ \phi_{\Fa}(1) {\pi\over128} ]\\
&=-2\pi \alpha \int\limits_0^\infty
t{1 - 2 t^2\over    (1+t^2)^2}{\Fa}(t)dt 
 \end{aligned}\end{equation}
 because
$$\begin{aligned}
&  \int   {(r^2 - r^4 + r^6)\over  2 (r-1)^2(1 + r^2)^4 } dr= {r (1 + r)^3 (1 - 4r + r^2) \over 64 (-1 + r) (1 +r^2)^3} + {1\over 64} \arctan r
 \end{aligned}$$
 and
 $$\begin{aligned}
&\int\limits  {(r^2 - r^4 + r^6)\over  2 (r-1)^2 (1 + r^2)^4 } \phi_{\Fa}(r) dr \\
&=\int\limits {(r^2 - r^4 + r^6)\over  2 (r-1)^2 (1 + r^2)^4 }\(1+r^2\over 1-r^2\)^2{(r-1)^2\over r}\int\limits_0^r t{1-t^2\over 1+t^2}{\Fa}(t)dt\\
&=\int\limits  {(r^2 - r^4 + r^6)\over  2r  (1 - r^4)^2  } \int\limits_0^r t{1-t^2\over 1+t^2}\Fa(t)dt\\
&={1 - 2 r^2\over  8 (-1 + r^4)}\int\limits_0^r t{1-t^2\over 1+t^2}\Fa(t)dt-\int
{1 - 2 r^2\over  8 (-1 + r^4)} r  {1-r^2\over 1+r^2}\Fa(r)dr \\
&={1 - 2 r^2\over  8 (-1 + r^4)}\int\limits_0^r t{1-t^2\over 1+t^2}\Fa(t)dt+\int
r{1 - 2 r^2\over  8  (1+r^2)^2}\Fa(r)dr \end{aligned}$$

We have to compute
$$\begin{aligned}
&\int\limits_0^\infty
t{1 - 2 t^2\over     (1+t^2)^2}\Fa(t)dt=4\int\limits_0^\infty
t{1 - 2 t^2\over    (1+t^2)^4} \(\ln {2\alpha^2\over (1+t^2)^2}+  {2(\alpha-2)\over\alpha}\ln t
\)^2dt\\
&=4(\ln2\alpha^2)^2\underbrace{\int\limits_0^\infty
t{1 - 2 t^2\over    (1+t^2)^4} dt}_{=0}+16\underbrace{\int\limits_0^\infty
t{1 - 2 t^2\over    (1+t^2)^4} (\ln(1+t^2))^2dt}_{=-\frac5{36}}+ 16{(\alpha-2)^2\over\alpha^2}\underbrace{\int\limits_0^\infty
t{1 - 2 t^2\over    (1+t^2)^4}(\ln t)^2 dt}_{= \frac1{8} }\\&
-16 (\ln2\alpha^2) \underbrace{\int\limits_0^\infty
t{1 - 2 t^2\over    (1+t^2)^4} \ln(1+t^2) dt}_{=-\frac1{12}}+ 16{(\alpha-2) \over\alpha } (\ln2\alpha^2) \underbrace{\int\limits_0^\infty
t{1 - 2 t^2\over    (1+t^2)^4}(\ln t)  dt}_{=-\frac1{8}}\\ &-32{(\alpha-2) \over\alpha } \underbrace{  \int\limits_0^\infty
t{1 - 2 t^2\over    (1+t^2)^4} (\ln t)\ln(1+t^2)}_{=-\frac1{16}}  dt\\
&=-\frac{20}9+\(\frac43-2\frac{\alpha-2}\alpha\)\ln(2\alpha^2)+2\frac{(\alpha-2)^2}{\alpha^2}+2\frac{\alpha-2}\alpha
\end{aligned}$$
Therefore,
$$\begin{aligned}
 I(\alpha) &=8\pi\alpha\( -\frac16\ln(2\alpha^2)+\frac 49+\frac59-\(\frac13- \frac{\alpha-2}{2\alpha}\)\ln(2\alpha^2)- \frac{(\alpha-2)^2}{2\alpha^2}-\frac{\alpha-2}{2\alpha}\)\\
&=8\pi\alpha\(-{\ln(2\alpha^2)\over\alpha}+\frac{3\alpha-2}{\alpha^2}\)=
8\pi \(- \ln(2\alpha^2) +\frac{3\alpha-2}{\alpha }\)
\end{aligned}$$
Finally, we have
$${I'(\alpha)}=8\pi \(-\frac2\alpha+\frac2{\alpha^2}\)<0\ \hbox{for any}\ \alpha\ge2$$
$$ $$ 
and so
$$I(\alpha)<I(2)=8\pi(-\ln8 +2)<0\ \hbox{for any}\ \alpha\ge2,$$
as asserted.
\end{proof}
%%%%%%%%%%%%%%%%%%%%%%%%%%%%%%%%%%%%%%%%%%%%%%%%%%%%%%%%%%%%%%%%%%%%%%%%%%%%%%%%%%%%%%%%%%%%%%%%%%%%%%%%%%%%%%%%
In view of Lemma~\ref{lem:linrel} we set 
\beqs
(\underline\si_n^*,\underline\ga_n)^T:=(\si_{1,n}^*,\ga_1^n,\si_{2,n}^*,\ga_{2,n},\ldots,\si_{k,n}^*,\ga_{k,n})^T.
\eeqs
Then, system~$(S_j^1)-(S_j^2)$ may be written in the form
\beqs
\mathcal M_{k,n}{\underline\si_n^*\choose \underline\ga_n}=o(1),
\eeqs
where $\mathcal M_{k,n}$ is the $2k\times2k$ matrix defined by
\beq
\mathcal M_{k,n}:=\left(\begin{matrix}
1&\mathcal I_{\al_1}&0&0&0&0&\dots&\dots&0&0\\
b_1& 2\pi&b_2&4\pi&b_3&4\pi&\dots&\dots&b_k&4\pi\\
2&0&1&\mathcal I_{\al_2}&0&0&\dots&\dots&0&0\\
b_2&0&b_2 &2\pi&b_3&4\pi&\dots&\dots&b_k&4\pi\\
2&0&2&0&1&\mathcal I_{\al_3}&\dots&\dots&0&0\\
b_3&0&b_3 & 0&b_3&2\pi&\dots&\dots&b_k&4\pi\\
\vdots&\vdots&\vdots & \vdots&\vdots & \vdots&\ddots&\ddots&\vdots&\vdots\\
\vdots&\vdots&\vdots & \vdots&\vdots & \vdots&\ddots&\ddots&\vdots&\vdots\\
2&0&2&0&0&0&\dots&\dots&1&\mathcal I_{\ak}\\
b_k&0&b_k & 0&b_k&0&\dots&\dots&b_k&2\pi
\end{matrix}\right).
\eeq
\begin{lemma}
\label{lem:Mkinvert}
There exists $c_0>0$ such that $|\mathcal M_{k,n}|\ge c_0>0$ for all $n\in\mathbb N$.
\end{lemma}
\begin{proof}
It is equivalent to prove that
$|\widetilde{\mathcal M}_{k,n}|\ge c_0>0$ for all $n\in\mathbb N$,
where $\widetilde{\mathcal M}_{k,n}$ is the $2k\times2k$ matrix  defined by
$$\widetilde{\mathcal M}_{k,n}=\left(\begin{matrix}
1&-a_1&0&0&0&0&\dots&\dots&0&0\\
b_1& 1&b_2&2&b_3&2&\dots&\dots& b_k&2\\
2&0&1&-a_2&0&0&\dots&\dots&0&0\\
b_2&0&b_2 & 1& b_3&2&\dots&\dots&b_k&2\\
2&0&2&0&1&-a_3&\dots&\dots&0&0\\
b_3&0&b_3 & 0&b_3&1&\dots&\dots& b_k&2\\
\vdots&\vdots&\vdots & \vdots&\vdots & \vdots&\ddots&\ddots&\vdots&\vdots\\
\vdots&\vdots&\vdots & \vdots&\vdots & \vdots&\ddots&\ddots&\vdots&\vdots\\
2&0&2&0&2&0&\dots&\dots&1&-a_k\\
b_k&0& b_k & 0& b_k&0&\dots&\dots&b_k&1\\
\end{matrix}\right),$$
whose entries satisfy $a_i,b_i>0$, for any $i=1,2,\ldots,k$, and $b_i\ge b_{i+1}$, $i=1,2,\ldots,k-1$.
In $\widetilde{\mathcal M}_{k,n}$ we replace row $2j$ by the difference between row $2j$ and row $(2j+2)$, $j=1,2,\ldots,k-1$.
Thus, we obtain the matrix $\widetilde{\widetilde{\mathcal M}}_{k,n}$:
$$\widetilde{\widetilde{\mathcal M}}_{k,n}=\left(\begin{matrix}
1&-a_1&0&0&0&0&\dots&\dots&0&0\\
c_1& 1&0&1&0&0&\dots&\dots&0&0\\
2&0&1&-a_2&0&0&\dots&\dots&0&0\\
c_2&0&c_2 & 1& 0&1&\dots&\dots&0&0\\
2&0&2&0&1&-a_3&\dots&\dots&0&0\\
c_3&0&c_3 & 0&c_3&1&\dots&\dots& 0&0\\
\vdots&\vdots&\vdots & \vdots&\vdots & \vdots&\ddots&\ddots&\vdots&\vdots\\
\vdots&\vdots&\vdots & \vdots&\vdots & \vdots&\ddots&\ddots&\vdots&\vdots\\
2&0&2&0&2&0&\dots&\dots&1&-a_k\\
c_k&0& c_k & 0& c_k&0&\dots&\dots&c_k&1\\
\end{matrix}\right),$$
where $c_i=b_i-b_{i+1}$, $i=1,2,\ldots,k-1$, $c_k=b_k$ satisfy $c_i>0$ for any $i$.
In $\widetilde{\widetilde{\mathcal M}}_{k,n}$
we replace
column $(2j-1)$ by the difference between column $(2j-1)$ and column $(2j+1)$, $j=1,2,\ldots,k-1$.
Thus, we obtain the matrix 
$$\mathcal A_k:=\left(\begin{matrix}
1&-a_1&0&0&0&0&\dots&\dots&0&0\\
c_1& 1&0&1&0&0&\dots&\dots&0&0\\
1&0&1&-a_2&0&0&\dots&\dots&0&0\\
0&0&c_2 & 1&0&1&\dots&\dots&0&0\\
0&0&1&0&1&-a_3&\dots&\dots&0&0\\
0&0&0 & 0&c_3&1&\dots&\dots&0&0\\
\vdots&\vdots&\vdots & \vdots&\vdots & \vdots&\ddots&\ddots&\vdots&\vdots\\
\vdots&\vdots&\vdots & \vdots&\vdots & \vdots&\ddots&\ddots&\vdots&\vdots\\
0&0&0&0&0&0&\dots&\dots&1&-a_n\\
0&0&0& 0&0&0&\dots&\dots&c_n&1\\
\end{matrix}\right),$$
and $|\widetilde{\widetilde{\mathcal M}}_{k,n}|=|\widetilde{\mathcal M}_{k,n}|=|\mathcal A_k|$.
A recurrence argument shows that
$$|\mathcal A_k|>0\ \hbox{for any}\ n\ge1.$$
Indeed, let us denote by $\mathcal A_j$ the submatrix of $\mathcal A_k$ obtained by deleting the
first $2(j-1)$ rows and the first $2(j-1)$ columns, namely
$$\mathcal A_j:=\left(\begin{matrix}
1&-a_j&0&0\dots&\dots&0&0\\
c_j& 1&0&1\dots&\dots&0&0\\
1&0&1&-a_{j+1}\dots&\dots&0&0\\
0&0&c_{j+1} & 1\dots&\dots&0&0\\
\vdots&\vdots&\vdots & \vdots&\ddots&\vdots&\vdots\\
\vdots& \vdots&\vdots & \vdots&\ddots&\vdots&\vdots\\
0&0&0&0&\dots&1&-a_k\\
0&0&0&0&\dots&c_k&1\\
\end{matrix}\right).$$
Similarly, we denote by $\mathcal B_j$ the submatrix of $\mathcal A_k$ obtained by deleting the first $2j-1$ rows,
the first $2(j-1)$ columns and the $2j$th column,
namely
$$\mathcal B_j:=\left(\begin{matrix}
c_j&0&1&0\dots&\dots&0&0\\
1& 1&-a_{j+1}&0\dots&\dots&0&0\\
0&c_{j+1}&1&0\dots&\dots&0&1\\
0&1&0 & \dots&\dots&0&0\\
\vdots&\vdots&\vdots & \vdots&\ddots&\vdots&\vdots\\
\vdots& \vdots&\vdots & \vdots&\ddots&\vdots&\vdots\\
0&0&0&0&\dots&1&-a_k\\
0&0&0&0&\dots&c_k&1\\
\end{matrix}\right).$$
With this notation, we readily check that
\beq
\label{eq:AjBj}
\left\{
\bal
|\mathcal A_j|=&|\mathcal A_{j+1}|+a_j|\mathcal B_j|\\
|\mathcal B_j|=&c_j|\mathcal A_{j+1}|+|\mathcal B_{j+1}|.
\eal
\right.
\eeq
%%%%%%%%%%%%%%%%%%%%%%%%%%%%%%%%%%%%%%%%%%%%%%%%%%%%%%%%%%%%%%%%%%%%%%%%%%%%%%%%%%%%%%%%%%%%%%%%%%%%%%%%%%%%%%%%%%%%%%%
Indeed we have
$$\begin{aligned}&|\mathcal M_n|=\left|\begin{matrix}
1&-a_1&0&0&0&0&\dots&\dots&0&0\\
c_1& 1&0&1&0&0&\dots&\dots&0&0\\
1&0&1&-a_2&0&0&\dots&\dots&0&0\\
0&0&c_2 & 1&0&1&\dots&\dots&0&0\\
0&0&1&0&1&-a_3&\dots&\dots&0&-0\\
0&0&0 & 0&c_3&1&\dots&\dots&0&0\\
\vdots&\vdots&\vdots & \vdots&\vdots & \vdots&\ddots&\ddots&\vdots&\vdots\\
\vdots&\vdots&\vdots & \vdots&\vdots & \vdots&\ddots&\ddots&\vdots&\vdots\\
0&0&0&0&0&0&\dots&\dots&1&-a_n\\
0&0&0& 0&0&0&\dots&\dots&c_n&1\\
\end{matrix}\right|\\
&=\left|\begin{matrix}
    1&-a_2&0&0&\dots&\dots&0&0\\
  c_2 & 1&0&1&\dots&\dots&0&0\\
 1&0&1&-a_3&\dots&\dots&1&-a_n\\
  0 & 0&c_3&1&\dots&\dots&0&0\\
  \vdots & \vdots&\vdots & \vdots&\ddots&\ddots&\vdots&\vdots\\
  \vdots & \vdots&\vdots & \vdots&\ddots&\ddots&\vdots&\vdots\\
 0&0&0&0&\dots&\dots&1&-a_n\\
  0& 0&0&0&\dots&\dots&c_n&1\\
\end{matrix}\right|+a_1
\left|\begin{matrix}
 c_1& 0&1&0&0&\dots&\dots&0&0\\
1&1&-a_2&0&0&\dots&\dots&0&0\\
0&c_2 & 1&0&1&\dots&\dots&0&0\\
0&1&0&1&-a_3&\dots&\dots&1&-a_n\\
0&0 & 0&c_3&1&\dots&\dots&0&0\\
\vdots&\vdots & \vdots&\vdots & \vdots&\ddots&\ddots&\vdots&\vdots\\
\vdots&\vdots & \vdots&\vdots & \vdots&\ddots&\ddots&\vdots&\vdots\\
0&0&0&0&0&\dots&\dots&1&-a_n\\
0&0& 0&0&0&\dots&\dots&c_n&1\\
\end{matrix}\right|\\
&=\left|\begin{matrix}
    1&-a_2&0&0&\dots&\dots&0&0\\
  c_2 & 1&0&1&\dots&\dots&0&0\\
 1&0&1&-a_3&\dots&\dots&1&-a_n\\
  0 & 0&c_3&1&\dots&\dots&0&0\\
  \vdots & \vdots&\vdots & \vdots&\ddots&\ddots&\vdots&\vdots\\
  \vdots & \vdots&\vdots & \vdots&\ddots&\ddots&\vdots&\vdots\\
 0&0&0&0&\dots&\dots&1&-a_n\\
  0& 0&0&0&\dots&\dots&c_n&1\\
\end{matrix}\right|\\ &+a_1\(c_1
\left|\begin{matrix}
  1&-a_2&0&0&\dots&\dots&0&0\\
 c_2 & 1&0&1&\dots&\dots&0&0\\
 1&0&1&-a_3&\dots&\dots&1&-a_n\\
 0 & 0&c_3&1&\dots&\dots&0&0\\
 \vdots & \vdots&\vdots & \vdots&\ddots&\ddots&\vdots&\vdots\\
 \vdots & \vdots&\vdots & \vdots&\ddots&\ddots&\vdots&\vdots\\
 0&0&0&0&\dots&\dots&1&-a_n\\
 0& 0&0&0&\dots&\dots&c_n&1\\
\end{matrix}\right|  +
\left|\begin{matrix}
  c_2 & 0&1&\dots&\dots&0&0\\
 1&1&-a_3&\dots&\dots&0&0\\
 0 &  c_3&1&\dots&\dots&0&0\\
 \vdots &  \vdots & \vdots&\ddots&\ddots&\vdots&\vdots\\
\vdots&  \vdots&\vdots & \ddots&\ddots&\vdots&\vdots\\
 0& 0&0&\dots&\dots&1&-a_n\\
 0&  0&0&\dots&\dots&c_n&1\\
\end{matrix}\right|\)\\
&=|\mathcal M_{n-1}|+a_1\(c_1|\mathcal M_{n-1}|+c_2|\mathcal M_{n-2}|+c_3|\mathcal M_{n-3}|+\dots+\left|\begin{matrix}
  c_{n-1} & 0&1\\
  1&1&-a_n\\
 0&   c_n&1\\
\end{matrix}\right|\)\\
&=|\mathcal M_{n-1}|+a_1\(c_1|\mathcal M_{n-1}|+c_2|\mathcal M_{n-2}|+c_3|\mathcal M_{n-3}|+\dots+c_{n-1}|\mathcal M_{1}|+c_n\)\end{aligned}$$
\end{proof}
%%%%%%%%%%%%%%%%%%%%%%%%%%%%%%%%%%%%%%%%%%%%%%%%%%%%%%%%%%%%%%%%%%%%%%%%%%%%%%%~~~~~
%%%%%%%%%%%%%%%%%%%%%%%%%%%%%%%%%%%%%%%%%%%%%%%%%%%%%%%%%%%%%%%%%%%%%%%%%%%%%%%
\section{Barrier estimate}
%%%%%%%%%%%%%%%%%%%%%%%%%%%%%%%%%%%%%%%%%%%%%%%%%%%%%%%%%%%%%%%%%%%%%%%%%%%%%%%~~~~~
%%%%%%%%%%%%%%%%%%%%%%%%%%%%%%%%%%%%%%%%%%%%%%%%%%%%%%%%%%%%%%%%%%%%%%%%%%%%%%%
Recall from \eqref{def:Lp} that $\Lp$ is the operator defined by
$\Lp\phi=\Delta\phi+\Wp(x)\phi$.
For every $i=1,2,\ldots,k$ let $\tRi\gg1$ be a fixed large constant and for every $i=2,\ldots,k$
let $0<\tri\ll1$ be a fixed small constant.
Let $\tAjj$ denote the shrinking annulus defined by
\beq
\label{def:tAjj}
\tAjj:=
\begin{cases}
B_{\trjj\dejj}\setminus B_{\tRj\dej}&\hbox{if }j=1,2,\ldots,k-1,\\
\Om\setminus B_{\tRk}&\hbox{if }j=k.
\end{cases}
\eeq
The aim of this section is to establish the following result:
\begin{prop}[Barrier estimate]
\label{prop:barrierest}
Suppose that $\Lp\phi=h$ in $\Om$, $\phi=0$ on $\pl\Om$.
Then, there exist suitable constants $0<\trj\ll1$, $\tRj\gg1$,
and $C>0$ independent of $p$ such that
\beq
\label{est:barrier}
\|\phi\|_{L^\infty(\tAjj)}\le C(\|\phi\|_{L^\infty(\pl\tAjj)}+\|h\|_{\rp}).
\eeq
\end{prop}
We begin by showing that the \lq\lq 0-order operator" $\Lpo$ defined by
\beq
\Lpo\phi=\Delta\phi+\ovC\sum_{i=1}^k|x|^{\ai-2}e^{\Udai(x)}\phi
\eeq
satifies a maximum principle in $\tAjj$.
%
%
%Let $\Lpj$ be the operator defined by
%\beq
%\label{def:Ljo}
%\Lpj\phi=\Delta\phi+(K_j|x|^{\aj-2}e^{\Udaj(x)}+K_{j+1}|x|^{\al_{j+1}-2}e^{U_{\al_{j+1},\de_{j+1}}(x)})\phi,
%\eeq
%where $K_j$, $K_{j+1}>0$ are constants.
%Our aim in this section is to show:
%\begin{prop}
%Suppose
%\beqs
%\bca
%\Lpj\phi=h&\hbox{in }\Om\\
%\phi=0&\hbox{on }\pl\Om.
%\eca
%\eeqs
%Then, for any $j=1,2,\ldots,k-1$, there exist $0<\trj\ll1$, $\tRj\gg1$ and $C_j>0$ independent of $p\to+\infty$
%such that
%\beq
%\label{prop:phinterm}
%\|\phi\|_{L^\infty(B_{\trj\de_{j+1}}\setminus B_{\tRj\dej})}
%\le C_j(\|\phi\|_{L^\infty(\pl B_{\tRj\dej})}+\|\phi\|_{L^\infty(\pl B_{\trj\de_{j+1}})}+\|h\|_{\rp}).
%\eeq
%\end{prop}
%We begin by showing that $\Lpj$ satisfies the maximum principle in $B_{\trj\de_{j+1}}\setminus B_{\tRj\dej}$.
To this end, we recall that
\beqs
\zajo(y)=\frac{1-|y|^{\aj}}{1+|y|^{\aj}}
\eeqs
and we define the functions:
\beq
\label{def:PsujPsoj}
\bal
\Psuj(x):=&-\zajo(\frac{\lamj x}{\dej}),\qquad&&\hbox{for }0<\lamj\ll1,\\
\Psoj(x):=&\zajo(\frac{\Lamj x}{\dej}),\qquad&&\hbox{for }\Lamj\gg1.
\eal
\eeq
\begin{lemma}
\label{lem:Psuj}
Fix $0<\co<1$ and $\uDj>0$. 
Suppose that $\tRj>0$ satisfies
\beq
\label{eq:tRj}
\tRj\ge\max\left\{\frac{1}{\lamj}(\frac{1+\co}{1-\co})^{1/\aj},
\left(\frac{\sqrt{\frac{\uDj}{\co\lamj^{\aj}}}-1}{1-\sqrt{\frac{\uDj\lamj^{\aj}}{\co}}}\right)^{1/\aj}\right\}
\eeq
Then,
\beqs
\left.
\bal
\Delta\Psuj&\le-\uDj|x|^{\aj-2}e^{\Udaj}\\
\Psuj&\ge\co>0
\eal
\right\}
\qquad\qquad\hbox{in }\rr^2\setminus B_{\tRj\dej}.
\eeqs
\end{lemma}
\begin{proof}
Claim~1. There holds $\Psuj(x)\ge\co>0$ if and only if
\beq
\label{eq:Psujfirst}
|x|\ge\frac{\dej}{\lamj}(\frac{1+\co}{1-\co})^{1/\aj}.
\eeq
Indeed, by a straightforward computation, we have 
\beqs
\Psuj(x)=\frac{|\frac{\lamj x}{\dej}|^{\aj}-1}{|\frac{\lamj x}{\dej}|^{\aj}+1}\ge\co
\eeqs
if and only if
\beqs
(1-\co)|\frac{\lamj x}{\dej}|^{\aj}\ge1+\co
\eeqs
and the asserted necessary and sufficient condition \eqref{eq:Psujfirst} readily follows.
\par
Now, we assume that \eqref{eq:Psujfirst} is satisfied.
\par
Claim~2. Suppose that $\Psuj\ge\co$. There holds 
\beq
\label{eq:Claim2Psuj}
\Delta\Psuj(x)\le-\uDj|x|^{\aj-2}e^{\Udaj}
\eeq
if
\beq
\label{eq:Claim2Psujsuffcond}
|x|\ge\left(\frac{\sqrt{\frac{\uDj}{\co\lamj^{\aj}}}-1}{1-\sqrt{\frac{\uDj\lamj^{\aj}}{\co}}}\right)^{1/\aj}\dej.
\eeq
Indeed, in view of \eqref{def:PsujPsoj} we have
\beqs
\bal
\Delta\Psuj=&-(\frac{\lamj}{\dej})^2(\Delta\zajo)(\frac{\lamj}{\dej})
=(\frac{\lamj}{\dej})^2|\frac{\lamj}{\dej}|^{\aj-2}e^{\vj(\lamj x/\dej)}\zajo(\frac{\lamj}{\dej})\\
=&-(\frac{\dej}{\lamj})^{\aj}\frac{2\aj^2|x|^{\aj-2}}{(|x|^{\aj}+(\frac{\dej}{\lamj})^{\aj})^2}\Psuj.
\eal
\eeqs
Thus, we estimate
\beqs
\Delta\Psuj\le-\co(\frac{\dej}{\lamj})^{\aj}\frac{2\aj^2|x|^{\aj-2}}{(|x|^{\aj}+(\frac{\dej}{\lamj})^{\aj})^2}
=-\frac{\co}{\lamj^{\aj}}|x|^{\aj-2}e^{\Udaj(x)}\frac{(|x|^{\aj}+\dej^{\aj})^2}{(|x|^{\aj}+(\frac{\dej}{\lamj})^{\aj})^2}.
\eeqs
It follows that a sufficient condition for \eqref{eq:Claim2Psuj} to hold true is that
\beqs
\frac{\co}{\lamj^{\aj}}\frac{(|x|^{\aj}+\dej^{\aj})^2}{(|x|^{\aj}+(\frac{\dej}{\lamj})^{\aj})^2}\ge\uDj,
\eeqs
equivalently
\beqs
\frac{|x|^{\aj}+\dej^{\aj}}{|x|^{\aj}+(\frac{\dej}{\lamj})^{\aj}}\ge\sqrt{\frac{\uDj\lamj^{\aj}}{\co}},
\eeqs
\beqs
(1-\sqrt{\frac{\uDj\lamj^{\aj}}{\co}})|x|^{\aj}\ge\sqrt{\frac{\uDj\lamj^{\aj}}{\co}}(\frac{\dej}{\lamj})^{\aj}-\dej^{\aj}
=(\sqrt{\frac{\uDj}{\co\lamj^{\aj}}}-1)\dej^{\aj},
\eeqs
and finally
\beqs
|x|^{\aj}\ge\frac{\sqrt{\frac{\uDj}{\co\lamj^{\aj}}}-1}{1-\sqrt{\frac{\uDj\lamj^{\aj}}{\co}}}\dej^{\aj},
\eeqs
from which \eqref{eq:Claim2Psujsuffcond} follows.
Claim~1 and Claim~2 yield the statement of the asserted lemma.
\end{proof}
\begin{lemma}
\label{lem:Psoj}
Fix $0<\co<1$ and $\oDj>0$.
Suppose that $\trj>0$ is such that
\beqs
\trj\le\min\left\{\frac{1}{\Lamj}(\frac{1-\co}{1+\co})^{1/\aj},
\left(\frac{1-\sqrt{\frac{\oDj}{\co\Lamj^{\aj}}}}{\sqrt{\frac{\oDj\Lamj^{\aj}}{\co}}-1}\right)^{1/\aj}\right\}.
\eeqs
Then,
\beqs
\left.
\bal
\Delta\Psoj&\le-\oDj|x|^{\aj-2}e^{\Udaj}\\
\Psoj&\ge\co>0
\eal
\right\}
\qquad\qquad\hbox{in }B_{\trj\dej}.
\eeqs
\end{lemma}
\begin{proof}
Similarly as in the proof of Lemma~\ref{lem:Psuj}, we first establish the following.
\par
Claim~1. There holds $\Psoj(x)\ge\co$ if and only if
\beq
\label{eq:PsojClaim1}
|x|\le\frac{\dej}{\Lamj}(\frac{1-\co}{1+\co})^{1/\aj}.
\eeq
Indeed, in view of \eqref{def:PsujPsoj} we have
\beqs
\Psoj(x)=\frac{1-|\frac{\Lamj x}{\dej}|^{\aj}}{1+|\frac{\Lamj x}{\dej}|^{\aj}}\ge\co
\eeqs
if and only if 
\beqs
(1+\co)|\frac{\Lamj x}{\dej}|^{\aj}\le1-\co
\eeqs
and Claim~1 follows.
\par
Claim~2. Suppose that $\Psoj\ge\co$. Then, there holds 
\beqs
\Delta\Psoj(x)\le-\oDj|x|^{\aj-2}e^{\Udaj(x)}
\eeqs
if
\beq
\label{eq:Psojsecondcond}
|x|\le\left(\frac{1-\sqrt{\frac{\oDj}{\co\Lamj^{\aj}}}}{\sqrt{\frac{\oDj\Lamj^{\aj}}{\co}}-1}\right)^{1/\aj}\dej.
\eeq
Indeed, in view of \eqref{def:PsujPsoj} we have
\beqs
\bal
\Delta\Psoj=&(\frac{\Lamj}{\dej})^2(\Delta\zajo)(\frac{\Lamj x}{\dej})
=-(\frac{\Lamj}{\dej})^2|\frac{\Lamj x}{\dej}|^{\aj-2}e^{\vj(\Lamj x/\dej)}\zajo(\frac{\Lamj x}{\dej})\\
=&-(\frac{\Lamj}{\dej})^{\aj}\frac{2\aj^2|x|^{\aj-2}}{(1+|\frac{\Lamj x}{\dej}|^{\aj})^2}\Psoj(x)
=-(\frac{\dej}{\Lamj})^{\aj}\frac{2\aj^2|x|^{\aj-2}}{(|x|^{\aj}+(\frac{\dej}{\Lamj})^{\aj})^2}\Psoj(x).
\eal
\eeqs
Therefore, we have the estimate
\beqs
\Delta\Psoj\le-\frac{\co}{\Lamj^{\aj}}|x|^{\aj-2}e^{\Udaj(x)}\frac{(|x|^{\aj}+\dej^{\aj})^2}{(|x|^{\aj}+(\frac{\dej}{\Lamj})^{\aj})^2}.
\eeqs
We deduce the sufficient condition
\beqs
\frac{\co}{\Lamj^{\aj}}\frac{(|x|^{\aj}+\dej^{\aj})^2}{(|x|^{\aj}+(\frac{\dej}{\Lamj})^{\aj})^2}\ge\oDj,
\eeqs
from which we derive
\beqs
|x|^{\aj}+\dej^{\aj}\ge\sqrt{\frac{\oDj\Lamj^{\aj}}{\co}}(|x|^{\aj}+(\frac{\dej}{\Lamj})^{\aj})
\eeqs
and 
\beqs
|x|^{\aj}\le\frac{1-\sqrt{\frac{\oDj}{\co\Lamj^{\aj}}}}{\sqrt{\frac{\oDj\Lamj^{\aj}}{\co}}-1}\dej^{\aj}.
\eeqs
This establishes the sufficient condition \eqref{eq:Psojsecondcond}.
\par
Claim~1 and Claim~2 imply the statement of Lemma~\ref{lem:Psoj}.
\end{proof}
\begin{lemma}[Maximum principle property for $\Lpo$ in $\tAjj$]
For any given $\ovC>0$ and $0<\co<1$ there exists a function $\Psij$
and constants $0<\trj\ll1$ and $\tRj\gg1$
such that
\beq
\label{eq:Psij}
\left.
\bal
&\Lpo\Psij<0\\
&k\co\le\Psij\le k
\eal
\right\}\qquad\hbox{in }\tAjj.
\eeq
\end{lemma}
\begin{proof}
Let 
\beqs
\Psij=\sum_{i=1}^j\Psui+\sum_{i=j+1}^k\overline\Psi_i, 
\eeqs
where $\Psui$, and $\Psoi$
are the functions defined in \eqref{def:PsujPsoj}.
We observe that
\beqs
\widetilde R_1\de_1\ll\widetilde r_2\de_2<\widetilde R_2\de_2
\ll\widetilde r_3\de_3<\ldots<\widetilde R_k\de_k\ll1.
\eeqs
Consequently,
\beqs
\bigcap_{i\le j}\{|x|\ge\tRi\dei\}=\{|x|\ge\tRj\dej\},\qquad
\qquad
\bigcap_{i\ge j+1}\{|x|\le\tri\dei\}=\{|x|\le\widetilde r_{j+1}\dejj\}
\eeqs
and therefore
\beqs
\left(\bigcap_{i\le j}\{|x|\ge\tRi\dei\}\right)\bigcap
\left(\bigcap_{i\ge j+1}\{|x|\le\tri\dei\}\right)=\tAjj,
\eeqs
provided that $p$ is sufficiently large.
Consequently, in view of Lemma~\ref{lem:Psuj} and Lemma~\ref{lem:Psoj}, 
we may find $\tRj\gg1$ and $0<\trj\ll1$ such that
\beqs
\left.
\bal
&\Delta\Psij\le-\sum_{i=1}^j\uDi|x|^{\ai-2}e^{\Udai}-\sum_{i=j+1}^k\oDi|x|^{\ai-2}e^{\Udai}\\
&k\co\le\Psij\le k
\eal
\right\}\qquad\hbox{in\ }\tAjj.
\eeqs
It follows that
\beqs
\bal
\Lpo\Psij\le&-\sum_{i=1}^j\uDi|x|^{\ai-2}e^{\Udai}-\sum_{i=j+1}^k\oDi|x|^{\ai-2}e^{\Udai}+\ovC\sum_{i=1}^k|x|^{\ai-2}e^{\Udai}\Psij\\
\le&-\sum_{i=1}^j(\uDi-k\ovC)|x|^{\ai-2}e^{\Udai}-\sum_{i=j+1}^k(\oDi-k\ovC)|x|^{\ai-2}e^{\Udai}.
\eal
\eeqs
Hence, the asserted supersolution property~\eqref{eq:Psij} readily follows by suitable choice of
$\uDi$, $\oDi$.
\end{proof}
%Now, we want to extend the maximum principle property to the \lq\lq full operator" $\Lp$.
%To this end, it is useful to observe that 
%in view of Proposition~\ref{prop:Wpest} we have
%\beq
%\label{eq:LpLpo}
%\Lp\psi\le\Lpo\psi+\ovC e^{-\ovb p}\psi,
%\eeq
%for any $\psi\ge0$.
%\par
\begin{lemma}[Mass decomposition]
\label{lem:massdecomp}
The following estimates hold for any $i=1,2,\ldots,k$:
\begin{align}
\tag{i}
&|x|^{\ai-2}e^{\Udai(x)}\le\frac{2\ai^2}{\tRi^{\ai-\eta}}\frac{\dei^\eta}{|x|^{2+\eta}},
\qquad\hbox{for all }|x|\ge\tRi\dei\\
\tag{ii}
&|x|^{\ai-2}e^{\Udai(x)}\le\frac{2\ai^2\tri^{\ai-2}}{\dei^{2}},
\qquad\hbox{for all }|x|\le\tri\dei.
\end{align}
\end{lemma}
\begin{proof}
Proof of (i). We compute, for $|x|\ge\tRi\dei$:
\beqs
\bal
|x|^{\ai-2}e^{\Udai(x)}=\frac{2\ai^2\dei^{\ai}|x|^{\ai-2}}{(\dei^{\ai}+|x|^{\ai})^2}
\le\frac{2\ai^2\dei^{\ai-\eta}}{|x|^{\ai-\eta}}\,\frac{\dei^\eta}{|x|^{2+\eta}}
\le\frac{2\ai^2}{\tRi^{\ai-\eta}}\,\frac{\dei\eta}{|x|^{2+\eta}},
\eal
\eeqs
where we used $\dei/|x|\le\tRi^{-1}$ in order to derive the last inequality.
\par
Proof of (ii).
We compute, for $|x|\le\tri\dei$:
\beqs
|x|^{\ai-2}e^{\Udai(x)}=\frac{2\ai^2\dei^{\ai}|x|^{\ai-2}}{(\dei^{\ai}+|x|^{\ai})^2}
\le\frac{2\ai^2|x|^{\ai-2}}{\dei^{\ai-2}}\frac{1}{\dei^{2}}
\le\frac{2\ai^2\tri^{\ai-2}}{\dei^{2}},
\eeqs
where we used $|x|/\dei\le\tri$ in order to derive the last inequality.
\end{proof}
In order to control the inhomogeneous term $h$, we define functions $\psij$, $\tpsij$
as follows. Let $M>2\diam\Om$.
Let $\psij$ be defined by
\beq
\label{def:psij}
\bca
-\Delta\psij=\frac{\dej^\eta}{|x|^{2+\eta}}&\hbox{in }B_M\setminus B_{\tRj\dej}\\
\psij=0&\hbox{on }\pl(B_M\setminus B_{\tRj\dej})
\eca
\eeq
and let $\tpsij$ be defined by
\beq
\label{def:tpsij}
\bca
-\Delta\tpsij=\frac{1}{\dej^{2}}&\hbox{in }B_{\trj\dej}\\
\tpsij=0&\hbox{on\ }\pl B_{\trj\dej}.
\eca
\eeq
\begin{lemma}
\label{lem:psij}
There holds
\beq
\label{eq:psij}
\psij(r)=-\frac{\dej^\eta}{\eta^2r^\eta}+C_{1,j}\ln r+C_{2,j},
\eeq
where $C_{1,j}$, $C_{2,j}$ are given by
\beqs
\bal
C_{1,j}=\frac{1}{\eta^2}(\frac{\dej^\eta}{M^\eta}-\frac{1}{\tRj^\eta})\frac{1}{\ln\frac{M}{\tRj\dej}},
\qquad\qquad C_{2,j}=\frac{\dej^\eta}{\eta^2M^\eta}-C_{1,j}\ln M,
\eal
\eeqs
and
\beq
\tpsij(r)=\frac{1}{4}[\trj^{2}-(\frac{r}{\dej})^{2}].
\eeq
Moreover, the following uniform bounds hold true
\beq
\label{eq:psibounds}
0\le\psij<\frac{1}{\eta^2\tRj^\eta}+o(1),
\qquad
0\le\tpsij\le\frac{\trj^{2}}{4}.
\eeq
\end{lemma}
\begin{proof}
By straightforward computations we find that $\psij$ is of the form \eqref{eq:psij}.
The boundary conditions imply that
\beq
\label{eq:psijbdrycond}
\bca
&-\frac{\dej^\eta}{\eta^2M^\eta}+C_{1,j}\ln M+C_{2,j}=0\\
&-\frac{1}{\eta^2\tRj^\eta}+C_{1,j}\ln(\tRj\dej)+C_{2,j}=0,
\eca
\eeq
from which the asserted form of $\psij$ follows.
We observe that
\beq
\label{eq:Cprop}
C_{1,j}<0,\qquad \hbox{and}\qquad C_{1,j}=o(1)\qquad\hbox{as }\dej\to0
\eeq
and
\beq
\label{eq:C2prop}
C_{2,j}>0,\qquad \hbox{and}\qquad C_{2,j}=o(1)\qquad\hbox{as }\dej\to0.
\eeq
We have $\psij\ge0$ by the maximum principle. In order to establish the upper bound, we observe that
\beqs
\psij'(r)=\frac{1}{r}(\frac{\dej^\eta}{\eta r^\eta}+C_{1,j}),
\eeqs
and therefore $\psij$ attains its maximum value for $r^\eta=-\dej^\eta/(\eta C_{1,j})$.
We compute:
\beqs
\max\psij=\frac{C_{1,j}}{\eta}+\frac{C_{1,j}}{\eta}\ln(-\frac{\dej}{\eta C_{1,j}})+C_{2,j}
=C_{1,j}\ln\dej+o(1).
\eeqs
The second boundary condition in \eqref{eq:psijbdrycond} and \eqref{eq:Cprop}--\eqref{eq:C2prop} yield
\beqs
C_{1,j}\ln\dej=\frac{1}{\eta^2{\tRj}^\eta}+o(1),
\eeqs
so that the first upper bound in \eqref{eq:psibounds} is established.
The remaining bounds in \eqref{eq:psibounds} are straightforward.
%Substitution yields
%\beqs
%\max\psij=\frac{1}{\tRj}-C_{1,j}\ln\tRj+C_{1,j}+C_{1,j}\ln(-\frac{1}{C_{1,j}})
%=\frac{1}{\tRj}+C_{1,j}+C_{1,j}\ln(-\frac{1}{C_{1,j}\tRj})<\frac{1}{\tRj},
%\eeqs
%where we used \eqref{eq:Cprop} to derive the last line.
\end{proof}
\begin{lemma}
\label{lem:Lpopsi}
The following estimates hold true in $\tAjj$:
\begin{align}
\tag{i}
&\Lpo\psij\le\Big(-1+\frac{2\aj^2\ovC}{\eta^2\tRj^{\aj}}+o(1)\Big)\frac{\dej^\eta}{|x|^{2+\eta}}
+\Big(\frac{2\ajj^2\ovC\,\trjj^{\ajj-2}}{\tRj}+o(1)\Big)\frac{1}{\dejj^{2}};\\
\tag{ii}
&\Lpo\tpsijj\le\Big(\frac{\trjj^{2}}{4}\frac{2\aj^2\ovC}{\tRj^{\aj}}+o(1)\Big)\frac{\dej^\eta}{|x|^{2+\eta}}
+\Big(-1+\frac{2\ajj^2\ovC\trjj^{\ajj}}{4}+o(1)\Big)\frac{1}{\dejj^{2}},
\end{align}
where $o(1)$ vanishes as $p\to+\infty$.
\end{lemma}
\begin{proof}
Proof of (i).
We compute, using Lemma~\ref{lem:massdecomp}:
\beqs
\bal
\Lpo\psij=&\Delta\psij+\ovC\sum_{i=1}^k|x|^{\ai-2}e^{\Udai(x)}\psij
\le-\frac{\dej^\eta}{|x|^{2+\eta}}+\Big(\frac{\ovC}{\eta^2\tRj^\eta}+o(1)\Big)\sum_{i=1}^k|x|^{\ai-2}e^{\Udai(x)}\\
\le&-\frac{\dej^\eta}{|x|^{2+\eta}}+\Big(\frac{\ovC}{\eta^2\tRj^\eta}+o(1)\Big)\sum_{i=1}^{j-1}\frac{2\ai^2}{\tRi^{\ai-\eta}}\frac{\dei^\eta}{|x|^{2+\eta}}
+\Big(\frac{2\aj^2\ovC}{\eta^2\tRj^{\aj}}+o(1)\Big)\frac{\dej^\eta}{|x|^{2+\eta}}\\
&\quad+\Big(\frac{\ovC}{\eta^2\tRj^\eta}+o(1)\Big)\frac{2\ajj^2\trjj^{\ajj-2}}{\dejj^{2}}
+\Big(\frac{\ovC}{\eta^2\tRj^\eta}+o(1)\Big)\sum_{i=j+2}^k\frac{2\ai^2\tri^{\ai-2}}{\dei^{2}}\\
\le&\Big(-1+\frac{2\aj^2\ovC}{\eta^2\tRj^{\aj}}
+\frac{\ovC}{\eta^2\tRj^\eta}\sum_{i=1}^{j-1}\frac{2\ai^2}{\tRi^{\ai-\eta}}(\frac{\dei}{\dej})^\eta+o(1)\Big)\frac{\dej^\eta}{|x|^{2+\eta}}\\
&\qquad+\Big(\frac{2\ajj^2\ovC\trjj^{\ajj-2}}{\eta^2\tRj^\eta}
+\frac{\ovC}{\eta^2\tRj^\eta}\sum_{i=j+2}^k2\ai^2\tri^{\ai-2}(\frac{\dejj}{\dei})^{2}+o(1)\Big)\frac{1}{\dejj^{2}}\\
=&\Big(-1+\frac{2\aj^2\ovC}{\eta^2\tRj^{\aj}}
+o(1)\Big)\frac{\dej^\eta}{|x|^{2+\eta}}
+\Big(\frac{2\ajj^2\ovC\trjj^{\ajj-2}}{\eta^2\tRj^\eta}
+o(1)\Big)\frac{1}{\dejj^{2}},
\eal
\eeqs
as asserted.
\par
Proof of (ii).
Similarly, we compute:
\beqs
\bal
\Lpo\tpsijj=&\Delta\tpsijj+\ovC\sum_{i=1}^k|x|^{\ai-2}e^{\Udai(x)}\tpsijj\\
\le&-\frac{1}{\dejj^{2}}+\frac{\ovC\trjj^{2}}{4}\sum_{i=1}^k|x|^{\ai-2}e^{\Udai(x)}\\
\le&-\frac{1}{\dejj^{2}}+\frac{\ovC\trjj^{2}}{4}\sum_{i=1}^{j-1}\frac{2\ai^2}{\tRi^{\ai-\eta}}\frac{\dei^\eta}{|x|^{2+\eta}}
+\frac{\ovC\trjj^{2}}{4}\frac{2\aj^2}{\tRj^{\aj-\eta}}\frac{\dej^\eta}{|x|^{2+\eta}}\\
&+\frac{\ovC\trjj^{2}}{4}\frac{2\ajj^2\trjj^{\ajj-2}}{\dejj^{2}}
+\frac{\ovC\trjj^{2}}{4}\sum_{i=j+2}^k\frac{2\ai^2\tri^{\ai-2}}{\dei^{2}}\\
\le&\frac{\ovC\trjj^{2}}{4}\Big(\frac{2\aj^2}{\tRj^{\aj-\eta}}
+\sum_{i=1}^{j-1}\frac{2\ai^2}{\tRi^{\ai-\eta}}(\frac{\dei}{\dej})^\eta\Big)\frac{\dej^\eta}{|x|^{2+\eta}}\\
&+\Big(-1+\frac{2\ajj^2\ovC\trjj^{\ajj}}{4}
+\frac{\ovC\trjj^{2}}{4}\sum_{i=j+2}^k2\ai^2\tri^{\ai-2}(\frac{\dejj}{\dei})^{2}\Big)\frac{1}{\dejj^{2}}\\
=&\frac{\ovC\trjj^{2}}{4}\Big(\frac{2\aj^2}{\tRj^{\aj-\eta}}
+o(1)\Big)\frac{\dej^\eta}{|x|^{2+\eta}}
+\Big(-1+\frac{2\ajj^2\ovC\trjj^{\ajj}}{4}
+o(1)\Big)\frac{1}{\dejj^{2}},
\eal
\eeqs
as asserted.
\end{proof}
It is useful to observe that
\beq
\label{eq:tAjjinclusion}
\tAjj\subset\Aj\cup\Ajj,
\eeq
provided that $p$ is sufficiently large.
\begin{lemma}
\label{lem:hdecomp}
The following estimates hold true:
\begin{align}
\tag{i}
&|h(x)|\le\frac{\|h\|_{\rp}\dej^\eta}{|x|^{2+\eta}},\qquad\hbox{for all }x\in\Aj\cap\tAjj\\
\tag{ii}
&|h(x)|\le\frac{\|h\|_{\rp}}{\dejj^{2}},\qquad\hbox{for all }x\in\Ajj\cap\tAjj
\end{align}
In particular, there holds
\beq
\tag{iii}
|h(x)|\le\|h\|_{\rp}\Big(\frac{\dej^\eta}{|x|^{2+\eta}}+\frac{1}{\dejj^{2}}\Big),
\qquad\hbox{for all }x\in\tAjj.
\eeq
\end{lemma}
\begin{proof}
We recall from \eqref{def:starnorm} that
$\|h\|_{\rp}=\sup_{1\le i\le k}\|\rho_i h\|_{L^\infty(\Ai)}$, where $\rho_i(x)=(\dei^{2+\eta}+|x|^{2+\eta})/\dei^\eta$.
Proof of (i).
For any $x\in\Aj\cap\tAjj$ we have
\beqs
\bal
|h(x)|\le\frac{\|h\|_{\rp}}{\rj(x)}=\frac{\|h\|_{\rp}}{\dej^{2+\eta}+|x|^{2+\eta}}\dej^\eta
\le\|h\|_{\rp}\frac{\dej^\eta}{|x|^{2+\eta}}.
\eal
\eeqs
\par
Proof of (ii). Similarly, we compute, for all $x\in\Ajj\cap\tAjj$:
\beqs
\bal
|h(x)|\le\frac{\|h\|_{\rp}}{\rjj(x)}=\frac{\|h\|_{\rp}}{\dejj^{2+\eta}+|x|^{2+\eta}}\dejj^\eta
\le\frac{\|h\|_{\rp}}{\dejj^2},
\eal
\eeqs
as asserted.
\par
Proof of (iii). The asserted estimate readily follows from (i)--(ii)
and \eqref{eq:tAjjinclusion}.
\end{proof}
Finally, we provide the proof of the main result in this section, namely Proposition~\ref{prop:barrierest}.
\begin{proof}
We define the barrier function
\beqs
\Phi(x):=\frac{\|\phi\|_{L^\infty(\pl\tAjj)}}{k\co}\Psij(x)+2\|h\|_{\rp}(\psij(x)+\tpsij(x)),
\qquad x\in\tAjj.
\eeqs
Since $\Psij\ge k\co>0$, $\psij\ge0$, $\tpsij\ge0$ in $\tAjj$, we readily have the boundary estimate 
\beq
\label{est:Phibdry}
\Phi\ge\|\phi\|_{L^\infty(\pl\tAjj)}\qquad\hbox{on }\pl\tAjj.
\eeq
We claim that 
\beq
\label{est:LpPhi}
\Lp\Phi\le\Lp\phi=h\qquad\hbox{in }\tAjj.
\eeq
Indeed, we have:
\beqs
\Lp\Phi=\frac{\|\phi\|_{L^\infty(\pl\tAjj)}}{k\co}\Lp\Psij(x)+2\|h\|_{\rp}\Lp(\psij+\tpsij)
\eeqs
We recall from Lemma~\ref{lem:Psuj} and Lemma~\ref{lem:Psoj} with $\uDi=\oDi=D$, that
\beq
\label{est:LptPsij}
\bal
\Lp\Psij\le&-(D-k\ovC)\sum_{i=1}^k|x|^{\ai-2}e^{\Udai}<0,
\eal
\eeq
provided that the $D>0$ is sufficiently large.
We recall from Lemma~\ref{lem:Lpopsi} that
\begin{align*}
&\Lpo(\psij+\tpsijj)\le\Big(-1+\frac{2\aj^2\ovC}{\eta^2\tRj^{\aj}}+\frac{\trjj^{2}}{4}\frac{2\aj^2\ovC}{\tRj^{\aj-\eta}}+o(1)\Big)\frac{\dej^\eta}{|x|^{2+\eta}}\\
&+\Big(-1+\frac{2\ajj^2\ovC\trjj^{\ajj}}{4}+\frac{2\ajj^2\ovC\,\trjj^{\ajj-2}}{\eta^2\tRj^\eta}+o(1)\Big)\frac{1}{\dejj^{2}}
\end{align*}
uniformly in $\tAjj$, where $o(1)$ vanishes as $p\to+\infty$,
so that by possibly choosing a larger $\tRj$ and a smaller $\trjj$ we obtain te estimate
\beqs
\Lp(\psij+\tpsijj)\le\Lpo(\psij+\tpsijj)\le-\frac{1}{2}\Big(\frac{\dej^\eta}{|x|^{2+\eta}}+\frac{1}{\dejj^{2}}\Big).
\eeqs
%and consequently
%\beq
%\label{est:Lppsijpsijj}
%\bal
%\Lp(\psij+\tpsijj)\le&\Lpo(\psij+\tpsijj)+\ovC e^{-\ovb p}(\psij+\tpsijj)\\
%\le&-\frac{1}{2}\Big(\frac{\dej}{|x|^3}+\frac{1}{\dejj^{2-\tho}|x|^{\tho}}\Big)
%+\ovC e^{-\ovb p}(\frac{1}{\tRj}+\frac{\trjj^{2-\tho}}{(2-\tho)^2}).
%\eal
%\eeq
%Finally, using \eqref{eq:LpLpo}, we have
%\beqs
%\Lp\psio\le\Lpo\psio+\ovC e^{-\ovb p}\psio\le-1+\frac{\ovC M^2}{4}\sum_{i=1}^k|x|^{\ai-2}e^{\Udai}+\frac{\ovC M^2}{4}e^{-\ovb p}.
%\eeq
In conclusion, from \eqref{est:LptPsij}, \eqref{est:LptPsij}, Lemma~\ref{lem:hdecomp}--(iii) and the above we derive
\beqs
\bal
\Lp\Phi
%&\frac{\|\phi\|_{L^\infty(\pl\tAjj)}+\tK e^{-\ovb p}}{k\co}\Big\{\sum_{i=1}^k\Big(-D_i+k\ovC+\frac{\Ko\ovC M^2}{4}e^{-\ovb p}\Big)|x|^{\ai-2}e^{\Udai(x)}\\
%&-e^{-\ovb p}\Big[\Ko-(k\ovC+\frac{\ovC M^2}{4})\Big]\Big\}\\
%&-\|h\|_{\rp}(\frac{\dej}{|x|^3}+\frac{1}{\dejj^{2-\tho}|x|^{\tho}})+2\|h\|_{\rp}\ovC e^{-\ovb p}(\frac{1}{\tRj}+\frac{\trjj^{2-\tho}}{(2-\tho)^2})\\
%&+\tK e^{-\ovb p}\Big[-1+\frac{\ovC M^2}{4}\sum_{i=1}^k|x|^{\ai-2}e^{\Udai}+\frac{\ovC M^2}{4}e^{-\ovb p}\Big]\\
%\le&-\|h\|(\frac{\dej}{|x|^3}+\frac{1}{\dejj^{2-\tho}|x|^{\tho}})\\
%&+\frac{\|\phi\|_{L^\infty(\pl\tAjj)}+\tK e^{-\ovb p}}{k\co}\sum_{i=1}^k
%\Big(-D_i+k\ovC+\frac{\ovC M^2}{4}+\frac{\Ko\ovC M^2}{4}e^{-\ovb p}\Big)|x|^{\ai-2}e^{\Udai(x)}\\
%&+e^{-\ovb p}\Big\{-\frac{\|\phi\|_{L^\infty(\pl\tAjj)}+\tK e^{-\ovb p}}{k\co}\Big[\Ko-(k\ovC+\frac{\ovC M^2}{4})\Big]\\
%&\qquad+2\|h\|_{\rp}\ovC(\frac{1}{\tRj}+\frac{\trjj^{2-\tho}}{(2-\tho)^2})
%+\tK\Big[-1+\frac{\ovC M^2}{4}e^{-\ovb p}\Big]\Big\}\\
\le&-\|h\|(\frac{\dej^\eta}{|x|^{2+\eta}}+\frac{1}{\dejj^{2}})
\le-|h(x)|=-|\Lp\phi|\le\Lp\phi.
\eal
\eeqs
and \eqref{est:LpPhi} is established. Now, the barrier estimate follows by the maximum principle.
\end{proof}
%%%%%%%%%%%%%%%%%%%%%%%%%%%%%%%%%%%%%%%%%%%%%%%%%%%%%%%%%%%%%%%%%%%%%%%%%%%%%%%~~~~~
%%%%%%%%%%%%%%%%%%%%%%%%%%%%%%%%%%%%%%%%%%%%%%%%%%%%%%%%%%%%%%%%%%%%%%%%%%%%%%%
%%%%%%%%%%%%%%%%%%%%%%%%%%%%%%%%%%%%%%%%%%%%%%%%%%%%%%%%%%%%%%%%%%%%%%%%%%%%%%%~~~~~
%%%%%%%%%%%%%%%%%%%%%%%%%%%%%%%%%%%%%%%%%%%%%%%%%%%%%%%%%%%%%%%%%%%%%%%%%%%%%%%
\section{The fixed point argument}
\label{sec:fixedpoint}
%%%%%%%%%%%%%%%%%%%%%%%%%%%%%%%%%%%%%%%%%%%%%%%%%%%%%%%%%%%%%%%%%%%%%%%%%%%%%%%~~~~~
%%%%%%%%%%%%%%%%%%%%%%%%%%%%%%%%%%%%%%%%%%%%%%%%%%%%%%%%%%%%%%%%%%%%%%%%%%%%%%%
We recall from \eqref{def:R} that the error $\Rp$ is defined by
\beqs
\Rp=\Delta\Up+\gp(\Up)
\eeqs
and that the operator $\Lp$ is defined by
\beqs
\Lp\phi=\Delta\phi+\Wp(x)\phi
\eeqs
for all $\phi\in C_0(\overline\Om)$.
Therefore, setting
\beq
\label{def:Np}
\Np(\phi):=\gp(\Up+\phi)-\gp(\Up)-\gpp(\Up)\phi,
\eeq
we see that $\up=\Up+\phip$ is a solution for \eqref{eq:pb}
if and only if $-\Lp\phip=\Rp+\Np(\phip)$.
Defining the operator 
\beq
\Tp(\phi):=-(\Lp)^{-1}(\Rp+\Np(\phi)),
\eeq
we may rewrite \eqref{eq:pb} with Ansatz~$u=\up=\Up+\phip$
in the form
\beq
\label{eq:fixedpoint}
\phip=\Tp\phip.
\eeq
In other words, we are reduced to seek a fixed point $\phip\in C_0(\overline\Om)$ for $\Tp$.
Let
\beq
\label{def:Fgamma}
\Fg:=\left\{\phi\in C_0(\overline\Om):\ \|\phi\|_\infty<\frac{\gamma}{p^3}\right\},
\eeq
where $\ga>0$.
\begin{prop}
\label{prop:fixedpoint}
There exist $\ga>0$ and $p_0>0$ such that for every $p\ge p_0$
the operator $\Tp:\mathcal F_\gamma\to\mathcal F_\gamma$
is a contraction.
\end{prop}
We first establish the following estimate.
\begin{lemma}
\label{lem:gppower}
There holds
\beqs
\left\|\left|\gp(\Up+O(\frac{1}{p^3}))\right|^{\frac{p-2}{p}}\right\|_{\rp}=O(\frac{1}{p}).
\eeqs
\end{lemma}
\begin{proof}
We have
\beqs
\bal
\rp(x)&|\gp(\Up+O(\frac{1}{p^3}))|^{(p-2)/p}\\
=&\sum_{j=1}^k\rj(x)\left|\tj|x|^{\aj-2}e^{\Udaj(x)}\{1+\frac{\Vaj(\frac{x}{\dej})+O(1)}{p}\}
\chi_{\Bj}+\omp(x)\chi_{\Aj\setminus\Bj}\right|^{\frac{p-2}{p}}\\
\le&\sum_{j=1}^k\rj(x)\left(\tj|x|^{\aj-2}e^{\Udaj(x)}|\{1+\frac{\Vaj(\frac{x}{\dej})+O(1)}{p}\}|^{\frac{p-2}{p}}\chi_{\Bj}
+|\omp(x)|^{\frac{p-2}{p}}\chi_{\Aj\setminus\Bj}\right)
\eal
\eeqs
\end{proof}
\begin{proof}[Proof of Proposition~\ref{prop:fixedpoint}]
\textit{Claim~1.} $\Tp:\Fg\to\Fg$.
By applying the mean value theorem twice, we may write
\beqs
\bal
&\Np(\phi)=\gps(\Up+\theta_p'\theta_p''\phi)\theta_p'\phi^2
\eal
\eeqs
for some $0\le\theta_p',\theta_p''\le1$.
Consequently, for $\phi\in\Fg$,
\beqs
\bal
|\Np(\phi)|\le&|\gps(\Up+O(\frac{1}{p^3}))|\,|\phi|^2
=p(p-1)|\gp(\Up+O(\frac{1}{p^3}))|^{(p-2)/p}\,|\phi|^2,
\eal
\eeqs
and, in view of Lemma~\ref{lem:gppower}:
\beqs
\bal
\|\Np(\phi)\|_{\rp}\le&p(p-1)\|\gps(\Up+O(\frac{1}{p^3}))\|_{\rp}\|\phi\|_\infty^2\\
\le&Cp\left(\frac{\ga}{p^3}\right)^2=O(\frac{1}{p^5}).
\eal
\eeqs
It follows that, for any $\phi\in\Fg$ we have
\beqs
\|\Tp(\phi)\|_\infty\le Cp(\|\Rp\|_{\rp}+\|\Np(\phi)\|_{\rp})
\le Cp(\frac{C}{p^4}+O(\frac{1}{p^5})).
\eeqs
Now Claim~1 follows by choosing $\ga$ suitably large.
\par
\textit{Claim~2.} $\Tp$ is a contraction.
Indeed, similarly as above, by two applications of the mean value theorem we obtain
\beqs
\Np(\phi_1)-\Np(\phi_2)=\gps(\Up+\eta_p''(\phi_2+\eta_p'(\phi_1-\phi_2)))
(\eta_p'\phi_1+(1-\eta_p')\phi_2)(\phi_1-\phi_2)
\eeqs
for some $0\le\eta_p',\eta_p''\le1$.
Therefore,
\beqs
\bal
|\Np(\phi_1)-\Np(\phi_2)|\le&|\gps(\Up+O(\frac{1}{p^3}))|\max_{i=1,2}\|\phi_i\|_\infty
|\phi_1-\phi_2|\\
\le&p(p-1)|\gp(\Up+O(\frac{1}{p^3}))|^{(p-2)/p}\frac{\ga}{p^3}
|\phi_1-\phi_2|
\eal
\eeqs
and
\beqs
\bal
\|\Np(\phi_1)-\Np(\phi_2)\|_{\rp}
\le&p(p-1)\||\gp(\Up+O(\frac{1}{p^3}))|^{(p-2)/p}\|_{\rp}\frac{\ga}{p^3}\|\phi_1-\phi_2\|_\infty\\
\le&\frac{C\ga}{p^2}\|\phi_1-\phi_2\|_\infty
\eal
\eeqs
Therefore, we deduce that
\beqs
\|\Tp(\phi_1)-\Tp((\phi_2)\|_\infty\le Cp(\|\Np(\phi_1)-\Np(\phi_2)\|_{\rp}\le\frac{C}{p}\|\phi_1-\phi_2\|_\infty.
\eeqs
This establishes Claim~2.
\end{proof}

%%%%%%%%%%%%%%%%%%%%%%%%%%%%%%%%%%%%%%%%%%%%%%%%%%%%%%%%%%%%%%%%%%%%%%%%%%%%%%%~~~~~
%%%%%%%%%%%%%%%%%%%%%%%%%%%%%%%%%%%%%%%%%%%%%%%%%%%%%%%%%%%%%%%%%%%%%%%%%%%%%%%
%%%%%%%%%%%%%%%%%%%%%%%%%%%%%%%%%%%%%%%%%%%%%%%%%%%%%%%%%%%%%%%%%%%%%%%%%%%%%%%~~~~~
%%%%%%%%%%%%%%%%%%%%%%%%%%%%%%%%%%%%%%%%%%%%%%%%%%%%%%%%%%%%%%%%%%%%%%%%%%%%%%%
\section{Appendix: Useful facts}
%%%%%%%%%%%%%%%%%%%%%%%%%%%%%%%%%%%%%%%%%%%%%%%%%%%%%%%%%%%%%%%%%%%%%%%%%%%%%%%~~~~~
%%%%%%%%%%%%%%%%%%%%%%%%%%%%%%%%%%%%%%%%%%%%%%%%%%%%%%%%%%%%%%%%%%%%%%%%%%%%%%%
In this Appendix we collect some useful properties and estimates.

\subsection{Estimates in the annuli $\Aj$}
We recall that
\beq
\label{eq:Ajprops}
\bal
&|x|<C_1^{\veps_1}C_2^{1-\veps_1}e^{-(\veps_1b_1+(1-\veps_1)b_2)p}\qquad\hbox{for all }x\in A_1;\\
&C_{j-1}^{\veps_{j-1}}C_{j}^{1-\veps_{j-1}}e^{-(\veps_{j-1}b_{j-1}+(1-\veps_{j-1})b_{j})p}
\le|x|\le C_j^{\ej}C_{j+1}^{1-\ej}e^{-(\ej\bj+(1-\ej)b_{j+1})p},\\
&\qquad\qquad\qquad\qquad\hbox{for all }x\in\Aj,\ j=2,\ldots,k-1;\\
&|x|\ge C_{k-1}^{\veps_{k-1}}C_k^{1-\veps_{k-1}}e^{-(\veps_{k-1}b_{k-1}+(1-\veps_{k-1})b_k)p},\qquad\hbox{for all }x\in\Ak.
\eal
\eeq
Consequently, we have
\beq
\bal
&\frac{\dei}{|x|}\le\frac{\dei}{\de_{j-1}^{\veps_{j-1}}\dej^{1-\veps_{j-1}}}
\le C\left(\frac{\de_{j-1}}{\dej}\right)^{1-\veps_{j-1}}\le Ce^{-(1-\veps_{j-1})(b_{j-1}-b_j)p},\\
&\qquad\qquad\qquad\qquad\qquad\qquad\hbox{for all }x\in\Aj,\ i<j,\ j=2,\ldots,k;\\
&\frac{|x|}{\dei}\le\frac{\dej^{\ej}\dejj^{1-\ej}}{\dei}\le C\left(\frac{\dej}{\dejj}\right)^{\ej}
\le Ce^{-\ej(b_j-b_{j+1})p},\\
&\qquad\qquad\qquad\qquad\qquad\qquad\hbox{for all }x\in\Aj,\ i>j,\ j=1,2,\ldots,k-1.
\eal
\eeq
The next lemma clarifies the leading term of the quantity
$\ln(\dei^{\ai}+|\dej y|^{\ai})^{-2}$ for $\dej y\in\Aj$, $i,j=1,2,\ldots,k$.
\begin{lemma}
\label{lem:logleadterm}
There holds:
\beqs
\ln\frac{1}{(\dei^{\ai}+|\dej y|^{\ai})^2}=
\left\{
\bal
&\ln\frac{1}{(1+|y|^{\aj})^2}&&i=j\ \hbox{(natural scaling)}\\
&-2\ai\ln\dei+O(\frac{\dej}{\dejj})^{\ej\ai}&&i>j,\ i=j+1,\ldots,k\ \hbox{(fast scaling)}\\
&2\ai\ln\frac{1}{|y|}-2\ai\ln\dej+O(\frac{\de_{j-1}}{\dej})^{(1-\veps_{j-1})\ai}&&i<j,\ i=1,\ldots,j-1,\ \hbox{(slow scaling)},
\eal
\right.
\eeqs
uniformly for $\dej y\in\Aj$.
\end{lemma}
\begin{proof}
It suffices to establish the following:
\beqs
\dei^{\ai}+|\dej y|^{\ai}=
\left\{
\bal
&\dej^{\aj}(1+|y|^{\aj})&&i=j\ \hbox{(natural scaling)}\\
&\dei^{\ai}\left(1+O(\frac{\dej}{\dejj})^{\ej\ai}\right)&&i>j,\ i=j+1,\ldots,k\ \hbox{(fast scaling)}\\
&|\dej y|^{\ai}\left(1+O(\frac{\de_{j-1}}{\dej})^{(1-\veps_{j-1})\ai}\right)
&&i<j,\ i=1,\ldots,j-1,\ \hbox{(slow scaling)},
\eal
\right.
\eeqs
uniformly for $\dej y\in\Aj$.
\par
To this end, suppose $i=j$. Then, we readily have
\beqs
\dei^{\ai}+|\dej y|^{\ai}=\dej^{\aj}(1+|y|^{\aj}).
\eeqs
Suppose $i>j$.
Then, $\de_{j+1}=O(\dei)$, $|x|=|\dej y|<\dej^{\ej}\de_{j+1}^{1-\ej}$, and therefore
\beq
\label{eq:dejydei}
\frac{|\dej y|}{\dei}=O(\frac{\dej^{\ej}\de_{j+1}^{1-\ej}}{\de_{j+1}})
=O(\frac{\dej}{\de_{j+1}})^{\ej}.
\eeq
Consequently,
\beqs
\dei^{\ai}+|\dej y|^{\ai}=\dei^{\ai}(1+(\frac{|\dej y|}{\dei})^{\ai})
=\dei^{\ai}(1+O(\frac{\dej}{\de_{j+1}})^{\ej\ai}),
\eeqs
as asserted.
\par
Suppose $i<j$.
Then, $\dei=O(\de_{j-1})$, $|x|=|\dej y|\ge\de_{j-1}^{\veps_{j-1}}\dej^{1-\veps_{j-1}}$,
and therefore
\beq
\label{eq:deidejy}
\frac{\dei}{|\dej y|}=O(\frac{\de_{j-1}}{\de_{j-1}^{\veps_{j-1}}\dej^{1-\veps_{j-1}}})
=O(\frac{\de_{j-1}}{\dej})^{1-\veps_{j-1}}.
\eeq
It follows that
\beqs
\dei^{\ai}+|\dej y|^{\ai}=|\dej y|^{\ai}(1+(\frac{\dei}{|\dej y|})^{\ai})
=|\dej y|^{\ai}(1+O(\frac{\de_{j-1}}{\dej})^{(1-\veps_{j-1})\ai}),
\eeqs
as asserted.
\end{proof}
\subsection{Expansions and scalings in the $\Aj$'s.}
We recall from \eqref{def:ta} that
\beqs
2\tai:=2\ai-\frac{\Coai}{p}-\frac{\Ciai}{p^2},
\eeqs
where we recall that $\Coai$, $\Ciai$ are defined by the property \eqref{bdrycond:wa}.
Moreover, we set
\beq
v_{\ai}^{w}:=\vai+\frac{\woai}{p}+\frac{\wiai}{p^2}.
\eeq
\begin{lemma}[Projection expansions]
\label{lem:projexp}
The following expansions hold true:
\beqs
P\Udai(x)=
\left\{
\bal
&\ln\frac{1}{(\dei^{\ai}+|x|^{\ai})^2}+4\pi\ai H(x,0)+O(\dei^{\ai}),
&&\mathrm{in\ }C^1(\overline\Om)\\
&4\pi\ai G(x,0)+O(\dei^{\ai})
&&\mathrm{in\ }C_{\mathrm{loc}}^1(\overline\Om\setminus\{0\})
\eal
\right.
\eeqs
\beqs
Pw_{\ai,\dei}^\ell(x)=
\left\{
\bal
&w_{\ai,\dei}^\ell(x)-2\pi C_{\ai}^\ell H(x,0)+C_{\ai}^\ell\ln\dei+O(\dei),
&&\mathrm{in\ }C^1(\overline\Om)\\
&-2\pi C_{\ai}^\ell G(x,0)+O(\dei),
&&\mathrm{in\ }C_{\mathrm{loc}}^1(\overline\Om\setminus\{0\})
\eal
\right.
\eeqs
\end{lemma}
\begin{proof}
The proof is well-known, see, e.g., \cite{emp1}.
We outline the proof for the sake of completeness.
\par
For any fixed $r>0$ we have
\beqs
\Udai(x)=\ln\frac{2\ai^2\dei^{\ai}}{(\dei^{\ai}+|x|^{\ai})^2}
=2\ai\ln\frac{1}{|x|}+\ln(2\ai^2\dei^{\ai})+O(\dei^{\ai}),
\eeqs
uniformly in $C^1(\Om\setminus B_r(0))$. Therefore, we may write
\beqs
\Udai(x)=4\pi\ai(G(x,0)-H(x,0))+\ln(2\ai^2\dei^{\ai})+O(\dei^{\ai}),
\eeqs
uniformly in $C^1(\Om\setminus B_r(0))$.
It follows that
\beqs
\bca
\Delta(\Udai(x)-\ln(2\ai^2\dei^{\ai})+4\pi\ai H(x,0))=\Delta\Udai(x)&\hbox{in }\Om\\
\Udai(x)-\ln(2\ai^2\dei^{\ai})+4\pi\ai H(x,0)=O(\dei^{\ai})&\hbox{on }\partial\Om,
\eca
\eeqs
so that $P\Udai=\Udai(x)-\ln(2\ai^2\dei^{\ai})+4\pi\ai H(x,0)+O(\dei^{\ai})$ in $C^1(\overline\Om)$,
as asserted.
\par
By properties of $\wlai$, $\ell=0,1$, we have
\beqs
\wldai(x)=\wlai(\frac{x}{\dei})=\Clai\ln|\frac{x}{\dei}|+O(\frac{\dei}{|x|}).
\eeqs
Hence, for any $r>0$ we may write
\beqs
\wldai(x)=-2\pi\Clai G(x,0)+2\pi\Clai H(x,0)-\Clai\ln\dei+O(\dei),
\eeqs
uniformly in $\Om\setminus B_r(0)$.
The statement follows observing that
\beqs
\begin{cases}
\Delta(\wldai-2\pi\Clai H(x,0)+\Clai\ln\dei)=\Delta\wldai&\hbox{in }\Om\\
\wldai-2\pi\Clai H(x,0)+\Clai\ln\dei=O(\dei)&\hbox{on }\partial\Om,
\end{cases}
\eeqs
so that $P\wldai=\wldai-2\pi\Clai H(x,0)+\Clai\ln\dei+O(\dei)$ in $C^1(\overline\Om)$,
as asserted.
\end{proof}
The aim of the next lemma is to establish the profile of the $i$-th bubble observed in the shrinking ring $\Aj$.
It will be useful to note that we may write
\beq
\label{eq:lnvai}
\ln\frac{1}{(\dei^{\ai}+|x|^{\ai})^2}=\vai(\frac{x}{\dei})-2\ai\ln\dei-\ln(2\ai^2).
\eeq
\begin{lemma}[Bubble scaling]
\label{lem:modbubproj}
Let $i,j=1,2,\ldots,k$.
The following expansions hold, uniformly for $x=\dej y\in\Aj$:
\[
P\Udai(x)=
\left\{
\bal
&\vj(y)-2\aj\ln\dej-\ln(2\aj^2)
+4\pi\aj h(0)
+O(|\dej y|)+O(\dej^{\aj}),\\
&\qquad\qquad\qquad\qquad\qquad\qquad\qquad\qquad\hbox{if }i=j\ \hbox{(natural scaling)};\\
&-2\ai\ln\dei+4\pi\ai h(0)+O(\frac{\dej}{\dejj})^{\ej\ai}+O(\dej|y|)+O(\dei^{\ai}),\\
&\qquad\qquad\qquad\qquad\qquad\qquad\qquad\qquad\hbox{if }i>j\ \hbox{(slow bubble)};\\
&2\ai\ln\frac{1}{|y|}
-2\ai\ln\dej+4\pi\ai h(0)
+O(\frac{\de_{j-1}}{\dej})^{(1-\veps_{j-1})\ai}+O(\dej|y|)+O(\de^{\ai}),\\ 
&\qquad\qquad\qquad\qquad\qquad\qquad\qquad\qquad\hbox{if }i<j\ \hbox{(fast bubble)};
\eal
\right.
\]
\[
P\wldai(x)=
\left\{
\bal
&\wlaj(y)+\Claj\ln\dej-2\pi\Claj h(0) 
+O(|\dej y|)+O(\dej),\\
&\qquad\qquad\qquad\qquad\qquad\qquad\qquad\qquad\hbox{if }i=j\ \hbox{(natural scaling)};\\
&\wlai(0)+\Clai\ln\dei-2\pi\Clai h(0)
+O(\frac{\dej}{\dejj})^{\ej}+O(\dej|y|)+O(\dei),\\
&\qquad\qquad\qquad\qquad\qquad\qquad\qquad\qquad\hbox{if }i>j\ \hbox{(slow bubble)};\\
&\Clai\ln|y|+\Clai\ln\dej-2\pi\Clai h(0)
+O(\frac{\de_{j-1}}{\dej})^{1-\veps_{j-1}}+O(\dej|y|)+O(\dei),\\ 
&\qquad\qquad\qquad\qquad\qquad\qquad\hbox{if }i<j\ \hbox{(fast bubble)}.
\eal
\right.
\]
%Consequently, the scalings of the modified bubble $P\tUdai$ are given by
%\[
%P\tUdai(x)=
%\left\{
%\bal
%&\tvj(y)
%-2\taj\ln\dej+4\pi\taj h(0)
%-\ln(2\aj^2)+O(|\dej y|)+O(\frac{\dej}{p}),\\
%&\qquad\qquad\qquad\qquad\qquad\qquad\qquad\qquad\hbox{if }j=i\ \hbox{(natural scaling)};\\
%&-2\tai\ln\dei+4\pi\tai h(0)+\frac{\woai(0)}{p}+\frac{\wiai(0)}{p^2}+O(\dej|y|)\\
%&\qquad\qquad+O((\frac{\dej}{\dei})^{\ai}|y|^{\ai})+O(\frac{\dei}{p}),\ \hbox{if }j<i\ \hbox{(fast scaling)};\\
%&2\tai\ln\frac{1}{|y|}
%-2\tai\ln\dej+4\pi\tai h(0)
%+O(|\dej y|)+O((\frac{\dei}{\dej})^{\ai}\frac{1}{|y|^{\ai}})+O(\frac{\dei}{p}),\\
%&\qquad\qquad\qquad\qquad\qquad\qquad\qquad\qquad\hbox{if }j>i\ \hbox{(slow scaling)}.
%\eal
%\right.
%\]
\end{lemma}
\begin{proof}
Suppose $i=j$.
In view of Lemma~\ref{lem:projexp} and \eqref{eq:lnvai} we readily derive the expansion.
\par
Suppose $i>j$. 
In view of Lemma~\ref{lem:logleadterm} and Lemma~\ref{lem:projexp} we have
\beqs
\bal
P\Udai(\dej y)=&\ln\frac{1}{(\dei^{\ai}+|\dej y|^{\ai})^2}+4\pi\ai H(\dej y,0)+O(\dei^{\ai})\\
=&-2\ai\ln\dei+O(\frac{\dej}{\dejj})^{\ej\ai}+4\pi\ai h(0)+O(|\dej y|)+O(\dei^{\ai}),
\eal
\eeqs
as asserted.
\par
Suppose $i<j$ (fast bubble). In view of Lemma~\ref{lem:logleadterm} and Lemma~\ref{lem:projexp} we have
\beqs
\bal
P\Udai(\dej y)=&\ln\frac{1}{(\dei^{\ai}+|\dej y|^{\ai})^2}+4\pi\ai H(\dej y,0)+O(\dei^{\ai})\\
=&2\ai\ln\frac{1}{|y|}-2\ai\ln\dej+O(\frac{\de_{j-1}}{\dej})^{(1-\veps_{j-1})\ai}+4\pi\ai h(0)+O(|\dej y|)+O(\dei^{\ai}),
\eal
\eeqs
and the asserted expansion follows.
\par
Now, we consider the correction term $\wldai$.
Suppose $i<j$ (fast bubble).
In view of Lemma~\ref{lem:projexp} we have
\beqs
P\wldai(\dej y)=\wlai(\frac{\dej y}{\dei})+\Clai\ln\dei-2\pi\Clai H(\dej y,0)+O(\dei).
\eeqs
Using \eqref{eq:deidejy} we obtain
\beqs
\frac{|\dej y|}{\dei}\ge C^{-1}(\frac{\dej}{\de_{j-1}})^{1-\veps_{j-1}}\to+\infty
\eeqs
and therefore
\beqs
\wlai(\frac{\dej y}{\dei})=\Clai\ln|\frac{\dej y}{\dei}|+O(\frac{\dei}{|\dej y|}).
\eeqs
It follows that
\beqs
\bal
P\wlai(\dej y)=&\Clai\ln|y|+\Clai\ln\frac{\dej}{\dei}+O(\frac{\dei}{|\dej y|})+\Clai\ln\dei
-2\pi\Clai h(0)+O(|\dej y|)+O(\dei)\\
=&\Clai\ln|y|+\Clai\ln\dej-2\pi\Clai h(0)+O(\frac{\de_{j-1}}{\dej})^{1-\veps_{j-1}}+O(|\dej y|)+O(\dei),
\eal
\eeqs
as asserted.
\end{proof}
\begin{lemma}[$\dej$-scaling of the $i$-th mass]
\label{lem:massscaling}
The following expansions holds true, uniformly for $x=\dej y\in\Aj$:
\beqs
\dej^2|x|^{\ai-2}e^{U_{\ai,\dei}(x)}=
\left\{
\bal
&|y|^{\aj-2}e^{v_{\aj}(y)},\\
&\qquad\qquad\qquad\qquad\qquad\qquad\qquad\qquad\qquad\hbox{if }i=j\\
&(\frac{\dej}{\dei})^{\ai}\frac{2\ai^2|y|^{\ai-2}}{1+O(\frac{\dej}{\dejj})^{\ej\ai}}
=O(\frac{\dej}{\dejj})^{\ej\ai+2(1-\ej)},\\
&\qquad\qquad\qquad\qquad\qquad\qquad\qquad\qquad\qquad\hbox{if }i>j\\
&(\frac{\dei}{\dej})^{\ai}\frac{2\ai^2}{|y|^{\ai+2}(1+O(\frac{\de_{j-1}}{\dej})^{(1-\veps_{j-1})\ai})}=O(\frac{\de_{j-1}}{\dej})^{(1-\veps_{j-1})\ai-2\veps_{j-1}},\\
&\qquad\qquad\qquad\qquad\qquad\qquad\qquad\qquad\qquad\hbox{if }i<j,
\eal
\right.
\eeqs
where $\ej\ai+2(1-\ej)>0$ and $(1-\veps_{j-1})\ai-2\veps_{j-1}>0$.
\end{lemma}
\begin{proof}
For $i=j$ the proof follows by the change of variables $x=\dej y$.
\par
For $i<j$, $x=\dej y\in\Aj$, we obtain by change of variables that
\beqs
\dej^2|x|^{\ai-2}e^{U_{\ai,\dei}(x)}=(\frac{\dei}{\dej})^{\ai}\frac{2\ai^2}{|y|^{\ai+2}(1+(\frac{\dei}{\dej|y|})^{\ai})^2}.
\eeqs
Recalling the definition of $\Aj$, we observe that
\beqs
\frac{\dei}{\dej|y|}=O(\frac{\de_{j-1}}{\dej|y|})=O(\frac{\de_{j-1}}{\dej})^{1-\veps_{j-1}}=o(1).
\eeqs
Finally, we observe that
\beqs
(\frac{\dei}{\dej})^{\ai}\frac{1}{|y|^{\ai+2}}=O(\frac{\de_{j-1}}{\dej})^{\ai}(\frac{\dej}{\de_{j-1}})^{\veps_{j-1}(\ai+2)}
=O(\frac{\de_{j-1}}{\dej})^{(1-\veps_{j-2})\ai-2\veps_{j-1}},
\eeqs
and, by definition of $\veps_{j-1}$,
\beqs
(1-\veps_{j-2})\ai-2\veps_{j-1}=\frac{\ai-2s_{j-1}}{1+s_{j-1}}>0.
\eeqs
Similarly, for $i>j$ we obtain by change of variables that
\beqs
\dej^2|x|^{\ai-2}e^{U_{\ai,\dei}(x)}=(\frac{\dej}{\dei})^{\ai}\frac{2\ai^2|y|^{\ai-2}}{(1+|\frac{\dej y}{\dei}|^{\ai})^2}.
\eeqs
Recalling the definition of $\Aj$ we have
\beqs
|\frac{\dej y}{\dei}|=O(|\frac{\dej y}{\dejj}|)=O(\frac{\dej}{\dejj})^{1-\ej}=o(1).
\eeqs
In order to conclude the proof, we observe that
\beqs
(\frac{\dej}{\dei})^{\ai}|y|^{\ai-2}=O(\frac{\dej}{\dejj})^{\ai}(\frac{\dejj}{\dei})^{(1-\ej)(\ai-2)}
=O(\frac{\dej}{\dejj})^{\ej\ai+2(1-\ej)}=o(1).
\eeqs
\end{proof}
\begin{lemma}
[$\dej$-scaling of the $i$-th radial eigenfunction]
\label{lem:radialefscaling}
There holds:
\beq
\label{eq:efexp}
P\zaiod=\zaiod+1+O(\dei^{\ai}),
\eeq
uniformly in $\Om$.
Moreover,
\beq
\label{eq:efscaling}
P\zaiod(\dej y)=
\left\{
\bal
&\frac{2}{1+|y|^{\ai}}+O(\dei^{\ai}),
&&\hbox{if }i=j\ \hbox{(natural scaling)};\\
&2+O(\frac{\dej|y|}{\dei})^{\ai},
&&\hbox{if }i>j\ \hbox{(fast scaling)};\\
&O(\frac{\dei}{\dej|y|})^{\ai}+O(\dei^{\ai}),
&&\hbox{if }i<j\ \hbox{slow scaling},
\eal
\right.
\eeq
uniformly for $x=\dej y\in\Aj$.
\end{lemma}
\begin{proof}
The proof is straightforward.
\end{proof}
%%%%%%%%%%%%%%%%%%%%%%%%%%%%%%%%%%%%%%%%%%%%%%%%%%%%%%%%%%%%%%%%%%%%%%%%%%%%%%%%%%%%%%%%%%%%%%%%%%%%%%%%%%%%%%%%%%%%%%
\subsection{Choice of the parameters $\de,\tau$ in the case $k=1$.}
We recall from Section~\ref{sec:intro} that
\[
\Ud(x)=\ln\frac{8\de^2}{(\de^2+|x|^2)^2},
\]
for all $\de>0$.
Then, in view of Lemma~\ref{lem:projexp} with $\ai=2$ and $\de=\dei$, we have
\[
P\Ud(x)=\ln\frac{1}{(\de^2+|x|^2)^2}+8\pi H(x,0)+O(\de^2)=v(\frac{x}{\de})+\ln\frac{1}{\de^4}+O(1).
\]
Therefore, taking $\ln\de^{-4}=p$, we may expand:
\[
\bal
(P\Ud)^p=&(\ln\de^{-4})^p\left(1+\frac{1}{\ln\de^{-4}}v(\frac{x}{\de})+O(\frac{1}{\ln\de^{-4}})\right)^p
=p^p\left(1+\frac{1}{p}v(\frac{x}{\de})+O(\frac{1}{p})\right)^p\\
=&p^pe^{v(x/\de)}(1+O(\frac{1}{p}))=p^p\de^2e^{\Ud(x)}(1+O(\frac{1}{p})).
\eal
\]
On the other hand, we have
\[
-\Delta(\tau P\Ud)=\tau e^{\Ud}.
\]
Choosing $\tau^{p-1}p^p\de^2=1$, we obtain
\beq
\label{eq:basicapprox}
\Delta(\tau P\Ud)+\left(\tau P\Ud\right)^p=\tau e^{\Ud(x)}\left(-1+\tau^{p-1}p^p\de^2(1+O(\frac{1}{p}))\right)=\tau e^{\Ud(x)}O(\frac{1}{p}),
\eeq
and therefore $\tau P\Ud$ is indeed an approximate solution for \eqref{eq:pb}.
We have obtained the following necessary conditions for the parameters:
\beq
\label{eq:onebubbleparameters}
\de=e^{-p/4}, \qquad\qquad \tau=\frac{e^{p/[2(p-1)]}}{p^{p/(p-1)}}=\frac{\sqrt{e}}{p}(1+O(\frac{\ln p}{p})).
\eeq
%%%%%%%%%%%%%%%%%%%%%%%%%%%%%%%%%%%%%%%%%%%%%%%%%%%%%%%%%%%%%%%%%%%%%%%%%%%%%%%%%%%%%%%%%%%%%%%%%%%%%%%%%%%%%%%%%%%%%%%%
\subsection{Properties of the weighted norm}
\begin{lemma}
[Properties of $\|\cdot\|_{\rp}$]
\label{lem:embedding}
The following properties hold true:
\begin{enumerate}
\item[(i)]
$\|h\|_{L^1(\Om)}\le C\|h\|_{\rp}$,
for some $C>0$ independent of $h$ and $p$;
\item[(ii)]
$\|\dej^2 h\|_{L^\infty(\Aj)}\le C\|h\|_{\rp}$; (relevant)
\item[(iii)]
$\|\rj h\|_{L^\infty(\Aj)}\le\|h\|_{\rp}$;(obvious)
\item[(iv)]
$|h(x)|\le\sum_{j=1}^k\frac{\dej^\eta}{\dej^{2+\eta}+|x|^{2+\eta}}\chi_{\Aj}(x)\|h\|_{\rp}$;
\item[(v)]
If $\eta\le2s_1$, then $\|\rp\|_{L^\infty(\Om)}\le C$. 
\end{enumerate}
In particular, for any $q>0$ we have the weighted mass estimate
\beq
\label{est:rpmass}
\|\rj(x)|x|^{\aj-2}e^{\Udaj(x)}(|\Vaj(\frac{x}{\dej})|^q+1)\|_{L^\infty(\Aj)}\le C.
\eeq
\end{lemma}
\begin{proof}
Proof of (i).
Recall that
\beqs
\bal
\rp(x)\,dx=&\sum_{j=1}^k\rj(x)\chi_{\Aj}(x)\,dx\\
\rj(x)\,dx=&\frac{\dej^{2+\eta}+|x|^{2+\eta}}{\dej^\eta}\,dx=\dej^2(1+|y|^{2+\eta})\,dy,
\qquad x=\dej y\in\Aj,
\eal
\eeqs
and $\|h\|_{\rp}=\|\rp h\|_{L^\infty(\Om)}$.
We readily check that:
\beqs
\int_{\Aj}|h(x)|\,dx=\int_{\Aj}\frac{\rj(x)|h(x)|}{\rj(x)}\,dx
\le\|h\|_{\rp}\int_{\Aj/\dej}\frac{dy}{1+|y|^{2+\eta}}\le C\|h\|_{\rp}.
\eeqs
\par
Proof of (ii).
For all $x\in\Aj$ we compute:
\beqs
\dej^2\,h(x)\le\frac{\dej^2\|\rj h\|_{L^\infty(\Aj)}}{\rj(x)}=\frac{\|\rj h\|_{L^\infty(\Aj)}}{1+|\frac{x}{\dej}|^{2+\eta}}\le\|h\|_{\rp}.
\eeqs
\par
Proof of (iii).
By definition of $\|\cdot\|_{\rp}$, we have
\beqs
\|h\|_{\rp}=\|\sum_{j=1}^k\rj\chi_{\Aj}h\|_{L^\infty(\Om)}=
\sum_{j=1}^k\|\rj h\|_{L^\infty(\Aj)}\ge\|\rj h\|_{L^\infty(\Aj)}.
\eeqs 
Proof of (iv). For any fixed $j=1,2,\ldots,k$ and $x\in\Aj$,
\beqs
|h(x)|=\frac{|\rj(x)h(x)|}{\rj(x)}=\frac{\dej^\eta}{\dej^{2+\eta}+|x|^{2+\eta}}|\rj(x)h(x)|
\le\frac{\dej^\eta}{\dej^{2+\eta}+|x|^{2+\eta}}\|h\|_{\rp}.
\eeqs
Proof of (v).
Let $j=1,2,\ldots,k$. Recalling that $x=\dej y\in\Aj$ implies that $|y|\le(\dejj/\dej)^{1-\ej}$,
we estimate
\beqs
\bal
\rj(x)=\dej^2(1+|y|^{2+\eta})\le\dej^2(1+(\frac{\dejj}{\dej})^{(1-\ej)(2+\eta)}).
\eal
\eeqs
Thus, a sufficient condition for boundedness of $\rj$ is given by $\eta\le2\ej/(1-\ej)=2\sj$.
In view of Proposition~\ref{prop:solvability}--(iii), by choosing $\eta\le2s_1\le\sj$, we obtain 
the asserted uniform boundedness for $\rp$.
\end{proof}
\begin{rmk}
For $r>1$ the above argument yields:
\beqs
\int_{\Aj}|h(x)|^r\,dx\le\|h\|_{\rp}\int_{\Aj}\frac{\dej^r}{(\dej^2+|x|^2)^{3r/2}}(\frac{\dej}{|x|})^{\tho r}\,dx
=\frac{\|h\|_{\rp}}{\dej^{2(r-1)}}\int_{\Aj/\dej}\frac{dy}{(1+|y|^2)^{3r/2}|y|^{\tho r}},
\eeqs
which does not yield a uniform embedding constant.
\end{rmk}
\begin{lemma}[Estimates for $\gp$]
\label{lem:gpp}
The following elementary inequalities hold true:
\beqs
\tag{i}
|\gp(1+\frac{s}{p})|=|1+\frac{s}{p}|^p\le\left\{
\bal
&e^s,
&&\hbox{if }s\ge-p\\
&e^{-(s+2p)},
&&\hbox{if }s<-p;
\eal
\right.
\eeqs
\beqs
\tag{ii}
|\gp'(1+\frac{s}{p})|=p|1+\frac{s}{p}|^{p-1}\le\left\{
\bal
&pe^{\frac{p-1}{p}s},
&&\hbox{if }s\ge-p\\
&pe^{-\frac{p-1}{p}(s+2p)},
&&\hbox{if }s<-p;
\eal
\right.
\eeqs
\beqs
\tag{iii}
|\gp''(1+\frac{s}{p})|=p(p-1)|1+\frac{s}{p}|^{p-2}\le\left\{
\bal
&p(p-1)e^{\frac{p-2}{p}s},
&&\hbox{if }s\ge-p\\
&p(p-1)e^{-\frac{p-2}{p}(s+2p)},
&&\hbox{if }s<-p.
\eal
\right.
\eeqs
\end{lemma}
\begin{proof}
The proof of (i) follows by concavity of the logarithmic function and reflection properties.
The proof of (ii)--(iii) follow by the identities $\gpp(t)=p|t|^{p-1}=p|\gp(t)|^{\frac{p-1}{p}}$
and $|\gp''(t)|=p(p-1)|t|^{p-2}=p(p-1)|\gp(t)|^{\frac{p-2}{p}}$.
\end{proof}
\begin{lemma}[Taylor expansion]
\label{lem:Taylor}
Let $a(t)$, $b(t)$, $c(t)$, $t>0$, be smooth, real-valued functions 
satisfying 
\beqs
\bal
&\ln\frac{t^\si}{1+t^\tau}-C\le a(t)\le\ln\frac{t^\si}{1+t^\tau}+C
\eal
\eeqs
and
\beqs
|b(t)|+|c(t)|\le C\ln(t+2),
\eeqs 
for some $0\le\si<\tau$ and $C>0$.
Let $\Eap\subset(0,+\infty)$ be defined by
\beqs
\Eap:=\left\{t>0:\ a(t)>-\frac{p}{2}\right\}.
\eeqs
Then, the following expansions hold as $p\to+\infty$, uniformly for $t\in\Eap$:
\[
\tag{i}
\bal
&\left(1+\frac{a(t)}{p}+\frac{b(t)}{p^2}+\frac{c(t)}{p^3}+o(\frac{1}{p^3})\right)^p\\
&\ =e^{a(t)}\left\{1+\frac{b(t)-\vphio(a(t))}{p}
+\frac{c(t)-\vphii(a(t),b(t))}{p^2}+\frac{O(|a(t)|^6+1)}{p^3}+o(\frac{1}{p^2})\right\},
\eal
\]
where $\vphio$, $\vphii$ are defined in \eqref{def:phi}.
\par
Moreover, for any fixed $0<\kappa<p$, there holds
\beqs
\tag{ii}
\bal
&\left(1+\frac{a(t)}{p}+\frac{b(t)}{p^2}+o(\frac{1}{p^2})\right)^{p-\kappa}\\
&\qquad\qquad=e^{a(t)}\left\{1+\frac{1}{p}\Big(b(t)-\vphio(a(t))-\kappa\,a(t)\Big)
+\frac{O(|a(t)|^4+1)}{p^2}+o(\frac{1}{p})\right\},
\eal
\eeqs
uniformly with respect to $t\in\Eap$. 
\end{lemma}
\begin{proof}
We shall repeatedly use the following properties:
\begin{enumerate}
  \item[]
$C^{-1}e^{-p/(2\si)}\le t\le Ce^{\frac{p}{2(\tau-\si)}}$, for some $C>0$ independent of $t\in\Eap$;
\item[]
$-\frac{p}{2}\le a(t)\le C$, in particular $a(t)=O(p)$ and $|b(t)|+|c(t)|=O(|a(t)|+1)=O(p)$,
uniformly for $t\in\Eap$.
\end{enumerate}
Proof of (i).
Let $\xip=\xip(t)$ be defined by
\[
\xip(t)=\frac{a(t)}{p}+\frac{b(t)}{p^2}+\frac{c(t)+o(1)}{p^3}=\frac{a(t)+O(1)}{p}.
\]
Since $a(t)$ is bounded from above, by taking $p$ sufficiently large we may assume that
$|\xip(t)|\le3/4$ in $\Eap$.
Therefore, by Taylor expansion of the logarithmic function up to the third order, we may write
\beqs
\bal
\log(1+\xip)=&\xip-\frac{\xip^2}{2}+\frac{\xip^3}{3}-\frac{\xip^2}{4(1+\theta_p\xip)^4}\\
=&\frac{a}{p}+\frac{1}{p^2}(b-\frac{a^2}{2})+\frac{1}{p^3}(c-ab+\frac{a^3}{3})-\frac{\xip^2}{4(1+\theta_p\xip)^4}+\frac{O(|a|^3+1)}{p^4}+o(\frac{1}{p^3})
\eal
\eeqs
for some $0\le\theta_p(t)\le1$, uniformly with respect to $p\to+\infty$ and $t\in\Eap$.
It follows that
\beqs
p\log(1+\xip)=a+\frac{1}{p}(b-\frac{a^2}{2})+\frac{1}{p^2}(c-ab+\frac{a^3}{3})-\frac{p\,\xip^2}{4(1+\theta_p\xip)^4}+\frac{O(|a|^3+1)}{p^3}+o(\frac{1}{p^2}),
\eeqs
uniformly with respect to $p\to+\infty$ and $t\in\Eap$.
We set
\beqs
\txip:=\frac{1}{p}(b-\frac{a^2}{2})+\frac{1}{p^2}(c-ab+\frac{a^3}{3})-\frac{p\,\xip^2}{4(1+\theta_p\xip)^4}+\frac{O(|a|^3+1)}{p^3}+o(\frac{1}{p^2}),
\eeqs
and we observe that $\txip(t)\le C$, namely $\txip$ is \emph{uniformly bounded from above}. 
Therefore, by Taylor expansion of the exponential function to the second order we may write
\beqs
e^{\txip}=\txip+\frac{\txip^2}{2}+\frac{e^{\widetilde\theta_p\,\txip}}{6}\txip^3,
\eeqs
for some $0\le\widetilde\theta_p\le1$, so that $e^{\widetilde\theta_p\txip}\le C$, 
namely $e^{\widetilde\theta_p\txip}$ is \emph{bounded uniformly} 
with respect to $p\to+\infty$ and $t\in\Eap$.
Observing that 
\beqs
p\,\xip^4=\frac{O(|a|^4+1)}{p^3},
\eeqs
we obtain the expansions
\beqs
\bal
\txip=&\frac{1}{p}(b-\frac{a^2}{2})+\frac{1}{p^2}(c-ab+\frac{a^3}{3})+\frac{O(|a|^4+1)}{p^3}+o(\frac{1}{p^2}),\\
\txip^2=&\frac{1}{p^2}(b-\frac{a^2}{2})^2+\frac{O(|a|^6+1)}{p^4}+o(\frac{1}{p^2})\\
\txip^3=&\frac{O(|a|^6+1)}{p^3},
\eal
\eeqs
from which we finally derive
\beqs
e^{\txip}=\frac{1}{p}(b-\frac{a^2}{2})+\frac{1}{p^2}(c-ab+\frac{a^3}{3}+\frac{1}{2}(b-\frac{a^2}{2})^2)+o(\frac{1}{p^2})
+\frac{O(|a|^6+1)}{p^3},
\eeqs
\emph{uniformly} with respect to $p\to+\infty$ and $t\in\Eap$.
This established the asserted expansion~(i).
%
%For any fixed $t\in\Eap$, the Taylor expansion up to the third order of the logarithmic function yields
%\beqs
%\bal
%\ln\left(1+\frac{a}{p}+\frac{b}{p^2}+\frac{c}{p^3}+o(\frac{1}{p^3})\right)
%=\frac{a}{p}+\frac{1}{p^2}(b-\frac{a^2}{2})+\frac{1}{p^3}(c-ab+\frac{a^3}{3})+O(\frac{|a|^4+1)}{p^4}+o(\frac{1}{p^3}).
%\eal
%\eeqs
%Hence, we may write
%\beqs
%\left(1+\frac{a(t)}{p}+\frac{b(t)}{p^2}+\frac{c(t)}{p^3}+o(\frac{1}{p^3})\right)^p
%=e^a\exp\left\{\frac{1}{p}(b-\frac{a^2}{2})+\frac{1}{p^2}(c-ab+\frac{a^3}{3})+O(\frac{|a|^4+1}{p^3})+o(\frac{1}{p^2})\right\}.
%\eeqs
%In turn, the Taylor expansion of the exponential up to the second order yields
%\beqs
%\bal
%\exp&\left\{\frac{1}{p}(b-\frac{a^2}{2})+\frac{1}{p^2}(c-ab+\frac{a^3}{3})+O(\frac{|a|^4+1}{p^3})+o(\frac{1}{p^2})\right\}\\
%&=1+\frac{1}{p}(b-\frac{a^2}{2})+\frac{1}{p^2}\left[c-ab+\frac{a^3}{3}+\frac{1}{2}(b-\frac{a^2}{2})^2\right]
%+O(\frac{|a|^6+1}{p^3})+o(\frac{1}{p^2})\\
%&=1+\frac{1}{p}(b-\vphio(a))+\frac{1}{p^2}\left[c-\vphii(a,b)\right]
%+O(\frac{|a|^6+1}{p^3})+o(\frac{1}{p^2}).
%\eal
%\eeqs
%Thus, the expansion~(i) is established.
\par
Proof of (ii). 
Similarly as above, let
\beqs
\zep:=\frac{a(t)}{p}+\frac{b(t)+o(1)}{p^2}=\frac{a(t)+O(1)}{p}.
\eeqs
By taking $p$ sufficiently large, we may assume that $|\zep(t)|\le3/4$
for all $t\in\Eap$.
Expansion of the logarithmic function to the second oder yields
\beqs
\bal
\log(1+\zep)=&\zep-\frac{\zep^2}{2}+\frac{\zep^3}{3(1+\theta_p\zep)^3}\\
=&\frac{a}{p}+\frac{1}{p^2}(b-\frac{a^2}{2})+\frac{O(|a|^2+1)}{p^3}+\frac{\zep^3}{3(1+\theta_p\zep)^3}+o(\frac{1}{p^2}),
\eal
\eeqs
uniformly with respect to $p\to+\infty$ and $t\in\Eap$.
We deduce that
\beqs
(p-\kappa)\log(1+\zep)=a+\frac{1}{p}(b-\frac{a^2}{2}-\kappa a)+\frac{O(|a|^2+1)}{p^2}+\frac{(p-\kappa)\zep^3}{3(1+\theta_p\zep)^3}
+o(\frac{1}{p}).
\eeqs
We set
\beqs
\tzep(t):=\frac{1}{p}(b-\frac{a^2}{2}-\kappa a)+\frac{O(|a|^2+1)}{p^2}+\frac{(p-\kappa)\zep^3}{3(1+\theta_p\zep)^3}
+o(\frac{1}{p})
\eeqs
and we observe that $\tzep(t)$ is \emph{uniformly bounded from above} with respect
to $p\to+\infty$ and $t\in\Eap$.
Therefore, expansion of the exponential function to the first order yields
\beqs
e^{\tzep(t)}=1+\tzep(t)+\frac{e^{\widetilde\theta_p(t)\tzep(t)}}{2}\tzep^2(t),
\eeqs
for some $0\le\widetilde\theta_p(t)\le1$, so that $e^{\widetilde\theta_p(t)\tzep(t)}\le C$
is \emph{uniformly bounded} with respect
to $p\to+\infty$ and $t\in\Eap$.
We deduce that
\beqs
e^{\tzep(t)}=1+\frac{1}{p}(b-\frac{a^2}{2}-\kappa a)+\frac{O(|a|^4+1)}{p^2}+o(\frac{1}{p}),
\eeqs
and the asserted expansion~(ii) follows.
\end{proof}
%\begin{lemma}[Properties of $\Uda$]
%\label{lem:Uda}
%We have
%\[
%\Uda(\de y)=U_1^\al(y)+\ln\frac{1}{\de^\al}
%=v_\al(y)+\ln\frac{1}{\de^\al}.
%\]
%Moreover, the following expansion holds:
%\[
%\bal
%(P\Uda)(x)=&\Uda(x)-\ln(2\al^2\de^\al)+4\pi\al H(x,0)+O(\de)\\
%=&-2\ln(\de^\al+|x|^\al)+4\pi\al H(x,0)+O(\de)
%\eal
%\]
%uniformly in $C^1(\overline\Om)$.
%\end{lemma}
Let
\beq
\label{def:phioa}
\phoa(r)=\frac{1-r^\al}{1+r^\al}.
\eeq
By adapting the arguments in \cite{emp1} we have the following result.
\begin{lemma}[Chae-Imanuvilov lemma]
\label{lem:CI}
Fix $\al\ge2$. 
Let $F\in C^1([0,+\infty))$ be such that
\beq
\label{assumpt:F}
|F(t)|\le C\frac{(|\ln t|+1)^4}{t^{\al+2}},
\qquad\hbox{as\ }t\to+\infty.
\eeq
Then, there exists a $C^2$ radial solution $w_F(y)$ to the equation
\beq
\label{eq:CI}
\Delta w+\frac{2\al^2|y|^{\al-2}}{(1+|y|^\al)^2}w=F(|y|)
\qquad in\rr^2
\eeq
satisfying
\beqs
w_F(y)=C_F\ln|y|+O(\frac{1}{|y|})
\qquad\hbox{as }|y|\to+\infty
\eeqs
where
\beqs
C_F=\int_0^{+\infty}t\phoa(t)F(t)\,dt.
\eeqs
\end{lemma}
\begin{proof}
In view of Lemma~2.1 in \cite{ChaeImanuvilov} it is known that there esists a $C^2$ radial solution $w(r)$
to \eqref{eq:CI} of the form
\beqs
w(r)=\phoa(r)\left\{\int_0^r\frac{\phi_F(s)-\phi_F(1)}{(1-s)^2}\,ds+\phi_F(1)\frac{r}{1-r}\right\}
\eeqs
with
\beq
\label{def:phiF}
\phi_F(s)=\frac{(1-s)^2}{s\phoa(s)^2}\int_0^st\phoa(t)F(t)\,dt,
\eeq
satisfying
\beqs
|w(r)|\le C(\ln^+r+1)\qquad\hbox{as\ }r\to+\infty,
\eeqs
where $\phiF(1)$ and $w(1)$ are defined as limits of $\phiF(r)$ and $w(r)$ as $r\to1$.
In order to derive the exact logarithmic growth factor, we write for $r\ge2$:
\beqs
\bal
\int_0^r\frac{\phiF(s)-\phiF(1)}{(s-1)^2}\,ds
=&\int_0^2\frac{\phiF(s)-\phiF(1)}{(s-1)^2}\,ds+\int_2^r\frac{\phiF(s)}{(s-1)^2}\,ds
-\phiF(1)\int_2^r\frac{ds}{(s-1)^2}\\
=&\int_2^r\frac{\phiF(s)}{(s-1)^2}\,ds+\int_0^2\frac{\phiF(s)-\phiF(1)}{(s-1)^2}\,ds
-\phiF(1)(1-\frac{1}{r-1})\\
=&\int_2^r\frac{\phiF(s)}{(s-1)^2}\,ds+D_{F,1}+O(\frac{1}{r}),
\eal
\eeqs
where
\beq
\label{def:DF1}
D_{F,1}=\int_0^2\frac{\phiF(s)-\phiF(1)}{(s-1)^2}\,ds-\phiF(1).
\eeq
In turn, we write, recalling \eqref{def:phiF} and the fact $\phoa^{-1}(s)=-1+2/(1-s^\al)$,
\beq
\label{phiFexp}
\bal
\int_2^r\frac{\phiF(s)}{(s-1)^2}\,ds
=&\int_2^r\frac{ds}{s\phoa^2(s)}\int_0^st\phoa(t)F(t)\,dt
=\int_2^r\frac{ds}{s}(-1+\frac{2}{1-s^\al})^2\int_0^st\phoa(t)F(t)\,dt\\
=&\int_2^r\frac{ds}{s}\int_0^st\phoa(t)F(t)\,dt
-4\int_2^r\frac{ds}{s(1-s^\al)}\int_0^st\phoa(t)F(t)\,dt\\
&\qquad\qquad+4\int_2^r\frac{ds}{s(1-s^\al)^2}\int_0^st\phoa(t)F(t)\,dt.
\eal
\eeq
Integration by parts yields
\beqs
\bal
\int_2^r\frac{ds}{s}&\int_0^st\phoa(t)F(t)\,dt
=\ln r\int_0^rt\phoa(t)F(t)\,dt+D_{F,2}
+\int_r^{+\infty}s\ln s\phoa(s)F(s)\,ds\\
=&C_F\ln r+D_{F,2}-\ln r\int_r^{+\infty}t\phoa(t)F(t)\,dt+\int_r^{+\infty}s\ln s\phoa(s)F(s)\,ds
\eal
\eeqs
where 
\beqs
\bal
D_{F,2}=&-\ln2\int_0^2t\phoa(t)F(t)\,dt-\int_2^{+\infty}s\ln s\phoa(s)F(s)\,ds\\
C_F=&\int_0^{+\infty}t\phoa(t)F(t)\,dt.
\eal
\eeqs
%We may also write
%\beqs
%\int_0^rt\phoa(t)F(t)\,dt=C_F-\int_r^{+\infty}t\phoa(t)F(t)\,dt
%\eeqs
%with
%\beqs
%C_F=\int_0^{+\infty}t\phoa(t)F(t)\,dt.
%\eeqs
It is straightforward to check that in view of \eqref{assumpt:F} there holds:
\beqs
\bal
&|\int_r^{+\infty}t\phoa(t)F(t)\,dt|\le C\int_r^{+\infty}t\frac{(|\ln t|+1)^4}{t^{\al+2}}\,dt
=O(\frac{1}{r^{\al-1/2}})\\
&|\int_r^{+\infty}s\ln s\phoa(s)F(s)\,ds|\le C\int_r^{+\infty}s|\ln s|\frac{(|\ln s|+1)^4}{s^{\al+2}}\,ds
=O(\frac{1}{r^{\al-1}})
\eal
\eeqs
so that the first term on the right hand side of \eqref{phiFexp} takes the form
\beqs
\int_2^r\frac{ds}{s}\int_0^st\phoa(t)F(t)\,dt
=C_F\ln r+D_{F,2}+O(\frac{1}{r^{\al-1}}).
\eeqs
Similarly, we write
\beqs
\bal
\int_2^r\frac{ds}{s(1-s^\al)}\int_0^st\phoa(t)F(t)\,dt
=&D_{F,3}
-\int_r^{+\infty}\frac{ds}{s(1-s^\al)}\int_0^st\phoa(t)F(t)\,dt\\
=&D_{F,3}+O(\frac{1}{r^\al})
\eal
\eeqs
where
\beq
D_{F,3}=\int_2^{+\infty}\frac{ds}{s(1-s^\al)}\int_0^st\phoa(t)F(t)\,dt.
\eeq
%and we estimate, using \eqref{assumptF},
%\beqs
%|\int_r^{+\infty}\frac{ds}{s^{\al+1}}\int_0^st\phoa(t)F(t)\,dt|
%\le C\int_r^{+\infty}\frac{ds}{s^{\al+1}}=O(\frac{1}{r^\al})
%\eeqs
Finally,
\beqs
\bal
\int_2^r\frac{ds}{s(1-s^\al)^2}\int_0^st\phoa(t)F(t)\,dt
=&D_{F,4}
-\int_r^{+\infty}\frac{ds}{s(1-s^\al)^2}\int_0^st\phoa(t)F(t)\,dt\\
=&D_{F,4}+O(\frac{1}{r^{2\al}})
\eal
\eeqs
where
\beq
D_{F,4}=\int_2^{+\infty}\frac{ds}{s(1-s^\al)^2}\int_0^st\phoa(t)F(t)\,dt.
\eeq
Observing that e may write $\phoa(r)=-1+O(\frac{1}{r^\al})$, we deduce from the above that
\beqs
w(r)=C_f\ln r+D_F\phoa(r)+O(\frac{1}{r}),
\eeqs
where $D_F=D_{F,1}+D_{F,2}-4D_{F,3}+4D_{F_4}$.
The desired solution is given by 
$w_F(r)=w(r)-D_F\phoa(r)$.
\end{proof}

\bibliography{biblio}
\bibliographystyle{abbrv}
\end{document}